\documentclass[11pt]{article}

\usepackage{amsmath,amsthm,amsfonts,amssymb}
\usepackage[margin=2.2cm]{geometry}
\usepackage{thmtools}
\usepackage{hyperref}
\usepackage{enumitem}
\usepackage{authblk}
\usepackage{orcidlink}

\newtheorem{thm}{Theorem}[section]
\newtheorem{lem}{Lemma}[thm]
\newtheorem{prop}{Proposition}[section]

\theoremstyle{definition}
\newtheorem{defin}{Definition}[thm]
\newtheorem{obs}{Observation}

\newtheorem{remark}{Remark}[thm]
\usepackage{bbold}
\usepackage{tikz}
\usetikzlibrary{positioning}
\usepackage{xfrac}
\usepackage{quiver}
\usepackage{graphicx}
\usepackage{enumitem}
\newcommand{\Crit}{\operatorname{Crit}}
\newcommand{\Ker}{\operatorname{Ker}}
\newcommand{\GS}{\operatorname{GS}}

\newcommand{\MV}{\operatorname{MV}}

\title{An effective Mayer-Vietoris Theorem for discrete Morse homology} 
\author[a]{Sajal Mukherjee~\orcidlink{0009-0004-6959-4334}\thanks{\texttt{sajal.mukherjee@tcgcrest.org}}}

\author[b]{Pritam Chandra Pramanik~\orcidlink{0009-0000-8332-3347}\thanks{\texttt{pritam.pramanik.80@tcgcrest.org}}}

\author[c]{Arundhati Rakshit~\orcidlink{0009-0002-8209-8584}\thanks{\texttt{arundhati.rakshit.124@tcgcrest.org}}}

\affil[a,b,c]{\small Institute for Advancing Intelligence (IAI), TCG CREST, Kolkata--700091, West Bengal, India}
\affil[a,c]{\small Mathematical \& Information Science, Academy of Scientific and Innovative Research (AcSIR), Ghaziabad--201002, Uttar Pradesh, India}

\date{}

\begin{document}
	\maketitle	
	
	\begin{abstract}
		 The Mayer-Vietoris theorem is known for its wide applications, especially in determining homology. In fact, this theorem provides us with a long exact sequence, where the underlying homology groups  fit in. However, this theorem does not provide an explicit way to compute homology. In this paper we prove an ``effective" version of the Mayer-Vietoris theorem using discrete Morse theory. Suppose, we have a Mayer-Vietoris type setup, i.e., let $X$ be a simplicial complex and $A$ and $B$ be two subcomplexes of $X$, such that $A \cup B=X$. Moreover, let $\mathcal{W}_A$, $\mathcal{W}_{B}$ and $\mathcal{W}_{A \cap B}$ be gradient vector fields on $A$, $B$ and $A \cap B$ respectively (which need not be ``coherent", i.e., they do not need to coincide on their intersection). Then, the main theorem of our paper provides an explicit way to compute the homology groups of $X$, using the combinatorial information regarding the trajectories of the aforementioned gradient vector fields, we do not even need to know the individual homology groups $H_{*}(A)$, $H_{*}(B)$ and $H_{*}(A \cap B)$. In principle, the homology of $X$ can always be computed explicitly using our theorem irrespective of the choice of the gradient vector fields. Further, if we choose the subcomplexes $A$ and $B$ wisely so that each of $A$, $B$ and $A \cap B$ admits an efficient gradient vector field, then the computation of the homology groups is considerably reduced.
	\end{abstract}
	
	\textbf{Keywords:}  simplicial complex, discrete Morse theory, chain complex, effective homology computation, Mayer-Vietoris sequence.
	
	\textit{MSC 2020:} 57Q70 (primary), 05E45, 55U15.
	
	\begin{section}{Introduction}\label{intro}
	
	 The Mayer-Vietoris theorem \cite{Munkres} is a powerful tool in algebraic topology, primarily used to determine the homology of simplicial complexes (in general, topological spaces. However, throughout the paper, we work in simplicial setting, as the main focus of our paper is explicit computation and simplicial complexes are more amenable to coding). Particularly, its utility lies in calculating the homology of a large, complicated simplcial complex by ``breaking" it into smaller and easier to handle subcomplexes. To be more precise, if $X$ is a simplicial complex and $A$ and $B$ are two subcomplexes of $X$ such that $A \cup B=X$. Then the Mayer-Vietoris theorem (in simplicial setting) states that the homology groups $H_{*}(X)$, $H_{*}(A)$, $H_{*}(B)$ fit in the following long exact sequence (w.r.t. some appropriate maps $f_{*}$, $g_{*}$ and $h_{*}$).
	
	$$ \dots \rightarrow H_n(A \cap B)  \xrightarrow{f_n} H_n(A) \oplus H_n(B) \xrightarrow{g_n} H_n(X) \xrightarrow{h_n} H_{n-1}(A \cap B) \rightarrow \dots H_0(X) \rightarrow 0.$$
	
	However, despite its multifarious uses in algebraic topology, a notable limitation of the Mayer-Vietoris theorem lies in its computational aspect. The Mayer-Vietoris theorem simply provides us with the above long exact sequence where the homology groups $H_{*}(X)$, $H_{*}(A)$, $H_{*}(B)$ fit in, from where we can compute the homology of $X$ in some limited cases. In general, it does not provide an explicit way to compute the homology groups of $X$. In this paper, we provide an ``effective" version of the Mayer-Vietoris theorem using discrete Morse theory. To state our result, we need the following notions from discrete Morse theory.\\
	
	 Let $X$ be a simplicial complex. Let us denote a simplex $\sigma$ of dimension $q$ as $\sigma^{(q)}$.
	
	A \emph{discrete vector field} $\mathcal{V}$ on $X$ is defined as a collection of pairs of simplices $\{(\sigma^{(q-1)},\tau^{(q)})\mid \sigma, \tau \in X, \sigma \subseteq \tau\}$ such that every simplex in $X$ appears in atmost one pair.\\

	\begin{defin}{(Forman trajectory)}
		A \emph{Forman $\mathcal{V}$-trajectory} (or simply a $\mathcal{V}$- trajectory) is defined as sequence of simplices of the following form.
		$$\tau_0^{(q)}, \sigma_1^{(q-1)}, \tau_1^{(q)}, \dots \sigma_k^{(q-1)}, \tau_k^{(q)},$$
		
		 where,  $(\sigma_{i}^{(q-1)}, \tau_{i}^{(q)}) \in \mathcal{V}$, $ \sigma_{i}^{(q-1)} \subseteq \tau_{i-1}^{(q)}$, $(\sigma_i^{(q-1)}, \tau_{i-1}^{(q)}) \notin \mathcal{V}$ for all $i\in [k]$.
		
		 We will refer to a $\mathcal{V}$-trajectory as simply a trajectory whenever the gradient vector field is clear from the context. A trajectory is said to be \emph{closed} if $\tau_0 = \tau_k$ for $k > 1$. 
		
		 Now we define an \emph{extended Forman $\mathcal{V}$-trajectory} (or, \emph{extended $\mathcal{V}$-trajectory} )  as a sequence of simplices of the following form.

		$$\tau_0^{(q)}, \sigma_1^{(q-1)}, \tau_1^{(q)}, \dots \sigma_k^{(q-1)}, \tau_k^{(q)}, \sigma_{k+1}^{(q-1)},$$
		 where  $(\sigma_{i}^{(q-1)}, \tau_{i}^{(q)}) \in \mathcal{V}$ for all $i \in [k]$,  $ \sigma_{i}^{(q-1)} \subseteq \tau_{i-1}^{(q)}$ and $(\sigma_i^{(q-1)}, \tau_{i-1}^{(q)}) \notin \mathcal{V}$ for all $i\in [k+1]$.

		 Here, $\tau_{0}$ is said to be the \emph{initial simplex} and $\sigma_{k+1}$ is said to be the \emph{terminal simplex} of the extended trajectory. An (extended) trajectory is said to be \emph{non-trivial} if $k > 0$. 
		 
		 From here onwards, we refer to an extended trajectory as simply trajectory, for the sake of simplicity, and note that this will not cause any terminological ambiguity.
		 
	\end{defin}

	A discrete vector field $\mathcal{V}$ is said to be \emph{acyclic} if there exists no non-trivial closed $\mathcal{V}$-trajectories. A discrete vector field $\mathcal{V}$ is said to be a \emph{gradient vector field} if $\mathcal{V}$ is acyclic.
	
	Given a gradient vector field $\mathcal{V}$ on a simplicial complex, a simplex is said to be \emph{$\mathcal{V}$-critical} if it does not appear in $\mathcal{V}$. We will refer to them as simply critical simplices when the gradient vector field is clear from the context. We denote the set of all $q$-dimensional $\mathcal{V}$-critical simplices in $X$ as $\Crit_q^{\mathcal{V}}(X)$.


	Let $X$ be a simplicial complex with a gradient vector field $\mathcal{V}$.
	
	\begin{defin}
		
		 Let $\gamma: \tau_{0}^{(q)}, \sigma_{1}^{(q-1)}, \tau_{1}^{(q)}, \dots, \sigma_{k}^{(q-1)}, \tau_{k}^{(q)},\sigma_{k+1}^{(q-1)}$ be a $\mathcal{V}$-trajectory in X. Then the weight of $\gamma$, denoted as $w(\gamma)$ is defined as,
		$$ w(\gamma):= \left(\prod_{i=0}^{k-1} -\langle  \tau_{i}, \sigma_{i+1},\rangle \langle \tau_{i+1}, \sigma_{i+1} \rangle \right)\langle \tau_k, \sigma_{k+1}\rangle.$$
		
	\end{defin}
	 The set of all $\mathcal{V}$-trajectories with $\tau^{(q)}$ as the initial simplex and $\sigma^{(q-1)}$ as the terminal simplex is denoted as $\Gamma(\tau, \sigma)$.

	For a simplicial complex $X$ and a gradient vector field $\mathcal{V}$ on it, we define the Thom-Smale chain group $C^{\mathcal{V}}_{q}(X, \mathbb{Z})$ as the abelian group generated by $\Crit_{q}^{\mathcal{V}}(X)$ over $\mathbb{Z}$ for each $q \geq 0$, i.e.,
	
	$$ C^{\mathcal{V}}_{q}(X, \mathbb{Z})= \mathbb{Z}\langle \Crit^{V}_{q}(X) \rangle  \text{, for each } q \geq 0.$$ 
	
	From here onwards, we work using the coefficent ring as $\mathbb{Z}$. So,  whenever we refer to chain groups and homology groups $K$, instead of $K_{\#}(X, \mathbb{Z})$, we simply write $K_{\#}(X)$.

	Now, we define the Thom-Smale boundary operator $\partial_{q+1}^\mathcal{V}: C^{\mathcal{V}}_{q+1}(X) \rightarrow C^{\mathcal{V}}_{q}(X)$ as,
	$$ \partial_{q+1}^\mathcal{V}(\tau) = \sum_{\sigma \in \Crit_{q}^{\mathcal{V}}(X)}\left(\sum_{P \in \Gamma(\tau,\sigma)}w(P)\right)\sigma \text{, for each } \tau \in \Crit_{q+1}^{\mathcal{V}}(X) \text{, } q  \geq 0.$$
	 and extend it linearly to the whole of $C^{\mathcal{V}}_{q}(X)$. Thus the Thom-Smale chain complex $(C_{\#}^{\mathcal{V}}(X), \partial_{\#}^{\mathcal{V}})$ runs as,
	$$ \dots \rightarrow C^{\mathcal{V}}_{q+1}(X) \xrightarrow{\partial^{\mathcal{V}}_{q+1}}  C^{\mathcal{V}}_{q}(X) \xrightarrow{\partial_q^{\mathcal{V}}} C^{\mathcal{V}}_{q-1}(X) \rightarrow \dots \rightarrow C^{\mathcal{V}}_0(X) \rightarrow 0. $$
	
	 The Thom-Smale homology group, $H^{\mathcal{V}}_{q}(X):= \frac{\Ker(\partial_{q}^{\mathcal{V}})}{Im (\partial^{\mathcal{V}}_{q+1})}$, for each $q \geq 0$.

    The following theorem of Forman \cite{Forman1, Forman2} is of utmost importance. 
	
	\begin{thm}[\cite{Forman1, Forman2}]\label{hom}
	The Thom-Smale chain complex of $X$ (with respect to a gradient vector field $\mathcal{V}$ on $X$), $(C^\mathcal{V}_{\#}(X),\partial^\mathcal{V}_{\#})$  is homotopy equivalent to the simplicial chain complex of $X$, $(C_{\#}(X), \partial_{\#})$. Hence $H^{\mathcal{V}}_{\#}(X) \cong H_{\#}(X)$.
    \end{thm}
	
	We note here that the power of discrete Morse theory rests on its ability to simplify the computation of simplicial homology by reducing the number of generators of the chain groups (see \cite{Matching}). The computation now relies on the information of only the critical simplices instead of all the simplices. 
	However the feasibility of the computation of Morse homology depends crucially on finding an efficient gradient vector field (i.e., one with lower number of critical simplices), on the underlying simplicial complex $X$, which is very often a challenging task (see \cite{Matching}). In fact, it is an $NP$-hard problem to determine an optimal gradient vector field on a simplicial complex (see \cite{NP1, NP2, NP3}).
	
	In this paper, we formulate a Forman type theorem,  (similar to \autoref{hom}) in a Mayer-Vietoris type setup, as follows.

	Let $X$ be a simplicial complex with two subcomplexes $A$ and $B$ such that $A \cup B=X$. Suppose, $\mathcal{W}_A$, $\mathcal{W}_{B}$, $\mathcal{W}_{A \cap B}$ be gradient vector fields on $A$, $B$ and $A \cap B$ respectively. The gradient vector fields on these subcomplexes do not have to be ``coherent", i.e., they do not need to coincide on the intersection nor do they need to be acyclic as a whole, on the entire simplicial complex $X$. Then, using the gradient vector fields $\mathcal{W}_A$, $\mathcal{W}_B$, and $\mathcal{W}_{A \cap B}$, we can determine the homology groups of $X$. The relaxation on the requirement that the gradient vector fields $\mathcal{W}_{A}$, $\mathcal{W}_{B}$ and $\mathcal{W}_{A \cap B}$ need to be coherent provides more flexibility to our theorem.
	

	 Suppose, $X$ is a simplicial complex and $Y$ is a subcomplex of $X$. We define a simplicial complex $\bar{Y}$ as,
	
	$$ V(\bar{Y}):= \{\bar{y}_1, \dots ,\bar{y}_n\}, \text{ where } V(Y)= \{y_1, \dots ,y_n\},$$
	$$ \text{and } \mathcal{F}(\bar{Y}):= \{\{\bar{y}_{i_1}, \dots ,\bar{y}_{i_k}\} \mid \{y_{i_1}, \dots , y_{i_k}\} \in \mathcal{F}(Y)\}.$$
	
	 We call $\bar{Y}$ a \emph{copy} of $Y$.
	
	 For each $\sigma= \{y_{i_0}, \dots , y_{i_k}\} \in Y$, we denote $\{\bar{y}_{i_0}, \dots , \bar{y}_{i_k}\} \in \mathcal{F}(\bar{Y})$ as $\sigma_{\bar{Y}}$.\\
	 Now, we consider a simplicial complex $X$, with subcomplexes $A$ and $B$ such that $A \cup B=X$. Let $\bar{A}, \bar{B}$ and $\overline{A \cap B}$ be \emph{disjoint copies} (i.e., $V(\bar{A})$, $V(\bar{B})$ and $V(\overline{A \cap B})$ are pairwise disjoint) of $A, B$ and $A \cap B$ respectively. Let $\mathcal{W}_{\bar{A}}$, $\mathcal{W}_{\bar{B}}$ and $\mathcal{W}_{\overline{A \cap B}}$ be gradient vector fields on $\bar{A}$, $\bar{B}$, $\overline{A \cap B}$ respectively. Let,
		
		$$ D_0(X):= \Crit_0^{\mathcal{W}_{\bar{A}}}(\bar{A}) \cup \Crit_0^{\mathcal{W}_{\bar{B}}}(\bar{B}), \text{ and}$$
		 $$ D_q(X):= \Crit_q^{\mathcal{W}_{\bar{A}}}(\bar{A}) \cup \Crit_q^{\mathcal{W}_{\bar{B}}}(\bar{B}) \cup \Crit_{q-1}^{\mathcal{W}_{\overline{A \cap B}}}(\overline{A \cap B}), q \geq 1 .$$
	 Next, we define our main object of interest, i.e., a \emph{Mayer-Vietoris trajectory} from a simplex $ \beta \in D_q(X)$ to a simplex $\alpha \in D_{q-1}(X)$.
	\begin{defin} (Mayer-Vietoris trajectory)
		Let $\beta \in D_q(X)$, $\alpha \in D_{q-1}(X)$. Then a \emph{Mayer-Vietoris trajectory} (in short, an MV trajectory) $P$ from $\beta$ to $\alpha$ is defined for the following five cases.
		
	\begin{enumerate}
		\item Let $\beta \in \Crit_q^{\mathcal{W}_{\bar{A}}}(\bar{A})$, and $ \alpha \in \Crit_{q-1}^{\mathcal{W}_{\bar{A}}}(\bar{A}) $. An \emph{MV trajectory} $P$ from $\beta$ to $\alpha$ is defined as a sequence of simplices. 
			$$(\beta=)~(\tau_0)_{\bar{A}}^{(q)}, (\sigma_1)_{\bar{A}}^{(q-1)}, (\tau_1)_{\bar{A}}^{(q)}, \dots ,(\sigma_k)_{\bar{A}}^{(q-1)}, (\tau_k)_{\bar{A}}^{(q)}, (\sigma_{k+1})_{\bar{A}}^{(q-1)}~(=\alpha),$$ 
			
			where, $((\sigma_i)_{\bar{A}}, (\tau_i)_{\bar{A}}) \in \mathcal{W}_{\bar{A}}$, for each $i \in [k]$, $(\sigma_{i})_{\bar{A}} \subseteq (\tau_{i-1})_{\bar{A}}$ and $((\sigma_i)_{\bar{A}}, (\tau_{i-1})_{\bar{A}}) \notin \mathcal{W}_{\bar{A}}$, for each $i \in [k+1]$.

			 The weight of such a trajectory is given by,
			
			$$w_M(P):= \left(\prod_{i=0}^{k-1}-\langle \tau_i, \sigma_{i+1}\rangle \langle \tau_{i+1},\sigma_{i+1} \rangle \right) \langle \tau_k, \sigma_{k+1}\rangle.$$
			
		\item  Let $\beta \in \Crit_q^{\mathcal{W}_{\bar{B}}}(\bar{B})$, and $\alpha \in \Crit_{q-1}^{\mathcal{W}_{\bar{B}}}(\bar{B}) $. An \emph{MV trajectory} $P$ from $\beta$ to $\alpha$ is defined as a sequence of simplices. 
		$$(\beta=)~(\tau_0)_{\bar{B}}^{(q)}, (\sigma_1)_{\bar{B}}^{(q-1)}, (\tau_1)_{\bar{B}}^{(q)}, \dots ,(\sigma_k)_{\bar{B}}^{(q-1)}, (\tau_k)_{\bar{B}}^{(q)}, (\sigma_{k+1})_{\bar{B}}^{(q-1)}~(=\alpha),$$ 
		
		 where, $((\sigma_i)_{\bar{B}}, (\tau_i)_{\bar{B}}) \in \mathcal{W}_{\bar{B}}$, for each $i \in [k]$, $(\sigma_{i})_{\bar{B}} \subseteq (\tau_{i-1})_{\bar{B}}$ and  $((\sigma_i)_{\bar{B}}, (\tau_{i-1})_{\bar{B}}) \notin \mathcal{W}_{\bar{B}}$ for each $i \in [k+1]$.\\

		 The weight of such a trajectory is defined as,
		
		$$w_M(P):= \left(\prod_{i=0}^{k-1}-\langle \tau_i, \sigma_{i+1}\rangle \langle \tau_{i+1},\sigma_{i+1} \rangle \right) \langle \tau_k, \sigma_{k+1}\rangle.$$
		
			\item  Let $\beta \in \Crit_{q-1}^{\mathcal{W}_{\overline{A \cap B}}}(\overline{A \cap B})$, and $ \alpha \in \Crit_{q-2}^{\mathcal{W}_{\overline{A \cap B}}}(\overline{A \cap B}) $. An \emph{MV trajectory} $P$ from $\beta$ to $\alpha$ is defined as a sequence of simplices, 
			$$(\beta=)~(\tau_0)_{\overline{A \cap B}}^{(q-1)}, (\sigma_1)_{\overline{A \cap B}}^{(q-2)}, (\tau_1)_{\overline{A \cap B}}^{(q-1)}, \dots ,(\sigma_k)_{\overline{A \cap B}}^{(q-2)}, (\tau_k)_{\overline{A \cap B}}^{(q-1)}, (\sigma_{k+1})_{\overline{A \cap B}}^{(q-2)}~(= \alpha),$$ 
			
			 where, $((\sigma_i)_{\overline{A \cap B}}, (\tau_i)_{\overline{A \cap B}}) \in \mathcal{W}_{\overline{A \cap B}}$, for each $i \in [k]$, $(\sigma_{i})_{\overline{A \cap B}} \subseteq (\tau_{i-1})_{\overline{A \cap B}}$ and \\ $((\sigma_i)_{\overline{A \cap B}}, (\tau_{i-1})_{\overline{A \cap B}}) \notin \mathcal{W}_{\overline{A \cap B}}$ for each $i \in [k+1]$.

			 The weight of such a trajectory is defined as,
			$$w_M(P):=-\left(\prod_{i=0}^{k-1}-\langle \tau_i, \sigma_{i+1}\rangle \langle \tau_{i+1},\sigma_{i+1} \rangle\right)\langle \tau_k, \sigma_{k+1}\rangle.$$
			
			\item Let $\beta \in \Crit_{q-1}^{\mathcal{W}_{\overline{A \cap B}}}(\overline{A \cap B})$, and $ \alpha \in \Crit_{q-1}^{\mathcal{W}_{\bar{A}}}(\bar{A}) $. An \emph{MV trajectory} $P$ from $\beta$ to $\alpha$ is defined as a sequence of simplices,
			\begin{equation*}
				\begin{aligned}
					&(\beta=)~(\tau_0)_{\overline{A \cap B}}^{(q-1)}, (\sigma_1)_{\overline{A \cap B}}^{(q-2)}, (\tau_1)_{\overline{A \cap B}}^{(q-1)}, \dots , (\sigma_p)_{\overline{A \cap B}}^{(q-2)}, (\tau_p)_{\overline{A \cap B}}^{(q-1)}, (\tau_p)_{\bar{A}}^{(q-1)}, (\alpha_{p})_{\bar{A}}^{(q)}, \dots ,(\tau_{p+l-1})_{\bar{A}}^{(q-1)},\\ & (\alpha_{p+l-1})_{\bar{A}}^{(q)}, (\tau_{p+l})_{\bar{A}}^{(q-1)}~(= \alpha),
				\end{aligned}
			\end{equation*}
			where, $((\sigma_i)_{\overline{A \cap B}}, (\tau_i)_{\overline{A \cap B}}) \in \mathcal{W}_{\overline{A \cap B}}$, for each $i \in [p]$, $(\sigma_{i})_{\overline{A \cap B}} \subseteq (\tau_{i-1})_{\overline{A \cap B}}$ and $((\sigma_i)_{\overline{A \cap B}}, (\tau_{i-1})_{\overline{A \cap B}}) \notin \mathcal{W}_{\overline{A \cap B}}$ for each $i \in [p]$, $((\tau_i)_{\bar{A}},(\alpha_i)_{\bar{A}}) \in \mathcal{W}_{\bar{A}}$, for $i \in \{p, \dots , (p+l-1)\}$, $(\tau_{i})_{\bar{A}} \subseteq (\alpha_{i-1})_{\bar{A}}$ and $((\tau_{i})_{\bar{A}}, (\alpha_{i-1})_{\bar{A}}) \notin \mathcal{W}_{\bar{A}}$ for each $i \in \{p+1, \dots ,p+l\}$, $p,l \geq 0$.

			 The weight of such a trajectory is given by,
			
			$$w_M(P):= - \left(\prod_{i=0}^{p-1}-\langle \tau_i, \sigma_{i+1} \rangle \langle \tau_{i+1}, \sigma_{i+1}\rangle \right) \left(\prod_{i=p}^{p+l-1}-\langle \alpha_i, \tau_i\rangle \langle \alpha_i, \tau_{i+1} \rangle\right).$$
			
			\item   Let $\beta \in \Crit_{q-1}^{\mathcal{W}_{\overline{A \cap B}}}(\overline{A \cap B})$, and $ \alpha \in \Crit_{q-1}^{\mathcal{W}_{\bar{B}}}(\bar{B}).$ An \emph{MV trajectory} $P$ from $\beta$ to $\alpha$ is defined as a sequence of simplices,
				\begin{equation*}
				\begin{aligned}
					&(\beta=)~(\tau_0)_{\overline{A \cap B}}^{(q-1)}, (\sigma_1)_{\overline{A \cap B}}^{(q-2)}, (\tau_1)_{\overline{A \cap B}}^{(q-1)}, \dots , (\sigma_p)_{\overline{A \cap B}}^{(q-2)}, (\tau_p)_{\overline{A \cap B}}^{(q-1)}, (\tau_p)_{\bar{B}}^{(q-1)}, (\alpha_{p})_{\bar{B}}^{(q)}, \dots ,(\tau_{p+l-1})_{\bar{B}}^{(q-1)},\\ & (\alpha_{p+l-1})_{\bar{B}}^{(q)}, (\tau_{p+l})_{\bar{B}}^{(q-1)}~(= \alpha),
				\end{aligned}
			\end{equation*}
			where, $((\sigma_i)_{\overline{A \cap B}}, (\tau_i)_{\overline{A \cap B}}) \in \mathcal{W}_{\overline{A \cap B}}$, for each $i \in [p]$, $(\sigma_{i})_{\overline{A \cap B}} \subseteq (\tau_{i-1})_{\overline{A \cap B}}$ and $((\sigma_i)_{\overline{A \cap B}}, (\tau_{i-1})_{\overline{A \cap B}}) \notin \mathcal{W}_{\overline{A \cap B}}$ for each $i \in [p]$, $((\tau_i)_{\bar{B}},(\alpha_{i})_{\bar{A}}) \in \mathcal{W}_{\bar{B}}$, for $i \in \{p, \dots , (p+l-1)\}$, $(\tau_{i})_{\bar{B}} \subseteq (\alpha_{i-1})_{\bar{B}}$ and $((\tau_i)_{\bar{B}}, (\alpha_{i-1})_{\bar{B}}) \notin \mathcal{W}_{\bar{B}}$ for each $i \in \{p+1, \dots ,p+l\}$, $p, l \geq 0$. 
		
			 The weight of such a trajectory is given by,
			
			$$w_M(P):= \left(\prod_{i=0}^{p-1}-\langle \tau_i, \sigma_{i+1} \rangle \langle \tau_{i+1}, \sigma_{i+1}\rangle \right)\left(\prod_{i=p}^{p+l-1}-\langle \alpha_i, \tau_i\rangle \langle \alpha_i, \tau_{i+1} \rangle\right).$$
			
	\end{enumerate}

	\end{defin}
	For any $\beta \in D_q(X)$ and $\alpha \in D_{q-1}(X)$, we denote the set of all MV trajectories from $\beta$ to $\alpha$ as $\MV(\beta, \alpha)$. 
	
     For a non-negative integer $q$, we define $\mathcal{D}_q(X)$ as, $$\mathcal{D}_q(X):= \mathbb{Z} \langle D_{q}(X)\rangle.$$

		 For a non-negative integer $q$, we define a map $\partial^{\mathcal{D}}_{q+1} : \mathcal{D}_{q+1}(X) \longrightarrow \mathcal{D}_q(X)$ on the set of generators as follows. For each $\beta \in D_{q+1}(X)$, we define, 
		
	$$ \partial_{q+1}^{\mathcal{D}}(\beta) = \sum_{\alpha \in D_q(X)} \left( \sum_{P \in \MV(\beta, \alpha)}w_M(P) \right) \alpha,$$
	 and extend it linearly to $\mathcal{D}_{q+1}(X)$.

		 Now we are ready to state our main result.
		
		\begin{thm}\label{main}
			Let $X$ be a $d$-dimensional simplicial complex with subcomplexes $A$ and $B$ such that $A \cup B=X$. If $\mathcal{W}_{\bar{A}}$, $\mathcal{W}_{\bar{B}}$ and $\mathcal{W}_{\overline{A \cap B}}$ are gradient vector fields on  $\bar{A}$, $\bar{B}$ and $\overline{A \cap B}$  respectively, where $\bar{A}$, $\bar{B}$ and $\overline{A \cap B}$ are disjoint copies of $A$, $B$, and $A \cap B$ respectively (as defined before), then $(\mathcal{D}_q(X), \partial^{\mathcal{D}}_q)_{q \geq 0}$ is a chain complex, and $H^{\mathcal{D}}_q(X) \cong H_q(X)$, for each $q=0, \dots ,d$, where $H_q(X)$ denotes the $q$-th simplicial homology group of $X$ and $H^{\mathcal{D}}_q(X)$ denotes the $q$-th homology group of $\mathcal{D}_q(X)$.
		\end{thm}
	
		One novel aspect of our theorem is that, unlike the usual Mayer-Vietoris theorem, we do not need to know the individual homology groups $H_{*}(A), H_{*}(B)$ and $H_{*}(A \cap B)$ (which would require a significant number of additional computational steps). We only need the information of the trajectories of the respective gradient vector fields. It is worthy of mention at this point, that in principle, we can always compute the homology of $X$ explicitly using our theorem (\autoref{main}). Moreover, if we choose the subcomplexes $A$ and $B$ wisely, such that each subcomplex $A$, $B$ and $A \cap B$ admits an efficient gradient vector field (i.e., with a lower number of critical simplices), then it eases the computation of the homology groups considerably.\\

	 As an illustrative example, we compute the homology of $S_0 * S_0 * S_0$ using \autoref{main}. \\
	
	 Let $X=S_0 * S_0 * S_0$. So, we can also write $X$ as $\{v_0, v_2\} * \{v_1, v_3\} * \{v_4, v_5\}$. Let $A$ and $B$ be two subcomplexes of $X$, where $A= \{v_0, v_2\} * \{v_1, v_3\} * \{v_5\}$, $B=\{v_0, v_2\} * \{v_1, v_3\} * \{v_4\}$. Thus, $A \cap B= \{v_0, v_2\} * \{v_1, v_3\}$. We relabel the vertices of $A$ as $a_i$, the vertices of $B$ as $b_i$ and those of $A \cap B$ as $c_i$ corresponding to $v_i$ for each $i$. Therefore, $\bar{A}= \{\bar{a}_0, \bar{a}_2\} * \{\bar{a}_1, \bar{a}_3\} * \{\bar{a}_5\}$, $\bar{B}=\{\bar{b}_0, \bar{b}_2\} * \{\bar{b}_1, \bar{b}_3\} * \{\bar{b}_4\}$ and, $\overline{A \cap B} = \{\bar{c}_0, \bar{c}_2\} * \{\bar{c}_1, \bar{c}_3\}$ (as depicted in Figure~\ref{exmp}).\\
	
	\begin{figure}[h]
		\centering
	\begin{tikzpicture}[
		vertex/.style={draw, circle, fill=black, inner sep=0.05pt, minimum size=0.7mm},
		baseline]
		
		\begin{scope}
			\node[vertex,label=below:$v_0$] (v1) at (0,0) {};
			\node[vertex,label=below left:$v_3$] (v2) at (1.8,-1) {};
			\node[vertex,label={[xshift=2pt,yshift=-1pt]above right: $v_1$}] (v3) at (2.8,0.3) {};
	 		\node[vertex,label=below right: $v_2$] (v4) at (4.5,-0.7) {};
	 		\node[vertex,label=above:$v_4$] (v6) at (2.2,2.5) {};
	 		\node[vertex,label=below:$v_5$] (v5) at (2.5,-3.4) {};
	 		\node at (2,-4.5) {$X$};
			
			\draw (v1) -- (v2);
			\draw (v2) -- (v4);
			\draw[dotted] (v3) -- (v4);
			\draw[dotted] (v1) -- (v3);
			\draw (v1) -- (v6);
			\draw (v2) -- (v6);
			\draw[dotted] (v3) -- (v6);
			\draw (v4) -- (v6);
			\draw (v1) -- (v5);
			\draw (v2) -- (v5);
			\draw[dotted] (v3) -- (v5);
			\draw (v4) -- (v5);
			

		\end{scope}

		\begin{scope}[xshift=9cm]
			\node[vertex, label=below:$\bar{b}_0$] (v1) at (0,0) {};
			\node[vertex,label=below left:$\bar{b}_3$] (v2) at (1.8,-1) {};
			\node[vertex,label=below right: $\bar{b}_2$] (v4) at (4.5,-0.7) {};
			\node[vertex,label={[xshift=0.1pt,yshift=-1pt]above right: $\bar{b}_1$}] (v3) at (2.8,0.3) {};
			
			\node[vertex, fill=black!90, minimum size=2mm, label=above:$\bar{b}_4$] (v6) at (2.2,2) {};
			\node at (3.5,2) {$\bar{B}$};
			\node[vertex,label=below:$\bar{c}_0$ ] (v7) at (0,-2) {};
			\node[vertex,label=below:$\bar{c}_3$] (v8) at (1.8,-3) {};
			\node[vertex, fill=black!90, minimum size=2mm, label=below:$\bar{c}_2$] (v9) at (4.5,-2.7) {};
			\node at (6, -2.7) {$\overline{A \cap B}$};
			\node[vertex, label=below:$\bar{c}_1$] (v10) at (2.8,-1.7) {};
			\node[vertex, label=below:$\bar{a}_0$] (v11) at (0,-4) {};
			\node[vertex, label={[xshift=5pt, yshift=-1pt]below left: $\bar{a}_3$}] (v12) at (1.8,-5) {};
			\node[vertex, label=below right:$\bar{a}_2$] (v13) at (4.5,-4.7) {};
			\node at (5, -6) {$\bar{A}$};
			\node[vertex, label={[xshift=0pt, yshift=-2pt]above right:$\bar{a}_1$}] (v14) at (2.8,-3.7) {};
			\node[vertex, fill=black!90, minimum size=2mm, label=below:$\bar{a}_5$ ] (v15) at (2.5,-7) {};

		\draw (v1) -- (v2);
		\draw (v2) -- (v4);
		\draw[dotted] (v3) -- (v4);
		\draw[dotted] (v1) -- (v3);
		\draw (v1) -- (v6);
		\draw (v2) -- (v6);
		\draw[dotted] (v3) -- (v6);
		\draw (v4) -- (v6);
		\draw (v7) -- (v8);
		\draw[ultra thick] (v8) -- (v9);
		\draw (v9) -- (v10);
		\draw (v7) -- (v10);
		\draw (v11) -- (v12);
		\draw (v12) -- (v13);
		\draw (v13) -- (v14);
		\draw[dotted] (v11) -- (v14);
		\draw (v11) -- (v15);
		\draw (v12) -- (v15);
		\draw (v13) -- (v15);
		\draw[dotted] (v14) -- (v15);
		\draw[->,thick] ($(v1)!0.5!(v3)$) -- (1.6,0.6);
		\draw[->,thick] ($(v1)!0.5!(v2)$) -- (1.2, 0);
		\draw[->,thick] ($(v2)!0.5!(v4)$) -- (2.9, -0.3);
		\draw[->,thick] ($(v3)!0.5!(v4)$) -- (3.3,0.3);
		\draw[->,thick] ($(v11)!0.5!(v12)$) -- (1.2,-5);
		\draw[->,thick] ($(v12)!0.5!(v13)$) -- (3,-5.5);
		\draw[->,thick] ($(v13)!0.5!(v14)$) -- (3.4,-4.7);
		\draw[->,thick] ($(v11)!0.5!(v14)$) -- (1.5,-4.3);
		\draw[->,thick] (v1) -- ($(v1)!0.2!(v6)$);
		\draw[->,thick] (v2) -- ($(v2)!0.2!(v6)$);
		\draw[->,thick] (v4) -- ($(v4)!0.2!(v6)$);
		\draw[->,thick] (v8) -- ($(v8)!0.2!(v7)$);
		\draw[->,thick] (v7) -- ($(v7)!0.2!(v10)$);
		\draw[->,thick] (v10) -- ($(v10)!0.3!(v9)$);
		\draw[->,thick] (v11) -- ($(v11)!0.2!(v15)$);
		\draw[->,thick] (v14) -- ($(v14)!0.2!(v15)$);
		\draw[->,thick] (v13) -- ($(v13)!0.25!(v15)$);
		\draw[->,thick] (v12) -- ($(v12)!0.3!(v15)$);
		\draw[->,thick] (v3) -- ($(v3)!0.3!(v6)$);

		\end{scope}
		
		\draw[->, thick] (6,0) -- (7,0);

	\end{tikzpicture}
	\caption{$S_0 * S_0 * S_0$, and the disjoint copies of its subcomplexes $A$, $B$ and $A \cap B$. The arrows represent the respective gradient vector fields while the critical simplices have been marked in bold.} \label{exmp}
	
\end{figure}
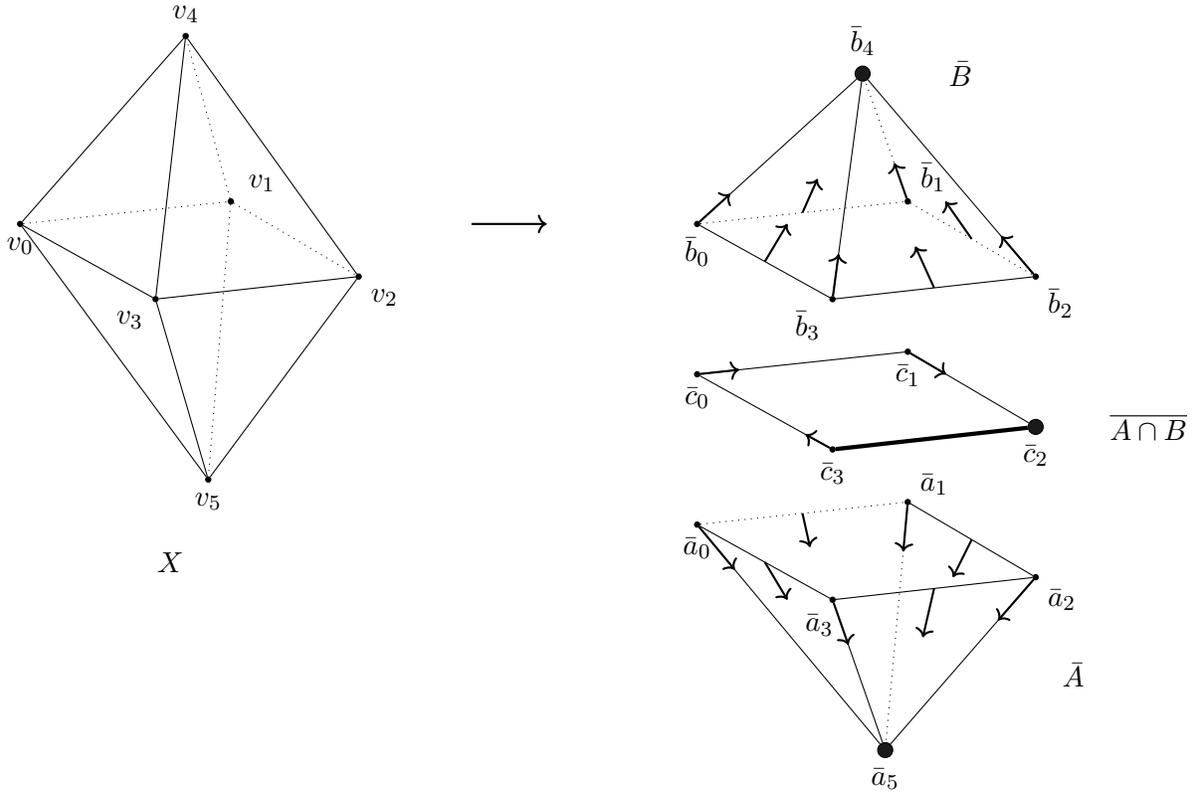
	
	Next, we define gradient vector fields $\mathcal{W}_{\bar{A}}$ on $\bar{A}$, $\mathcal{W}_{\bar{B}}$ on $\bar{B}$, and $\mathcal{W}_{\overline{A \cap B}}$ on $\overline{A \cap B}$. Let,
			$$\mathcal{W}_{\bar{A}} = \{ (\alpha \setminus \{\bar{b}_4\}, \alpha) \mid \alpha \in \bar{A}, \bar{b}_4 \in \alpha\},$$
			

		$$	\mathcal{W}_{\bar{B}} = \{ (\alpha \setminus \{\bar{a}_5\}, \alpha) \mid \alpha \in \bar{B}, \bar{a}_5 \in \alpha\}, $$
			
			
			$$\mathcal{W}_{\overline{A \cap B}} = \{([\bar{c}_3], [\bar{c}_0, \bar{c}_3]),([\bar{c}_0],[\bar{c}_0, \bar{c}_1]), ([\bar{c}_1], [\bar{c}_1, \bar{c}_2]) \}$$.
	
	 Hence, $D_{\#}(X)$ is given by,
	
	$$	D_0(X) =\Crit^{\mathcal{W}_{\bar{A}}}_0(\bar{A}) \cup \Crit^{\mathcal{W}_{\bar{B}}}_0(\bar{B}) = \{[\bar{a}_5], [\bar{b}_4]\},$$
	
	$$ D_1(X)= \Crit^{\mathcal{W}_{\bar{A}}}_1(\bar{A}) \cup \Crit_1^{\mathcal{W}_{\bar{B}}}(\bar{B}) \cup \Crit_0^{\mathcal{W}_{\overline{A \cap B}}}(\overline{A \cap B})=\{[\bar{c}_2]\},$$
	
	$$  D_2(X)= \Crit^{\mathcal{W}_{\bar{A}}}_2(\bar{A}) \cup \Crit^{\mathcal{W}_{\bar{B}}}_2(\bar{B}) \cup \Crit^{\mathcal{W}_{\overline{A \cap B}}}_1(\overline{A \cap B}) = \{[\bar{c}_2, \bar{c}_3]\}. $$
	
	 The chain complex goes as follows.
	
	$$ 0 \rightarrow \mathcal{D}_2(X) \xrightarrow{\partial_2^{\mathcal{D}}} \mathcal{D}_1(X) \xrightarrow{\partial_1^{\mathcal{D}}} \mathcal{D}_0(X) \xrightarrow{\partial_0^{\mathcal{D}}} 0.$$
	
	 The $\MV$ trajectories from $D_1(X)$ to $D_0(X)$ are given by,
	
	$$ P_1: [\bar{c}_2], [\bar{a}_2, \bar{a}_5], [\bar{a}_5].$$
	$$ P_2: [\bar{c}_2], [\bar{b_2}, \bar{b}_4], [\bar{b}_4].$$
	
	 The weights of these $\MV$ trajectories are,
	
	$$w_M(P_1)= - (-\langle [v_2, v_5], [v_2] \rangle \langle [v_2, v_5], [v_5] \rangle) = -1.$$
	$$ w_M(P_2)= (-\langle [v_2, v_4], [v_2] \rangle \langle [v_2, v_4], [v_4] \rangle) = 1.$$
	
	 Thus, 
	\begin{equation*}
		\begin{aligned}
			\partial_1^{\mathcal{D}}([\bar{c}_2])&= w_M(P_1)[\bar{a}_5] + w_M(P_2)[\bar{b}_4]\\
			&=[\bar{b}_4] -[\bar{a}_5].
		\end{aligned}
	\end{equation*}

	  Therefore, $H_0^{\mathcal{D}}(X)= \frac{\Ker(\partial_0^{\mathcal{D}})}{\operatorname{Im}(\partial_1^{\mathcal{D}})} \cong \frac{\langle [\bar{a}_5], [\bar{b}_4] \rangle}{\langle [\bar{b}_4] - [\bar{a}_5]\rangle} \cong \mathbb{Z}$.\\
	 
	  It follows that $\Ker(\partial_1^{\mathcal{D}})= 0$. Thus, $H_1^{\mathcal{D}}(X)= \frac{\Ker(\partial_1^{\mathcal{D}})}{\operatorname{Im}(\partial_2^{\mathcal{D}})} \cong 0.$
	 
	  Next, the $\MV$ trajectories from $D_2$ to $D_1$ are given by,
	 
	 $$ Q_1: [\bar{c}_2, \bar{c}_3], [\bar{c}_2].$$
	 $$ Q_2: [\bar{c}_2, \bar{c}_3], [\bar{c}_3], [\bar{c}_0, \bar{c}_3], [\bar{c}_0], [\bar{c}_0, \bar{c}_1], [\bar{c_1}], [\bar{c}_1, \bar{c}_2], [\bar{c}_2] . $$
	 
	  Therefore, the weights of these $\MV$ trajectories are given by,
	 
	 $$w_M(Q_1)= - \langle [v_2, v_3], [v_2] \rangle = 1. $$
	 \begin{equation*}
	 	\begin{aligned}
	 		 w_M(Q_2)&= - (-\langle [v_2, v_3], [v_3]\rangle \langle [v_0,v_3], [v_3]\rangle)(-\langle [v_0, v_3], [v_0]\rangle \langle [v_0,v_1], [v_0]\rangle)\\ & \hspace{2em}(-\langle [v_0,v_1], [v_1] \rangle \langle [v_1, v_2], [v_1]\rangle) \langle[v_1, v_2], [v_2] \rangle\\
	 		 &=-1.
	 	\end{aligned}
	 \end{equation*} 

	  Thus,

		\begin{equation*}
		\begin{aligned}
			\partial_2^{\mathcal{D}}([\bar{c}_2, \bar{c}_3])&= w_M(Q_1)[\bar{c}_2] + w_M(Q_2)[\bar{c}_2]\\
			&=[\bar{c}_2] -[\bar{c}_2] =0.
		\end{aligned}
	\end{equation*}
		
		 Therefore, $H_2^{\mathcal{D}}(X)= \Ker(\partial_1^{\mathcal{D}}) \cong \langle [\bar{c}_2, \bar{c}_3]\rangle \cong \mathbb{Z}$.
		
	\end{section}

	\begin{section}{Preliminaries}\label{prelim}
			
		A \emph{simplicial complex} $X$ is defined as an ordered pair $(V, \mathcal{F})$, where $V$ is a finite non-empty set and $\mathcal{F}$ denotes a non-empty collection of subsets of $V$, with the property that, if $\sigma \in \mathcal{F}$, then, every subset of $\sigma$ must also be in $\mathcal{F}$. An element of $V$ are said to be a vertex while an element of $\mathcal{F}$ is termed as a simplex.

	  For any simplical complex $X$, the vertex set is denoted by $V(X)$ and the set of simplices is denoted by $\mathcal{F}(X)$. Abusing the notation, we denote the set of simplices as simply $X$ instead of $\mathcal{F}(X)$. The dimension of a simplex $\sigma$ is given by $\dim(\sigma)= |\sigma| - 1$. The dimension of a simplicial complex $X$ is defined as, $\dim(X) = \max\{\dim(\sigma) \mid \sigma \in X\}$. If $\sigma \subseteq \tau$, and $\dim(\sigma)= \dim(\tau)-1$, then $\sigma$ is said to be a \emph{facet} of $\tau$. If $\sigma$ is a simplex, which is not contained in any other simplex, then $\sigma$ is said to be a \emph{maximal} simplex.

	A simplicial complex $A$ is said to be a subcomplex of a simplicial complex $B$ if $A \subseteq B$. The union of two simplicial complexes, $X$ and $Y$, is defined as $X\cup Y$. The intersection of $X$ and $Y$ is defined as $X \cap Y$. It can be verified that both union and intersection of simplicial complexes are themselves simplicial complexes.

			\begin{defin}{(Prism of a Simplicial Complex).}
				Let $X$ be a simplicial complex with $V(X) = \{v_0,v_1, \dots ,v_n\}$. We define \emph{prism of the simplicial complex} $X$, as $\mathbb{P}(X)$, whose  set of vertices is given by, $V(\mathbb{P}(X))=\{a_0, a_1, \dots ,a_n, b_0, b_1, \dots ,b_n\}$, where each $a_i$, $b_i$ corresponds to $v_i$, $i \in \{0, \dots ,n\}$ such that $\{a_0, \dots ,a_n\}$ and $\{b_0, \dots ,b_n\}$ are disjoint. A maximal simplex of $\mathbb{P}(X)$ is of the form, $\{a_{i_0}, \dots ,a_{i_r}, b_{i_r},\dots ,b_{i_q} \}$, $i_0 < i_1, \dots < i_q$, where $\{v_{i_0},\dots ,v_{i_q} \}$ is a maximal simplex of $X$, $0 \leq r \leq q$.  
				\end{defin}

				 Let $\alpha=\{v_{i_0}, \dots ,v_{i_q}\} \in X$. 
				
				Then, we define $A_{\alpha}:= \{ \{a_{i_0}, \dots ,a_{i_r},b_{i_r}, \dots ,b_{i_q}\} \mid 0 \leq r \leq q\}$ and $B_{\alpha} := \{ \{a_{i_0}, \dots , a_{i_{r-1}}, b_{i_r}, \dots \bar{b}_{i_q}\} \mid  1 \leq r \leq q\} \cup \{\{\bar{a}_{i_0}, \dots ,\bar{a}_{i_q}\}\}  \cup \{\{\bar{b}_{i_0}, \dots ,\bar{b}_{i_q}\}\}$. We define, $S_{\alpha}:= A_{\alpha} \cup B_{\alpha}$ and note that $\bigcup_{\alpha \in X}S_{\alpha}= \mathbb{P}(X)$, where $S_{\alpha}$-s are mutually disjoint sets.
				
				 For any simplex $\sigma \in S_{\alpha}$, $\alpha$ is said to be the \emph{ground simplex} of $\sigma$, denoted as $\GS(\sigma)$.
				
			The following diagram illustrates the prism of a two-dimensional simplicial complex $X$, where $X= \{\{v_0\}, \{v_1\}, \{v_2\}, \{v_0, v_1\}, \{v_0, v_2\}, \{v_1, v_2\}, \{v_0, v_1, v_2\}\}$.
				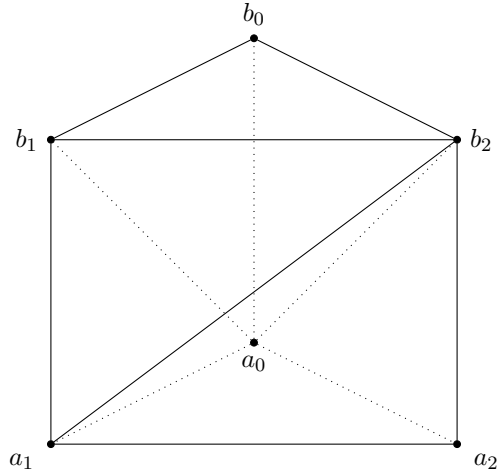
\begin{figure}[h]
					\centering
				
					\begin{tikzpicture}[ vertex/.style={draw, circle, fill=black, inner sep=0.05pt, minimum size=1mm},baseline, scale=0.9, transform shape]
					\node[vertex,label=below: $a_0$] (v1) at (0,0) {};
					\node[vertex, label=above: $b_0$ ] (v2) at (0,4.5) {};
					\node[vertex, label={[xshift=-2pt,yshift=1pt]below left: $a_1$}] (v3) at (-3,-1.5){};
					\node[vertex, label={[xshift=2pt,yshift=1pt]below right: $a_2$}] (v4) at (3,-1.5){};
					\node[vertex, label=left:$b_1$] (v5) at (-3,3){};
					\node[vertex, label=right:$b_2$] (v6) at (3,3){};
					
					\draw (v3) -- (v6);
					\draw[dotted] (v1) -- (v2);
					\draw (v2) -- (v5);
					\draw (v2) -- (v6);
					\draw (v5) -- (v6);
					\draw (v3) -- (v5);
					\draw (v4) -- (v6);
					\draw (v3) -- (v4);
					\draw[dotted] (v1) -- (v3);
					\draw [dotted] (v1) -- (v4);
					\draw[dotted] (v1) -- (v5);
					\draw[dotted] (v1) -- (v6);
					
				\end{tikzpicture}	
				\caption{An illustration of $\mathbb{P}(X)$.}
				\end{figure}
				
	We note from the figure that the maximal simplices of $\mathbb{P}(X)$ are $\{a_0, b_0, b_1, b_2\}, \{a_0, a_1, b_1, b_2\}$ and $\{a_0, a_1, a_2, b_2\}$.
			\begin{defin}
				The \emph{orientation} of a simplex $\{ v_{0},v_{1}, \dots v_{n} \}$ is defined as an ordering of the vertices of the simplex based on the following equivalence relation on all possible orderings of $\{ v_{0},v_{1}, \dots v_{n}\}$.
				$$ \{v_{i_{0}},v_{i_{1}}, \dots v_{i_{n}}\} \sim \{v_{j_{0}},v_{j_{1}}, \dots v_{j_{n}}\} \text{ if } \{i_{0},i_{1}, \dots ,i_{n}\} \text{ and } \{j_{0},j_{1}, \dots ,j_{n}\} \text{ differ by an even permutation}.$$
			\end{defin}
			 Two orientations are said to be same if they belong to the same equivalence class and different otherwise. Thus, for a simplex $\sigma$ such that $\dim(\sigma) \geq 1$, there are precisely two orientations corresponding to the two equivalence classes. A simplex with a given orientation is said to be an \emph{oriented} simplex. An oriented simplex $\{v_0, \dots ,v_n\}$ with the orientation given by the ordering $v_0, \dots ,v_n$ is denoted by $[v_0, \dots ,v_n]$. This represents the orientation of all the simplices in the same equivalence class as $\{ v_{0},v_{1}, \dots v_{n} \}$.\\

			Let $[v_0, \dots ,v_n]$ be an oriented simplex. Then $[v_0, \dots ,\widehat{v_i}, \dots ,v_n]$ is the oriented simplex with the vertex set $\{v_0, \dots ,v_n\} \setminus \{ v_i\}$ and the orientation given by the ordering $v_0, \dots,v_{i-1},v_{i+1}, \dots ,v_n$. 
			
			\begin{defin}{(Incidence Number).} 
				Let $X$ be a simplicial complex and $\sigma, \tau \in X$ such that  $\dim(\tau)=q$, $\dim(\sigma)=q-1$. Let $\tau=[v_0, \dots ,v_q]$. Then the \emph{incidence number} of $\sigma$ with respect to $\tau$ is given by,
				$$ \langle \tau, \sigma \rangle := \begin{cases}
					0 & \text{if } \sigma \not\subset \tau,\\
					(-1)^i & \text{if } \sigma=[v_0,\dots, \widehat{v_i}, \dots ,v_q]. \\
				\end{cases}$$
			\end{defin}
			\vspace{0.15em}
			 \textbf{The Simplicial Chain Complex:} Let $S_{q}(X)$ denote the set of all $q$-dimensional simplices of $X$, and $C_{q}(X)$ be the free abelian group generated by $S_{q}(X)$ over $\mathbb{Z}$. We define the $(q+1)^{th}$ boundary operator, $\partial_{q+1}: C_{q+1}(X) \rightarrow C_{q}(X)$ as,
			
			$$ \partial_{q+1} (\tau) = \sum_{\sigma \in S_{q}(X)} \langle \tau, \sigma \rangle \sigma, \text{    for each    } \text{ } \tau \in S_{q+1}(X),$$
			and extend it linearly to all of $C_{q+1}(X)$.
			Thus, the simplicial chain complex goes as,
			$$ \dots \rightarrow C_{q+1}(X) \xrightarrow{\partial_{q+1}}  C_{q}(X) \xrightarrow{\partial_q} C_{q-1}(X) \xrightarrow{\partial_q} \cdots \rightarrow C_0(X) \rightarrow 0.$$
			
			 The $q^{th}$ homology group is given by $H_{q}(X):= \frac{\Ker(\partial_{q})}{ \operatorname{Im}(\partial_{q+1})}$.\\
			
		 Now we explain some basic notions of homotopy equivalence and isomorphism between chain complexes.
		
		 A \emph{chain map} between two chain complexes $(K_{\#} =\{k_n\}_{n \geq 0}, \partial_{n}^{C})$, and $(L_{\#} =\{L_n\}_{n \geq 0}, \partial_{n}^{L})$ is defined as a sequence $f_{\#}=\{f_{n}\}_{n\geq0}$ of module homomorphisms from $K_{\#}$ to $L_{\#}$ such that,
		  $$ f_{n-1}\circ \partial_{n}=\partial_{n}\circ f_{n}  \text{ for each } i \geq 0.$$
		 
			\[\begin{tikzcd}
				\cdots & { K_{n+1}} && {K_{n}} && {K_{n-1}} & \cdots \\
				\\
				\cdots & {L_{n+1}} && {L_{n}} && {L_{n-1}} & \cdots \\
				\\
				&&&&&&& {}
				\arrow[from=1-1, to=1-2]
				\arrow["{\partial_{n+1}^{K}}", from=1-2, to=1-4]
				\arrow["{f_{n+1}}"', from=1-2, to=3-2]
				\arrow["{\partial_{n}^{K}}", from=1-4, to=1-6]
				\arrow["{\lambda_{n}}"', from=1-4, to=3-2]
				\arrow["{f_{n}}"', from=1-4, to=3-4]
				\arrow[from=1-6, to=1-7]
				\arrow["{\lambda_{n-1}}"', from=1-6, to=3-4]
				\arrow["{f_{n-1}}"', from=1-6, to=3-6]
				\arrow[from=3-1, to=3-2]
				\arrow["{\partial_{n+1}^{L}}"', from=3-2, to=3-4]
				\arrow["{\partial_{n}^{L}}"', from=3-4, to=3-6]
				\arrow[from=3-6, to=3-7]
			\end{tikzcd}\]
				Given two chain maps $f_{\#}$ and $g_{\#}$ between the chain complexes $K_{\#}$ and $L_{\#}$, a \emph{chain homotopy} between $f_{\#}$ and $g_{\#}$ is defined as sequence of module homomorphisms, $\lambda_{n} : K_{n} \longrightarrow L_{n+1}, n \geq 0$ such that, 
			$$ \lambda_{n-1}\circ \partial_{n}^K + \partial_{n+1}^L\circ \lambda_{n} = f_{n} - g_{n}, \text{  for all } n \geq 0.$$ 
				In this case $f_{n}$ and $g_{n}$ are said to be chain homotopic and denoted as $f_{n} \simeq g_{n}$.

			Two chain complexes $K_{\#}$ and $L_{\#}$ are said to be \emph{homotopy equivalent} if for a chain map $f_{\#}:K_{\#} \longrightarrow L_{\#}$, there exists another chain map $g_{\#}: L_{\#} \longrightarrow K_{\#}$ such that, $f_{\#} \circ g_{\#} \simeq id_{L_{\#}}$ and $g_{\#} \circ f_{\#} \simeq id_{K_{\#}}$. The chain map $f_{\#}$ is said to be a \emph{homotopy equivalence}. If two chain complexes are homotopy equivalent, then their respective homology groups are isomorphic. Two chain complexes $K_{\#}$ and $L_{\#}$ are said to be \emph{isomorphic} if, there exists a chain map $f_{\#}:K_{\#} \longrightarrow L_{\#}$ which is an isomorphism. If two chain complexes are isomorphic, then they are homotopy equivalent.\\

		Next, we introduce a few notations related to discrete Morse theory, which we will use in the subsequent sections in our paper.
		
		 Let $X$ be a simplicial complex and $\mathcal{V}$ be a gradient vector field on $X$. Then the simplices which appear in $\mathcal{V}$ are termed as \emph{regular} simplices. We also recall that the simplices in $X$ which do not appear in $\mathcal{V}$ are termed as $\mathcal{V}$-critical. The set of all $\mathcal{V}$-critical simplices in $X$ is denoted as $\Crit^{\mathcal{V}}(X)$. Thus, $$\Crit^{\mathcal{V}}(X) = \bigcup_{q=0}^{\dim(X)} \Crit_q^{\mathcal{V}}(X).$$ 
		where $\Crit_q^{\mathcal{V}}(X)$ represents the set of all $q$-dimensional $\mathcal{V}$-critical simplices in $X$.
		
	 Next, we recall that an extended $\mathcal{V}$-trajectory $P$ is a sequence of simplices of the following form.

	 $$\tau_0^{(q)}, \sigma_1^{(q-1)}, \tau_1^{(q)}, \dots \sigma_k^{(q-1)}, \tau_k^{(q)}, \sigma_{k+1}^{(q-1)},$$
	  where  $(\sigma_{i}^{(q-1)}, \tau_{i}^{(q)}) \in \mathcal{V}$ for all $i \in [k]$,  $ \sigma_{i}^{(q-1)} \subseteq \tau_{i-1}^{(q)}$ and $(\sigma_i^{(q-1)}, \tau_{i-1}^{(q)}) \notin \mathcal{V}$ for all $i\in [k+1]$.
	 
	  We denote the terminal simplex $\sigma_{k+1}$ of $P$ as $t(P)$.

	Now, let $P: \alpha_0, \alpha_1, \dots ,\alpha_k$ be a (Forman) trajectory. Let $\{i_1, \dots ,i_r\} \subseteq \{0, \dots ,k\}$, such that $\{i_1, \dots ,i_r\}= \{i_1, i_1 +1, i_1+2, \dots , i_1 +(r-2), i_1 +(r-1)\}$ and $i_j=i_1 + (j-1)$, for each $j \in [r]$. Then we denote the subsequence $\alpha_{i_1}, \dots, \alpha_{i_r}$, of $P$ as $\alpha_{i_1} P \alpha_{i_r}$. For example, let, $P: \tau_0, \sigma_1, \tau_1, \dots , \sigma_k, \tau_k, \sigma_{k+1}$ be a trajectory. Then we denote a subsequence $\sigma_2, \tau_2, \sigma_3, \tau_3, \sigma_4$  of $P$ as $\sigma_1 P \sigma_4$. Similarly, for an $\MV$-trajectory $P:  \alpha_0, \alpha_1, \dots ,\alpha_k $, we denote the subsequence $\alpha_{i_1}, \dots, \alpha_{i_r}$, of $P$ as $\alpha_{i_1} P \alpha_{i_r}$.

	\end{section}
	
	\begin{section}{Proof of the main theorem}\label{first}

	 Let $X$ be a simplicial complex with subcomplexes $A$ and $B$ such that $A \cup B =X$. In this section, we first construct a new simplicial complex $\widetilde{X}$ from $X$ and show that the simplicial homology of $X$ and $\widetilde{X}$ are isomorphic. Next, we use the gradient vector fields $\mathcal{W}_{\bar{A}}$ on $\bar{A}$, $\mathcal{W}_{\bar{B}}$ on $\bar{B}$ and $\mathcal{W}_{\overline{A \cap B}}$ on $\overline{A \cap B}$ (where $\bar{A}$, $\bar{B}$ and $\overline{A \cap B}$ are disjoint copies of $A$, $B$ and $A \cap B$ respectively) to show that the simplicial homology of $\widetilde{X}$ is isomorphic to the homology group of $\mathcal{D}_q(X)$. This proves our main result.
	
		\subsection{Construction of a new simplicial complex $\widetilde{X}$}
			
		Let $X$ be a simplicial complex with subcomplexes $A$ and $B$ such that $A \cup B=X$. Let $V(A)=\{a_0, \dots ,a_n\}$, $V(B)=\{b_0, \dots ,b_p\}$ and $V(A \cap B)=\{c_0, \dots ,c_m\}$, where, without loss of generality, we assume, $c_i=a_i=b_i$ for each $i \in \{0, \dots ,m\}$. Consider the copies $\bar{A}$ and $\bar{B}$ of $A$ and $B$ respectively as we have defined in Section~\ref{intro}. Let $(A \cap B)_{\bar{A}}$ and $(A \cap B)_{\bar{B}}$ be the copies of $A \cap B$ in $\bar{A}$ and $\bar{ B}$ respectively. Now we construct $\mathbb{P}(\overline{A \cap B})$ with vertex set $V((A \cap B)_{\bar{A}}) \cup V((A \cap B)_{\bar{B}})$ and the set of simplices as defined in Section~\ref{prelim}, i.e., a maximal simplex of $\mathbb{P}(\overline{A \cap B})$ is of the form $[\bar{a}_{i_0}, \dots ,\bar{a}_{i_r}, \bar{b}_{i_r}, \dots ,\bar{b}_{i_q}]$, $r= 0 , \dots , k$, $i_0 < \dots < i_q$, where $[\bar{c}_{i_0}, \dots ,\bar{c}_{i_q}]$ is a maximal simplex of $X$.
	
		 Now we define a new simplicial complex $\widetilde{X}$ as $\widetilde{X}= \bar{A} \cup \mathbb{P}(\overline{A \cap B)} \cup \bar{B}$. We note here that $\widetilde{X}$, being union of simplicial complexes, is a simplicial complex.

		 An important observation here is that, for any simplex $\sigma \in \mathbb{P}(\overline{A \cap B)}$, if $a_i,b_j \in \sigma$, then $i \leq j$.

		Next, we recall from Section~\ref{prelim}, that for each $\alpha=[\bar{c}_{i_0}, \dots ,\bar{c}_{i_q}] \in \overline{A \cap B}$, the simplices of $\mathbb{P}(\overline{A \cap B})$ can be categorized into $A_{\alpha}= \{ [\bar{a}_{i_0}, \dots ,\bar{a}_{i_r},\bar{b}_{i_r}, \dots ,\bar{b}_{i_q}] \mid 0 \leq r \leq q\}$, $B_{\alpha} = \{ [\bar{a}_{i_0}, \dots , \bar{a}_{i_{r-1}}, \bar{b}_{i_r}, \dots \bar{b}_{i_q}] \mid  1 \leq r \leq q\} \cup \{[\bar{a}_{i_0}, \dots ,\bar{a}_{i_q}]\}  \cup \{[\bar{b}_{i_0}, \dots ,\bar{b}_{i_q}]\}$. Further we have defined $S_{\alpha}= A_{\alpha} \cup B_{\alpha}$. 
		
		 Thus, $\bigcup_{\alpha \in \overline{A \cap B}}S_{\alpha}= \mathbb{P}(\overline{A \cap B})$.
		
		We recall that for any simplex $\sigma \in S_{\alpha}$, $\alpha$ is said to be the \emph{ground simplex} of $\sigma$, which we denote as $\GS(\sigma)$.
		
	\subsection{Constructing a gradient vector field on $\widetilde{X}$}
		Here we construct a gradient vector field on $\widetilde{X}$ so that the Thom-Smale chain complex of $\widetilde{X}$ becomes isomorphic to the simplicial chain complex of $X$.\\
		
		  We know that $\mathbb{P}(\overline{A \cap B})=\bigcup_{\alpha \in \overline{A \cap B}}S_{\alpha}$. We define a collection of pairs on $S_{\alpha}$ for each $\alpha \in \overline{A \cap B}$. This gives us a discrete vector field $\mathcal{V}$ on $\widetilde{X}$.
		 
		 The pairing scheme goes as follows.  
		
		 Let $\alpha=[\bar{c}_{i_0}, \dots ,\bar{c}_{i_q}] \in \overline{A \cap B}$.
		\begin{enumerate}[label=(\alph*)]
			\item  Let $\tau=[\bar{a}_{i_0}, \bar{b}_{i_0}, \dots , \bar{b}_{i_q}] \in A_{\alpha}$. Then $([\bar{b}_{i_0}, \dots , \bar{b}_{i_q}], [\bar{a}_{i_0}, \bar{b}_{i_0}, \dots , \bar{b}_{i_q}]) \in \mathcal{V}$.
			
			\item Let $\tau=[\bar{a}_{i_0}, \dots ,\bar{a}_{i_{r-1}}, \bar{a}_{i_r},\bar{b}_{i_r}, \dots ,\bar{b}_{i_q}]  \in A_{\alpha}$, 0 $\leq r \leq q$. \\ Then, $([\bar{a}_{i_0}, \dots ,\bar{a}_{i_{r-1}},\bar{b}_{i_r}, \dots ,\bar{b}_{i_q}], [\bar{a}_{i_0}, \dots ,\bar{a}_{i_{r-1}}, \bar{a}_{i_r},\bar{b}_{i_r}, \dots ,\bar{b}_{i_q}] ) \in \mathcal{V}$.
		\end{enumerate}
		 
		Thus, each $\tau \in A_{\alpha}$ gets paired with a unique simplex from $B_{\alpha}$. We note that $|A_{\alpha}|=q+1$, $|B_{\alpha}|=q+2$. So, the only unpaired simplex left in $S_{\alpha}$ is $[\bar{a}_{i_0}, \dots ,\bar{a}_{i_q}]$.
		
		 Therefore, the only unpaired simplices left in $\mathbb{P}(\overline{A \cap B})$ are, $$\{[\bar{a}_{i_0}, \dots ,\bar{a}_{i_q}] \mid [\bar{c}_{i_0}, \dots ,\bar{c}_{i_q}] \in \overline{A \cap B}\}= (A \cap B)_{\bar{A}},$$
		
		 and hence the unpaired simplices in $\widetilde{X}$ are given by $(\widetilde{X} \setminus \mathbb{P}(\overline{A \cap B})) \cup (A \cap B)_{\bar{A}}$.
	\begin{remark}
		For each $(\sigma^{(q-1)}, \tau^{(q)}) \in \mathcal{V}$, $\sigma \in B_{\alpha}, \tau \in A_{\alpha}$. Conversely, for any $\alpha \in \overline{A \cap B}$, every element of $A_{\alpha}$ is paired with a lower dimensional simplex, while every simplex of $B_{\alpha}$ which has been paired, is paired with a higher dimensional simplex.
	\end{remark}
		
	 Now we need to show that $\mathcal{V}$ is a gradient vector field.
		
		\begin{thm}
			The discrete vector field $\mathcal{V}$ is a gradient vector field.
		\end{thm}
		
		\begin{proof}
			
			Let, if possible, $\tau_{0}^{(q)}, \sigma_{1}^{(q-1)}, \tau_{1}^{(q)}, \dots ,\sigma_{k}^{(q-1)}, \tau_{k}^{(q)} =\tau_0$ be a closed $\mathcal{V}$-trajectory. 
			
			We note that $\tau_0~(=\tau_k)$, being paired with a lower dimensional simplex must be in $A_{\alpha}$ for some $\alpha \in \overline{A \cap B}$. So, let $\tau_0=[\bar{a}_{i_0}, \dots ,\bar{a}_{i_r}, \bar{b}_{i_r}, \dots ,\bar{b}_{i_q}]$.
			
			Now, $\sigma_1$ being paired with a higher dimensional simplex is in $B_{\alpha}$, for some $\alpha \in \overline{A \cap B}$. So, $\sigma_1$ can be either $[\bar{a}_{i_0}, \dots ,\bar{a}_{i_{r-1}}, \bar{b}_{i_r}, \dots ,\bar{b}_{i_q}]$ or $\sigma_1=[\bar{a}_{i_0}, \dots ,\bar{a}_{i_r}, \bar{b}_{i_{r+1}}, \dots ,\bar{b}_{i_q}]$. We note that the former simplex is paired with $\tau_0$, which is not possible.
			
			Therefore $\sigma_1= [\bar{a}_{i_0}, \dots ,\bar{a}_{i_r}, \bar{b}_{i_{r+1}}, \dots ,\bar{b}_{i_q}]$. According to the construction of $\mathcal{V}$, \\ $\tau_1=  [\bar{a}_{i_0}, \dots ,\bar{a}_{i_r}, \bar{a}_{i_{r+1}}, \bar{b}_{i_{r+1}}, \dots ,\bar{b}_{i_q}]$. Again, $\sigma_2$ being paired with a higher dimensional simplex is in $B_{\alpha}$ for some $\alpha \in \overline{A \cap B}$. Since, $\sigma_2 \ne \sigma_1$, therefore $\sigma_2=  [\bar{a}_{i_0}, \dots ,\bar{a}_{i_{r+1}}, \bar{b}_{i_{r+2}}, \dots ,\bar{b}_{i_q}]$, which makes $\tau_2= [\bar{a}_{i_0}, \dots ,\bar{a}_{i_{r+1}}, \bar{a}_{i_{r+2}}, \bar{b}_{i_{r+2}}, \dots ,\bar{b}_{i_q}] $. Thus, at each lower dimensional simplex $\sigma_j$, we remove a vertex $\bar{b}_i$ from $\tau_{j-1}$ and add a new vertex $\bar{a}_{i+1}$, when we pair it with the next higher dimensional simplex $\tau_j$. We observe that the vertex removed at each lower dimensional simplex does not return in the following simplices in the trajectory. Consequently, a simplex which has occurred once in the trajectory cannot appear again. So, $\tau_k \neq \tau_0$. 
			
			Thus, such a closed $\mathcal{V}$-trajectory is not possible.
			
		\end{proof}
		
		 So, the $\mathcal{V}$-critical simplices are precisely $(\widetilde{X} \setminus \mathbb{P}(\overline{A \cap B})) \cup (A \cap B)_{\bar{A}} = \bar{A} \cup (\bar{B} \setminus (A \cap B)_{\bar{B}})$.
		
		In the next subsection, we prove that the simplicial chain complex of $X$ and the Thom-Smale chain complex of $\widetilde{X}$ with respect to $\mathcal{V}$ are isomorphic.
		
		\begin{remark}
			It is clear that $\widetilde{X}$ and $X$ are homotopy equivalent and hence $H_q(\widetilde{X}) \cong H_q(X)$ for each $q \geq 0$. However, here we choose to provide a combinatorial argument of the same using \autoref{hom}, as a warm-up for the second step of the proof, which is a bit more involved but uses a somewhat similar idea.
		\end{remark}
		
		\subsection{Isomorphism between simplicial chain complex of $X$ and Thom-Smale complex of $\widetilde{X}$ with respect to $\mathcal{V}$}
		
			\begin{thm}\label{iso}
			The following isomorphism holds.
			$$(C_{\#}(X),\partial_{\#}) \cong (C_{\#}^\mathcal{V}(\widetilde{X}, \mathbb{Z}), \partial^\mathcal{V}_{\#}).$$ 
			Hence, $H_{\#}(X) \cong H_{\#}^{\mathcal{V}}(\widetilde{X}, \mathbb{Z})$.
		\end{thm}
		
		\begin{proof}
			
		First, we need to find a chain map, $g_{\#}:C_{\#}^{\mathcal{V}}(\widetilde{X}, \mathbb{Z}) \rightarrow C_{\#}(X)$, such that $g_{\#}$ is an isomorphism and the following diagram commutes.
				
		\[\begin{tikzcd}
			\cdots & {C_{q+1}(X)} && {C_{q}(X)} & \cdots \\
			\\
			\cdots & {C_{q+1}^\mathcal{V}(\widetilde{X}, \mathbb{Z})} && {C_{q}^\mathcal{V}(\widetilde{X}, \mathbb{Z})} & \cdots
			\arrow[from=1-1, to=1-2]
			\arrow["{\partial_{q+1}}", from=1-2, to=1-4]
			\arrow["{g_{q+1}}"', from=3-2, to=1-2]
			\arrow[from=1-4, to=1-5]
			\arrow["{g_{q}}"', from=3-4, to=1-4]
			\arrow[from=3-1, to=3-2]
			\arrow["{\partial^\mathcal{V}_{{q+1}}}"', from=3-2, to=3-4]
			\arrow[from=3-4, to=3-5]
			\arrow[draw=none, from=1-2, to=3-4, "{\text{\Huge$\circlearrowleft$}}" description]
		\end{tikzcd}\]

		 	 For $0 \leq q \leq \dim(X)$, it suffices to define $g_q$ on the generators. Let $\alpha \in \Crit_q^{\mathcal{V}}(\widetilde{X})$. Then, $\alpha= \sigma_{\bar{A}}$, for some $\alpha \in A$, or $\alpha= \sigma_{\bar{B}}$, for some $\sigma \in B \setminus A$. In both cases we define $ g_q(\alpha)= \sigma$. 
		 	
			
		 For $0 \leq q \leq \dim(X)$, we define $f_q: C_{q}(X) \rightarrow C_q^{\mathcal{V}}(\widetilde{X}, \mathbb{Z})$. It suffices to define $f_q$ on the generators. Let $\sigma \in S_q(X)$. Then we define,
		$$ f_q(\sigma)= \begin{cases}
			\sigma_{\bar{A}} & \text{ if } \sigma \in S_q(X) \cap A, \text{ and}\\
			\sigma_{\bar{B}} & \text{ if } \sigma \in S_q(X) \setminus A.\\
			
		\end{cases}$$

		Now, we show that these maps are isomorphisms. It suffices to show that $f_q \circ g_q= \operatorname{id}$ and $g_q \circ f_q = \operatorname{id}$ on the set of generators, where $\operatorname{id}$ represents the identity map. 
		
		Now, let $\sigma_{\bar{A}} \in \bar{A}$ for some $\sigma \in A$. Then $f_q \circ g_q(\sigma_{\bar{A}})=f_q(\sigma) =\sigma_{\bar{A}}$, since $\sigma \in A$. Let $\sigma_{\bar{B}} \in \bar{B} \setminus (A \cap B)_{\bar{B}}$, for some $\sigma \in B \setminus A$. Then,  $f_q \circ g_q(\sigma_{\bar{B}})=f_q(\sigma)=\sigma_{\bar{B}}$ since, $ \sigma \in B \setminus A$. Now, let $\sigma \in S_q(X) \cap A$. Then $g_q \circ f_q(\sigma)=g_q(\sigma_{\bar{A}})=\sigma$. Let $\sigma \in S_q(X) \setminus A$. Then, $g_q \circ f_q (\sigma)=g_q(\sigma_{\bar{B}})= \sigma$.
			
		 Thus, we have established an isomorphism between $C_{\#}(X)$ and $C_{\#}^{\mathcal{V}}(\widetilde{X}, \mathbb{Z})$. It remains to show that this is indeed a chain map, that is, $\partial_{q+1} \circ g_{q+1} = g_{q} \circ \partial_{q+1}^{\mathcal{V}}$. It suffices to show that this holds on the set of generators.\\
		
		 \textbf{Case I: } Let $\tau_A \in \Crit_{q+1}^{\mathcal{V}}(\widetilde{X})\cap \bar{A}$. As usual, here $\tau_{\bar{A}}$ denotes the copy of $\tau \in A$ in $\bar{A}$.
		
		 Then,  $$\partial_{q+1}(g_{q+1}(\tau_{\bar{A}}))=\partial_{q+1}(\tau)= \sum_{\sigma \in S_{q}(X)} \langle \tau, \sigma\rangle \sigma.$$  \\
		 However,  $$\partial_{q+1}^{\mathcal{V}}(\tau_{\bar{A}})= \sum_{\sigma \in \Crit_{q}^{\mathcal{V}}(\widetilde{X})}\left(\sum_{P \in \Gamma(\tau_{\bar{A}}, \sigma)}w(P)\right)\sigma.$$\\
		
		 Clearly, any facet of $\tau_{\bar{A}}$ is in $\bar{A}$ and hence, critical. So, for every $\sigma \in \Crit_{q-1}^{\mathcal{V}}(\widetilde{X})$, the term $\sum_{P \in \Gamma(\tau_{\bar{A}}, \sigma)}w(P)\sigma$ survives, only when $\sigma$ is a facet of $\tau_{\bar{A}}$. We can write $\sigma$ as $\alpha_{\bar{A}}$ for some $\alpha \in A$. Since every facet of $\tau_{\bar{A}}$ is critical, therefore the only trajectory from $\tau_{\bar{A}}$ to $\alpha_{\bar{A}}$ is the trivial trajectory $P: \tau_{\bar{A}}, \alpha_{\bar{A}}$. Thus, $w(P)= \langle \tau_{\bar{A}}, \alpha_{\bar{A}} \rangle= \langle \tau, \alpha \rangle $. So, \\

				\begin{equation*}
					\begin{aligned}
				\partial_{q+1}^{\mathcal{V}}(\tau_{\bar{A}}) & = \sum_{\alpha_{\bar{A}} \in \Crit_{q}^{\mathcal{V}}(\widetilde{X})}\langle \tau_{\bar{A}}, \alpha_{\bar{A}}\rangle \alpha_{\bar{A}}\\
				\implies g_{q}(\partial_{q+1}^{\mathcal{V}}(\tau_{\bar{A}}))&= \sum_{\alpha_{\bar{A}} \in \Crit_{q}^{\mathcal{V}}(\widetilde{X})}\langle \tau_{\bar{A}}, \alpha_{\bar{A}}\rangle g_{q}(\alpha_{\bar{A}})\\
				&= \sum_{\alpha \in S_{q}(X)}\langle \tau, \alpha \rangle \alpha \text{\hspace{2.5em} (since, $g_q$ is an isomorphism. } )\\
				& = \partial_{q+1}(g_{q+1}(\tau_{\bar{A}})). 
			\end{aligned}
		\end{equation*}
		
		 \textbf{Case II: } Let $\tau_{\bar{B}} \in \Crit_{q+1}^{\mathcal{V}}(\widetilde{X})\cap (\bar{B} \setminus (A \cap B)_{\bar{B}})$. Again, $\tau_{\bar{B}} \in \bar{B}$ denotes the copy of $\tau \in B$ in $\bar{B}$.
		 Now,

		\begin{equation*}
			\begin{aligned}
				\partial_{q+1}^{\mathcal{V}}(\tau_{\bar{B}}) &=\sum_{\sigma \in \Crit_{q}^{\mathcal{V}}(\widetilde{X})}\left(\sum_{P \in \Gamma(\tau_{\bar{B}}, \sigma)}w(P)\right)\sigma.
			\end{aligned}
		\end{equation*}

		Now we show that only the following two types of trajectories are possible from $\tau_{\bar{B}}$ that ends at some $\sigma \in \Crit_{q}^{\mathcal{V}}(\widetilde{X})$, depending on the facets of $\tau_{\bar{B}}$. We note that any facet of $\tau_{\bar{B}}$, must always be in $\bar{B}$.\\

		\emph{Type I:} $P: \tau_{\bar{B}}, \sigma_{\bar{B}}, \dots , t(P)$, where $\sigma_{\bar{B}} \in (A \cap B)_{\bar{B}}$ and the terminal simplex $t(P) \in \Crit_{q}^{\mathcal{V}}(\widetilde{X})$.
		Moreover, we show that, for $\sigma_{\bar{B}} \in (A \cap B)_{\bar{B}}$, any trajectory $P$ of the above form, i.e., $P: \tau_{\bar{B}}, \sigma_{\bar{B}}, \dots , t(P)$, where $t(P) \in \Crit_{q}^{\mathcal{V}}(\widetilde{X})$ is unique, and $t(P)= \sigma_{\bar{A}} \in (A \cap B)_{\bar{A}}$. We further show that the weight of this trajectory $P$ is $w(P)= \langle \tau_{\bar{B}}, \sigma_{\bar{B}}\rangle = \langle \tau, \sigma\rangle$. We begin by trying to find all such trajectories. Let $P: (\tau_{\bar{B}}=)~\tau_0^{(q+1)}, (\sigma_{\bar{B}}=)~\sigma_1^{(q)}, \tau_1^{(q+1)} \dots \sigma_k^{(q)}, \tau_k^{(q+1)}, \sigma_{k+1}^{(q)}~(=t(P))$.
		
 		 Let $\sigma_{\bar{B}}=[\bar{b}_{i_0}, \dots ,\bar{b}_{i_{q}}]$. According to our construction of $\mathcal{V}$, $\tau_1= [\bar{a}_{i_0}, \bar{b}_{i_0}, \dots ,\bar{b}_{i_{q}}]$. We observe that, none of the facets of $\tau_1$ can be critical so $\sigma_2$ must be paired with a higher dimensional simplex for the trajectory to continue. Therefore, $\sigma_2 \in B_{\alpha}$ for some $ \alpha \in \overline{A \cap B}$. Since, $\sigma_1 \ne \sigma_2$, therefore, $\sigma_2=[\bar{a}_{i_0}, \bar{b}_{i_1}, \dots ,\bar{b}_{i_{q}}]$, which makes $\tau_2=[\bar{a}_{i_0}, \bar{a}_{i_1}, \bar{b}_{i_1}, \dots ,\bar{b}_{i_{q}}]$. We follow the same argument to continue the trajectory. Thus, we have $\sigma_j=[\bar{a}_{i_0}, \dots,\bar{a}_{i_{j-2}},\bar{b}_{i_{j-1}}, \dots ,\bar{b}_{i_{q}}], \tau_j=[\bar{a}_{i_0}, \dots, \bar{a}_{i_{j-2}}, \bar{a}_{i_{j-1}},\bar{b}_{i_{j-1}}, \dots ,\bar{b}_{i_{q}}], \sigma_{j+1}=[\bar{a}_{i_0}, \dots ,\bar{a}_{i_{j-1}},\bar{b}_{i_j}, \dots,\bar{b}_{i_{q}}]$. We continue to reach $\tau_{q+1}=[\bar{a}_{i_0}, \dots ,\bar{a}_{i_q}, \bar{b}_{i_{q}}]$. We note here that there is no facet of $\tau_{q+1}$, distinct from $\sigma_{q+1}$ which is paired with a higher dimensional simplex. So the trajectory must end at $\sigma_{q+2}=[\bar{a}_{i_0}, \dots ,\bar{a}_{i_q}]=\sigma_{\bar{A}}$, which is critical. While continuing the trajectory, at every step, we have made sure that there is no other possible option for the next simplex, which makes $P$ unique. Moreover, the weight of $P$ is given by,
		  $$w(P)= \langle \tau_{\bar{B}}, \sigma_{\bar{B}}\rangle \prod_{j=1}^{q}\left(-\langle \tau_j, \sigma_j\rangle \langle \tau_j, \sigma_{j+1} \rangle \right) =\langle \tau_{\bar{B}}, \sigma_{\bar{B}}\rangle \prod_{j=1} ^{q}(-(-1)^{j-1}(-1)^j)=\langle \tau_{\bar{B}}, \sigma_{\bar{B}}\rangle = \langle \tau, \sigma \rangle$$.
		 
		  \emph{Type II:} $P: \tau_{\bar{B}}, \sigma_{\bar{B}}$, where $\sigma_{\bar{B}}$ is a facet of $\tau_{\bar{B}}$ such that $\sigma_{\bar{B}} \in \bar{B} \setminus (A \cap B)_{\bar{B}}$ and is hence critical.
		 
		 We note here that if $\sigma_{\bar{B}}$ is a facet of $\tau_{\bar{B}}$ such that $\sigma_{\bar{B}} \in \bar{B} \setminus (A \cap B)_{\bar{B}}$, then the only possible trajectory from $\tau_{\bar{B}}$ to $\sigma_{\bar{B}}$ is of the above form. This is because if there exists any other such trajectory $P'$, then $P'$ must traverse through a regular facet of $\tau_{\bar{B}}$ (we recall from Section~\ref{prelim} that a simplex which has appeared in $\mathcal{V}$ is termed as regular), which must be in $(A \cap B)_{\bar{B}}$. Then $P'$ must be a trajectory of Type I and hence as we have proved in the case of Type I, the terminal simplex cannot be in $\bar{B}$. Therefore, $w(P)= \langle \tau_{\bar{B}}, \sigma_{\bar{B}}\rangle = \langle \tau, \sigma\rangle$.
		 
		  We can now write,
		 
		  \begin{equation*}
		  	\begin{aligned}
		  	\partial_{q+1}^{\mathcal{V}}(\tau_{\bar{B}})&= \sum _{\sigma_{\bar{A}}^{(q)} \in (A \cap B)_{\bar{A}}}\langle \tau, \sigma \rangle\sigma_{\bar{A}} + \sum_{\sigma_{\bar{B}}^{(q)} \in  (\overline{B} \setminus (A \cap B)_{\bar{B}})}\langle \tau, \sigma\rangle\sigma_{\bar{B}}\\
			 \implies g_{q}(\partial_{q+1}^{\mathcal{V}}(\tau_{\bar{B}}))&= \sum _{\sigma_{\bar{A}}^{(q)} \in (A \cap B)_{\bar{A}}}\langle \tau, \sigma \rangle g_{q}(\sigma_{\bar{A}}) + \sum_{\sigma_{\bar{B}}^{(q)} \in  (\overline{B} \setminus (A \cap B)_{\bar{B}})}\langle \tau, \sigma\rangle g_{q}(\sigma_{\bar{B}}) \\
			&= \sum_{\sigma \in S_{q}(X) \cap (A \cap B)} \langle \tau, \sigma\rangle \sigma + \sum_{\sigma \in S_{q}(X) \setminus A}\langle \tau, \sigma\rangle \sigma, \hspace{0.5em} (\text{ since, $g_q$ is an isomorphism} )\\
		   &= \sum_{\sigma \in S_{q}(B)}\langle \tau, \sigma \rangle \sigma \\ &= \partial_{q+1}(g_{q+1}(\tau_{\bar{B}})).
			\end{aligned}
		\end{equation*}
	
	\noindent This completes the proof of Theorem ~\ref{iso}.
				\end{proof}
				
				\noindent So, we have successfully arrived at the following result.
				
			\begin{thm}\label{homeq}
				The simplicial homology of $X$ is isomorphic to the simplicial homology of $\widetilde{X}$, i.e., $H_{\#}(X) \cong H_{\#}(\widetilde{X}, \mathbb{Z})$.
			\end{thm}
			
					
				\begin{proof}
					
					The result follows from Theorem~\ref{iso} and Theorem~\ref{hom}.
				\end{proof}

			\noindent In the following subsections, we use the gradient vector fields $\mathcal{W}_{\bar{A}}$, $\mathcal{W}_{\bar{B}}$ and $\mathcal{W}_{\overline{A \cap B}}$ on $\bar{A}$, $\bar{B}$, and $\overline{A \cap B}$ respectively, to construct a new gradient vector field $\mathcal{W}$ on $\widetilde{X}$. We show that the Thom-Smale complex of $\widetilde{X}$ with respect to $\mathcal{W}$ is isomorphic to $\mathcal{D}_{\#}(X)$, which we constructed in Section~\ref{intro}. This result, together with, Theorem ~\ref{homeq}, proves our main theorem.
			

				\subsection{Construction of the gradient vector field $\mathcal{W}$ on $\widetilde{X}$ using $\mathcal{W}_{\bar{A}}$, $\mathcal{W}_{\bar{B}}$ and $\mathcal{W}_{\overline{A \cap B}}$} \label{proof}
				\noindent We now construct a new discrete vector field $\mathcal{W}$ on $\widetilde{X}$ using $\mathcal{W}_{\bar{A}}$, $\mathcal{W}_{\bar{B}}$ and $\mathcal{W}_{\overline{A \cap B}}$. 
				
				Before we move into the construction, let us briefly recall $S_{\alpha}$ which we introduced in Section~\ref{prelim}. For any $\alpha=[\bar{c}_{i_0}, \dots ,\bar{c}_{i_q}] \in \overline{A \cap B}$, we have defined $A_{\alpha}=\{[\bar{a}_{i_0}, \dots , \bar{a}_{i_r}, \bar{b}_{i_r}, \dots ,b_{i_q}] \mid 0 \leq r \leq q\}$, $B_{\alpha}= \{[\bar{a}_{i_0}, \dots , \bar{a}_{i_{r-1}}, \bar{b}_{i_r}, \dots ,b_{i_q}] \mid 1 \leq r \leq q\}$ and $S_{\alpha}=A_{\alpha} \cup B_{\alpha}$. Hence we observe that $\mathbb{P}(\overline{A \cap B})= \bigcup_{\alpha \in \overline{ A \cap B}}S_{\alpha}$, where $S_{\alpha}$-s are mutually disjoint.
				
				Let us denote the collection of the \emph{interior simplices} of $\mathbb{P}(\overline{A \cap B})$, i.e., $\mathbb{P}(\overline{A \cap B}) \setminus ((A \cap B)_{\bar{A}} \cup (A \cap B)_{\bar{B}})$ as $\mathbb{P}_{int}(\overline{A \cap B})$.

			 \begin{obs}
			 	Next, we list some simple observations which follow directly from the construction of $\widetilde{X}$.
				\begin{enumerate}[label=(\alph*)]
					\item Let $\sigma^{(q-1)} \subseteq \tau^{(q)}$ such that $\tau \in \mathbb{P}_{int}(\overline{A \cap B})$, $\sigma \in \bar{A}$ (or $\sigma \in \bar{B}$), then $\sigma \in (A \cap B)_{\bar{A}}$ (or $ \sigma \in (A \cap B)_{\bar{B}}$) and $\GS(\sigma)=\GS(\tau)$. Conversely, if $\sigma^{(q-1)} \subseteq \tau^{(q)}$, such that $\sigma \in (A \cap B)_{\bar{A}}$ (or $\sigma \in (A \cap B)_{\bar{B}}$), $\tau \in \mathbb{P}(\overline{A \cap B})$  and $\GS(\sigma)=\GS(\tau)$, then $\tau \in \mathbb{P}_{int}(\overline{A \cap B})$.\label{obs2a}
					\item Let $\sigma \in (A \cap B)_{\bar{A}}$ ( or $(A \cap B)_{\bar{B}}$) so that $\sigma= \alpha_{\bar{A}}$ ( or $\alpha_{\bar{B}}$) for some $\alpha \in A$ (or $B$). Then $\GS(\sigma)= \alpha_{\overline{A \cap B}}$. \label{obs2b}
				\item	Let $\sigma^{(q-1)} \subseteq \tau^{(q)}$, where $\tau \in A_{\beta}$ for some $\beta \in \overline{A \cap B}$. Then, $\GS(\tau) = \GS(\sigma)$, iff $\sigma \in B_{\beta}$. Also,  $\GS(\sigma)\ne \GS(\tau)$ iff $\sigma \in A_{\alpha}$ for some $\alpha \in \overline{A \cap B}$. 	\label{obsc}
					
				\item  Let $\sigma^{(q-1)} \subseteq\tau^{(q)}$ such that $\GS(\sigma) \neq \GS(\tau)= \beta$ and $\tau \in A_{\beta}$. Then  $(\GS(\sigma))^{(q-2)} \subseteq (\GS(\tau))^{(q-1)}$. \label{obsb}
				\end{enumerate}
			\end{obs}

				For each $\alpha \in \overline{A \cap B}$, we denote the collection of interior simplices of $\mathbb{P}(A \cap B)$ in $S_{\alpha}$, i.e., $S_{\alpha} \cap \mathbb{P}_{int}(\overline{A \cap B})$ as $S_{\alpha}'$, and the collection of interior simplices of $\mathbb{P}(A \cap B)$ in $B_{\alpha}$, i.e., $B_{\alpha} \cap \mathbb{P}_{int}(\overline{A \cap B})$ as $B_{\alpha}'$.
				
				The construction of $\mathcal{W}$ goes as follows. The simplices in $\bar{A}$ and $\bar{B}$ are paired using $\mathcal{W}_{\bar{A}}$ and $\mathcal{W}_{\bar{B}}$ respectively. Next, for each pair in $\mathcal{W}_{\overline{A\cap B}}$, we define a collection of pairs on $\mathbb{P}_{int}(\overline{A \cap B})$, which we denote as $\mathcal{W}'$. The discrete vector field $\mathcal{W}$ on $\widetilde{X}$ is given by $\mathcal{W}_{\bar{A}} \cup \mathcal{W}_{\bar{B}} \cup \mathcal{W'}$. 
				
				\noindent Now we define $\mathcal{W'}$.

	 Suppose, $(\alpha^{(q-1)}, \beta^{(q)}) \in \mathcal{W}_{\overline{A \cap B}}$. Let $\beta=[\bar{c}_{i_0}, \dots ,\bar{c}_{i_q}]$ and $\alpha= [\bar{c}_{i_0}, \dots, \widehat{\bar{c}_{i_j}}, \dots \bar{c}_{i_q}], j \in \{0, \dots ,q\}$. \\
	
	\noindent \textbf{Case I:} Let $j=0$.
	\begin{enumerate}
	\item $([\bar{a}_{i_1}, \bar{b}_{i_1}, \dots , \bar{b}_{i_q}], [\bar{a}_{i_0}, \bar{a}_{i_1}, \bar{b}_{i_1}, \dots , \bar{b}_{i_q}]) \in \mathcal{W'}$, where $[\bar{a}_{i_0}, \bar{a}_{i_1}, \bar{b}_{i_1}, \dots , \bar{b}_{i_q}] \in S_{\beta}'$, $[\bar{a}_{i_1}, \bar{b}_{i_1}, \dots , \bar{b}_{i_q}] \in S_{\alpha}'$.
	
	\item  (Matching the simplices in $S_{\beta}'$) 
	\begin{enumerate}
		\item $([\bar{a}_{i_0}, \bar{b}_{i_1}, \dots , \bar{b}_{i_q}], [\bar{a}_{i_0},\bar{b}_{i_0}, \bar{b}_{i_1},\dots , \bar{b}_{i_q}]) \in \mathcal{W}'$.
		\item $([\bar{a}_{i_0}, \dots ,\bar{a}_{i_r}, \bar{b}_{i_{r+1}}, \dots \bar{b}_{i_q}], [\bar{a}_{i_0}, \dots ,\bar{a}_{i_r}, \bar{a}_{i_{r+1}}, \bar{b}_{i_{r+1}}, \dots \bar{b}_{i_q}]) \in \mathcal{W}'$ for $1 \leq r \leq (q-1)$.
	\end{enumerate}
	
	\item (Matching the simplices in $S_{\alpha}'$)  
	 $( [\bar{a}_{i_1}, \dots ,\bar{a}_{i_r}, \bar{b}_{i_{r+1}}, \dots ,\bar{b}_{i_q}],[\bar{a}_{i_1}, \dots ,\bar{a}_{i_r},\bar{a}_{i_{r+1}}, \bar{b}_{i_{r+1}}, \dots ,\bar{b}_{i_q}]) \in \mathcal{W}'$, for $1 \leq r \leq (q-1)$.
	
	\end{enumerate}

\noindent \textbf{Case II:} Let $j >0$.

\begin{enumerate}
	\item $( [\bar{a}_{i_0}, \bar{b}_{i_0}, \dots \widehat{\bar{b}_{i_j}}, \dots \bar{b}_{i_q}], [\bar{a}_{i_0},\bar{b}_{i_0}, \dots , \bar{b}_{i_q}]) \in \mathcal{W'}$, where  $[\bar{a}_{i_0}, \bar{b}_{i_0}, \dots \widehat{\bar{b}_{i_j}}, \dots \bar{b}_{i_q}] \in S_{\alpha}'$, $[\bar{a}_{i_0},\bar{b}_{i_0}, \dots , \bar{b}_{i_q}] \in S_{\beta}'$.
	
	\item  (Matching the simplices in $S_{\beta}'$)
	
	$( [\bar{a}_{i_0}, \dots ,\bar{a}_{i_r}, \bar{b}_{i_{r+1}}, \dots \bar{b}_{i_q}], [\bar{a}_{i_0}, \dots ,\bar{a}_{i_r},\bar{a}_{i_{r+1}}, \bar{b}_{i_{r+1}}, \dots \bar{b}_{i_q}]) \in \mathcal{W'}$, for $0 \leq r \leq (q-1)$.
	
	\item (Matching the simplices in $S_{\alpha}'$)
	
	We relabel $\alpha$ as $[\bar{c}_{l_0}, \bar{c}_{l_1}, \dots ,\bar{c}_{l_{q-1}}]$, where, $\{l_0, \dots ,l_{q-1}\} = \{i_0, \dots ,i_q\} \setminus \{i_j\}$, $l_0 < l_1 \dots < l_{q-1}$.
	
	 $( [\bar{a}_{l_0}, \dots ,\bar{a}_{l_r}, \bar{b}_{l_{r+1}}, \dots \bar{b}_{l_{q-1}}], [\bar{a}_{l_0}, \dots ,\bar{a}_{l_r},\bar{a}_{l_{r+1}}, \bar{b}_{l_{r+1}}, \dots \bar{b}_{l_{q-1}}]) \in \mathcal{W'}$ for $0 \leq r \leq (q-2)$.

\end{enumerate}
	
	So, we see that, if $(\alpha^{(q-1)}, \beta^{(q)}) \in \mathcal{W}_{\overline{A \cap B}}$, then all the simplices in $S_{\alpha}' \cup S_{\beta}'$ are paired.  Next, we pair the simplices of  $S_{\gamma}'$ for each $\gamma \in$ $\Crit^{\mathcal{W}_{\overline{A \cap B}}}(\overline{A \cap B})$.

	 Let $\gamma=[\bar{c}_{i_0}, \bar{c}_{i_1}, \dots ,\bar{c}_{i_q}] \in \Crit_q^{\mathcal{W}_{\overline{A \cap B}}}(\overline{A \cap B})$. Then, $$([\bar{a}_{i_0}, \dots ,\bar{a}_{i_{r-1}}, \bar{b}_{i_r}, \dots \bar{b}_{i_q}], [\bar{a}_{i_0}, \dots ,\bar{a}_{i_{r-1}}, \bar{a}_{i_r}, \bar{b}_{i_r}, \dots \bar{b}_{i_q}]) \in \mathcal{W'} \text{, for } 1 \leq r \leq q.$$
	
	 Therefore, the only simplices that remain unpaired with respect to $\mathcal{W'}$ are $\{\sigma =[\bar{a}_{i_0}, \bar{b}_{i_0}, \dots \bar{b}_{i_q}]\in A_{\gamma} \mid \gamma \in \Crit^{\mathcal{W}_{\overline{A \cap B}}}(\overline{A \cap B}) \}$. This completes the construction of $\mathcal{W}'$. The well-definedness of $\mathcal{W}'$ follows from the fact that $S_{\alpha}'$-s are pairwise disjoint.

	 \begin{obs}
	 	Now we list some important observations which we will use subsequently. These observations follow directly from the construction of $\mathcal{W}$.
	 	
	 	\begin{enumerate}[label=(\alph*)]
	 		\item 	Let $(\sigma^{(q-1)}, \tau^{(q)}) \in \mathcal{W'}$, then, \label{obsa}
	 		\begin{enumerate}[label=(\roman*)]
	 			\item If $\GS(\sigma) = \GS(\tau) = \alpha$, then $\sigma \in B_{\alpha}$, $\tau \in A_{\alpha}$. Conversely, if $\sigma \in B_{\alpha}$, for some $\alpha \in \overline{A \cap B}$, then $\GS(\sigma) = \GS(\tau)$ and $\tau \in A_{\alpha}$. In this case, $(\sigma, \tau)$ is either of the form $([\bar{a}_{i_0}, \bar{b}_{i_1}, \dots ,\bar{b}_{i_q}], [\bar{a}_{i_0}, \bar{b}_{i_0}, \bar{b}_{i_1}, \dots ,\bar{b}_{i_q}])$ or of the form \\  $([\bar{a}_{i_0}, \dots, \bar{a}_{i_r}, \bar{b}_{i_{r+1}}, \dots ,\bar{b}_{i_q}], [\bar{a}_{i_0}, \dots, \bar{a}_{i_r}, \bar{a}_{i_{r+1}}, \bar{b}_{i_{r+1}}, \dots ,\bar{b}_{i_q}])$, $0 \leq r \leq (q-1)$.

	 			\item If $\alpha= \GS(\sigma) \ne \GS(\tau) = \beta$, then $\sigma \in A_{\alpha}$, $\tau \in A_{\beta}$. Moreover, $\sigma$ is of the form $[\bar{a}_{i_0},\bar{b}_{i_0}, \dots ,\bar{b}_{i_{q-2}}]$. Conversely, if $\sigma \in A_{\alpha}$ for some $\alpha \in \overline{A \cap B}$, then $\GS(\sigma) \ne \GS(\tau)$, and $\tau \in A_{\beta}$ where $\beta= \GS(\tau)$. In this case, $\tau$ is either of the form $[\bar{a}_{i_0}, \bar{b}_{i_0}, \dots ,\bar{b}_{i_{q-1}}]$ or of the form $[\bar{a}_{i_0}, \bar{a}_{i_1}, \bar{b}_{i_1}, \dots ,\bar{b}_{i_{q-1}}]$.
	 		\end{enumerate}

	  	\item Let $(\sigma^{(q-1)}, \tau^{(q)}) \in \mathcal{W}'$ such that, $\GS(\sigma) \neq \GS(\tau)$. Then $(\GS(\sigma), \GS(\tau)) \in \mathcal{W}_{\overline{A \cap B}}$. Conversely, if $(\alpha, \beta) \in \mathcal{W}_{\overline{A \cap B}}$, then, there exists a unique pair $ (\sigma^{(q-1)}, \tau^{(q)}) \in \mathcal{W'}$ such that $\GS(\sigma)= \alpha$, $\GS(\tau)=\beta$. \label{obse}

	  \item	For each $\gamma \in \Crit^{\mathcal{W}_{\overline{A \cap B}}}(\overline{A \cap B})$, there exists a unique simplex $\sigma$ in $S'_{\gamma}$, which is unpaired with respect  to $\mathcal{W}$. Furthermore, $\sigma \in A_{\gamma}$. Conversely, if $\sigma \in \mathbb{P}_{int}(\overline{A \cap B})$ and is unpaired with respect to $\mathcal{W}$, then $\GS(\sigma) \in \Crit^{\mathcal{W}_{\overline{A \cap B}}}(\overline{A \cap B})$.\label{obsd}

	 		\end{enumerate}
	 \end{obs}

	 \begin{prop}{(Types of $\mathcal{W}$-trajectories)} \label{prop}
	 	Let $P: \tau_0^{(q)}, \sigma_1^{(q-1)}, \tau_1^{(q)}, \dots , \sigma_k^{(q-1)}, \tau_k^{(q)}, \sigma_{k+1}^{(q-1)}$ be a  $\mathcal{W}$-trajectory. Then $P$ must be of either of the following four forms.
	 	
	 	\begin{enumerate}
	 	\item $\tau_0 \in \bar{A}$. In this case, $(\sigma_i, \tau_i) \in \mathcal{W}_{\bar{A}}$ for each $i \in [k]$ and $\sigma_{k+1} \in \bar{A}$.
	 	
	 	\item   $\tau_0 \in \bar{B}$. In this case, $(\sigma_i, \tau_i) \in \mathcal{W}_{\bar{B}}$ for each $i \in [k]$ and $\sigma_{k+1} \in \bar{B}$.
	 	
	 	\item $\tau_0 \in \mathbb{P}_{int}(\overline{A \cap B})$, and $(\sigma_i, \tau_i) \in \mathcal{W}'$ for each $i \in [k]$, $\sigma_{k+1} \in \mathbb{P}_{int}(\overline{A \cap B})$.
	 		
	 	\item $\tau_0 \in \mathbb{P}_{int}(\overline{A \cap B})$ such that, there exists $1 \leq r \leq (k+1)$ for which, $\sigma_i \in \bar{A}$ for all $r \leq i \leq (k+1)$ and $\tau_i \in \mathbb{P}_{int}(\overline{A \cap B})$ for all $0  \leq i \leq r-1$.
	 		
	 	\item $\tau_0 \in \mathbb{P}_{int}(\overline{A \cap B})$ such that, there exists $1 \leq r \leq (k+1)$ for which, $\sigma_i \in \bar{B}$ for all $r \leq i \leq (k+1)$ and $\tau_i \in \mathbb{P}_{int}(\overline{A \cap B})$ for all $0 \leq i \leq r-1$.
	 
	 	\end{enumerate}
	 \end{prop}
	 
	 \begin{proof}
	 	Let $\tau_0 \in \bar{A}$. In this case, every facet of $\tau_0$ is in $\bar{A}$. Therefore, $(\sigma_1, \tau_1) \in \mathcal{W}_{\bar{A}}$. Following the same argument, it can be observed that $(\sigma_i, \tau_i) \in \mathcal{W}_{\bar{A}}$ for each $i \in [k]$, and $\sigma_{k+1} \in \bar{A}$. A similar argument works if $\tau_0 \in \bar{B}$. 
	 	
	 	Let $\tau_0 \in \mathbb{P}_{int}(\overline{A \cap B})$. If $(\sigma_i, \tau_i) \in \mathcal{W}'$ for each $i \in [k]$, and $\sigma_{k+1} \in \mathbb{P}_{int}(\overline{A \cap B})$, then $P$ is of the third type of trajectory. Otherwise, there exists $j_0 \in [k]$, such that $\sigma_{j_0} \in \bar{A}$ or $\bar{B}$ and $\tau_i \in \mathbb{P}_{int}(\overline{A \cap B})$ for all $i < j_0$. If, $j_0=k+1$, then we are done. If $j_0 < k+1$, and $\sigma_{j_0} \in \bar{A}$, then, following the same argument as for the first type of trajectory, we can observe that $(\sigma_j, \tau_j) \in \mathcal{W}_{\bar{A}}$, for all $j_0 \leq j \leq k $, and $\sigma_{k+1} \in \bar{A}$. Similarly, if $j_0 < k+1$, and $\sigma_{j_0} \in \bar{B}$, then $(\sigma_j, \tau_j) \in \mathcal{W}_{\bar{B}}$, for all $j_0 \leq j \leq k $, and $\sigma_{k+1} \in \bar{B}$. These are precisely the fourth and fifth type of trajectories respectively.
	 \end{proof}
	 
	 Henceforth, we refer to these five types trajectories as trajectories of the first, second, third, fourth and fifth type respectively.
	 
	\begin{prop}
		The discrete vector field $\mathcal{W}$ on $\widetilde{X}$ is acyclic. In other words, $\mathcal{W}$ is a gradient vector field.
	\end{prop}
	
	\begin{proof}
		Let, if possible, $P:\tau_0^{(q)}, \sigma_1^{(q-1)}, \tau_1^{(q)}, \dots, \sigma_{k}^{(q-1)}, \tau_{k}^{(q)}(=\tau_0^{(q)} $) be a closed $\mathcal{W}$-trajectory in $\widetilde{X}$. Without loss of generality, we may assume that $\sigma_i \ne \sigma_j$ for each $i \ne j$, $i, j \in [k]$ and $\tau_i \ne \tau_j$, for each $i \ne j$, $\{i,j\} \ne \{0,k\}$.\\
		
		 It follows from \autoref{prop} that $(\sigma_i, \tau_i) \in \mathcal{W'}$ for all $i \in [k]$. 
		
		 Suppose, the ground simplex remains same for each simplex in $P$, i.e., $\GS(\sigma_i)=\GS(\tau_i)= \beta$ for all $i \in [k]$.
		
		 It follows from Observation 2\ref{obsa} that $\tau_0 \in A_{\beta}$ for some $\beta \in \overline{A \cap B}$. From Observation 1\ref{obsc}, it follows that $\sigma_1 \in B_{\beta}$, i.e., $\sigma_1$ is of the form $[\bar{a}_{i_0}, \dots , \bar{a}_{i_r}, \bar{b}_{i_{r+1}}, \dots ,\bar{b}_{i_{q-1}}]$, $0 \leq r \leq (q-2)$. From Observation 2\ref{obsa}, it follows that if $r > 0$, then $\tau_1=[\bar{a}_{i_0}, \dots , \bar{a}_{i_r}, \bar{a}_{i_{r+1}}, \bar{b}_{i_{r+1}}, \dots ,\bar{b}_{i_{q-1}}]$. If $r=0$, then $\tau_1$ is either $[\bar{a}_{i_0}, \bar{b}_{i_0}, \dots ,\bar{b}_{i_{q-1}}]$ or $[\bar{a}_{i_0}, \bar{a}_{i_1},\bar{b}_{i_1}, \dots ,\bar{b}_{i_{q-1}}]$. In the first case, we use Observation 1\ref{obsc} to deduce that $\sigma_2= [\bar{a}_{i_0}, \dots , \bar{a}_{i_{r+1}}, \bar{b}_{i_{r+2}}, \dots ,\bar{b}_{i_{q-1}}]$. Here, we observe that $(r+1) \geq 1$. So, from Observation 2\ref{obsa}, $\tau_2=  [\bar{a}_{i_0}, \dots , \bar{a}_{i_{r+1}}, \bar{a}_{i_{r+2}}, \bar{b}_{i_{r+2}}, \dots ,\bar{b}_{i_{q-1}}]$. We continue in this way. Thus, at each step, we delete a vertex $\bar{b}_{i_j}$ from a higher dimensional simplex $\tau_i$ and add a new vertex $\bar{a}_{i_{j+1}}$ (which has not appeared before), to the next lower dimensional simplex $\sigma_{i+1}$. Now, for $r=0$, if $\tau_1= [\bar{a}_{i_0}, \bar{a}_{i_1},\bar{b}_{i_1}, \dots ,\bar{b}_{i_{q-1}}]$, then we follow the same argument as in the previous case and show that this is not possible. So, let $\tau_1= [\bar{a}_{i_0}, \bar{b}_{i_0}, \dots ,\bar{b}_{i_{q-1}}]$. Here we observe that there is no facet of $\tau_1$ in $\mathbb{P}_{int}(\overline{A \cap B})$ distinct from $\sigma_1$, and with the same ground simplex as $\tau_1$. So, the trajectory cannot continue. Therefore, it is not possible for any simplex to recur in $P$. So $P$ cannot be closed.
		
		Therefore, there exists $\tau_i$ or $\sigma_i$, whose ground simplex is different from those of its preceding simplices, i.e.,
		$$ \GS(\tau_0) = \GS(\sigma_1) = \dots = \GS(\tau_{i-1})= \GS(\sigma_i) \ne \GS(\tau_i),$$
		\begin{center}
			or
		\end{center}
		$$ \GS(\tau_0) = \GS(\sigma_1) = \dots = \GS(\tau_{i-1})\neq \GS(\sigma_i).$$
		
		In the first case, it follows from Observation 2\ref{obsa} that $\tau_{i_1 -1} \in A_{\alpha}$ for some $\alpha \in \overline{A \cap B}$ and from Observation 1\ref{obsc}, we know that $\sigma_{i_1} \in B_{\alpha}$. We again use Observation 2\ref{obsa}, to deduce that $\GS(\tau_{i_1})= \GS(\sigma_{i_1})$, which is a contradiction. 
		
		In the second case, the ground simplex differs for the first time at a lower dimensional simplex $\sigma_i$. Let us rename it as $\sigma_{j_1}$. Here, the ground simplex remains same till $\tau_{j_1 -1}$. So, from Observation 2\ref{obsa}, $\tau_{j_1-1} \in A_{\alpha}$ for some $\alpha \in \overline{A \cap B}$. Then, we deduce from Observation 1\ref{obsc} that $\sigma_{j_1} \in A_{\beta}$ for some $\beta \in \overline{ A \cap B}$ and from Observation 2\ref{obsa}, it follows that $\GS(\sigma_{j_1}) \ne \GS(\tau_{j_1})$. If the ground simplex changes again in $P$, then we follow the same argument as before, and show that the ground simplex changes next at a lower dimensional simplex,  $\sigma_{j_2}$ and $\GS(\sigma_{j_2}) \ne \GS(\tau_{j_2})$. We continue in this way. Suppose, $j_m$ is the maximum in $[k]$, for which the ground simplex differs at $\tau_{j_m}$, i.e., $\GS(\sigma_{j_m}) \neq \GS(\tau_{j_m}) = \GS(\sigma_{j_m +1}) = \dots = \GS(\tau_k)=\GS(\tau_0)$. Then, corresponding to $P$, we obtain a sequence of simplices in $\overline{A \cap B}$ as follows.
		
		$$ \GS(\tau_0), \GS(\sigma_{j_1}), \GS(\tau_{j_1}), \dots ,\GS(\sigma_{j_m}), \GS(\tau_{j_m})~(=\GS(\tau_0)).$$
		
		We note that $\GS(\tau_0)= \GS(\tau_{j_1-1}) \ne \GS(\sigma_{j_1})$ and for each $1 \leq i \leq (m-1)$, $\GS(\tau_{j_i})= \GS(\tau_{j_{i+1}-1}) \ne \GS(\sigma_{j_{i+1}})$. Therefore, it follows from Observation 1\ref{obsb} that $\GS(\sigma_{j_1})^{(q-2)} \subseteq \GS(\tau_0)^{(q-1)}$ and $\GS(\sigma_{j_{i+1}})^{(q-2)} \\ \subseteq \GS(\tau_{j_i})^{(q-1)}$ for each $1 \leq i \leq (m-1)$. Also from Observation 2\ref{obse} it follows that, $(\GS(\sigma_{j_i}), \GS(\tau_{j_i})) \in \mathcal{W}_{\overline{A \cap B}}$. Now, if $(\GS(\sigma_{j_i}), \GS(\tau_{j_{i-1}})) \in \mathcal{W}_{\overline{A \cap B}}$ for some  $i \in \{ 2, \dots ,(m-1)\}$, then $\GS(\sigma_{j_{i-1}})= \GS(\sigma_{j_i})$. Then we again use Observation 2\ref{obse} to deduce that $\sigma_{j_i}= \sigma_{j_{i-1}}$ which contradicts our assumption that the simplices in $P$, except the initial and terminal simplex, are distinct. Similarly, if $(\GS(\tau_0),\GS(\sigma_{j_1})) \in \mathcal{W}_{\overline{A \cap B}}$, then $\GS(\tau_0)=\GS(\tau_{j_1})$. Again, from Observation 2\ref{obse}, $\tau_0 = \tau_{j_1}$ which,as before, contradicts our assumption. Therefore, $(\GS(\sigma_{j_i}), \GS(\tau_{j_{i-1}})) \notin \mathcal{W}_{\overline{A \cap B}}$ for all $0 \leq i \leq (m-1)$ and $(\GS(\tau_0),\GS(\sigma_{j_1})) \notin \mathcal{W}_{\overline{A \cap B}}$. 
		
		Hence this is a closed $\mathcal{W}_{\overline{A \cap B}}$-trajectory. This is a contradiction, as $\mathcal{W}_{\overline{A \cap B}}$ is a gradient vector field. Therefore, such a closed $\mathcal{W}$-trajectory is not possible.


	\end{proof}
	 		
	Now that $\mathcal{W}$ is a gradient vector field, the critical simplices of $\mathcal{W}$ of dimension $q$ can be listed as follows.
				
		\begin{enumerate}
			\item $\Crit_q^{\mathcal{W}_{\bar{A}}}(\bar{A})$,
			\item  $\Crit_q^{\mathcal{W}_{\bar{B}}}(\bar{B})$,
			\item  $\{ \alpha=[\bar{a}_{i_0}, \bar{b}_{i_0}, \dots ,\bar{b}_{i_{q-1}}] \in A_{ \gamma} \mid \gamma=[\bar{c}_{i_0}, \dots ,\bar{c}_{i_{q-1}}] \in \Crit_{q-1}^{\mathcal{W}_{\overline{A \cap B}}}(\overline{A \cap B})\}$.
		\end{enumerate}
			
			Let us denote these three sets of critical simplices as $C^{\bar{A}}_q$, $C_q^{\bar{B}}$ and $C_q^{int}$ respectively.
	\begin{prop}\label{prop1}
		Let $\tau_1 \in \mathbb{P}_{int}(\overline{A \cap B})$ such that $\tau_1 \in \Crit_q^{\mathcal{W}}(\widetilde{X})$, or  $(\sigma_1^{(q-1)}, \tau_1^{(q)}) \in \mathcal{W}$ such that $\GS(\sigma_1) \neq \GS(\tau_1)$. Let $ \sigma_2 \in \mathbb{P}_{int}(\overline{A \cap B})$ such that $\sigma_2 \ne \sigma_1$ and $\sigma_2 \in \Crit_{q-1}^{\mathcal{W}}(\widetilde{X})$, or $(\sigma_2^{(q-1)}, \tau_2^{(q)}) \in \mathcal{W}$, where $\GS(\sigma_2) \ne \GS(\tau_2)$. Further, suppose that $\GS(\sigma_2)$ is a facet of $\GS(\tau_1)$. Then there exists a unique $\mathcal{W}$-trajectory $P$ from $\tau_1$ to $\sigma_2$ with the property that the ground simplex remains same for each simplex in $P$ except $\sigma_2$. (There may be more than one trajectories from $\tau_1$ to $\sigma_2$ but if there exists a trajectory $P'$ from $\tau_1$ to $\sigma_2$ such that the ground simplex remains same throughout $P'$ except $\sigma_2$, then $P'=P$.)
	\end{prop}
	
	\begin{proof}
	  Let, if possible, $P: \tau_1^{(q)}, \widetilde{\sigma_1}^{(q-1)}, \widetilde{\tau_1}^{(q)}, \dots , \widetilde{\sigma_k}^{(q-1)}, \widetilde{\tau_k}^{(q)}, \sigma_2^{(q-1)}$ be a $\mathcal{W}$-trajectory from $\tau_1$ to $\sigma_2$ such that $\GS(\tau_1)=\GS( \widetilde{\sigma_1})= \dots =\GS( \widetilde{\tau_k}) \ne \GS(\sigma_2)$. Since $\sigma_2 \in \mathbb{P}_{int}(\overline{A \cap B})$, therefore, from \autoref{prop}, $P$ is a $\mathcal{W}$-trajectory of the third type, i.e., $(\widetilde{\sigma_i}, \widetilde{\tau_i}) \in \mathcal{W}'$, for each $i \in [k]$. We will show that such a trajectory exists and is unique.
	 
	From Observation 2\ref{obsa}, and from the construction of $\mathcal{W}$, it follows that  $\tau_1$ is either of the two forms, $[\bar{a}_{i_0}, \bar{b}_{i_0}, \dots ,\bar{b}_{i_{q-1}}]$, $[\bar{a}_{i_0}, \bar{a}_{i_1},\bar{b}_{i_1}, \dots ,\bar{b}_{i_{q-1}}]$. Let, $\GS(\tau_1)=[\bar{c}_{i_0}, \dots ,\bar{c}_{i_{q-1}}]$. Then $\GS(\sigma_2)$ must be $[\bar{c}_{i_0}, \dots , \widehat{\bar{c}_{i_m}}, \dots ,\bar{c}_{i_{q-1}}]$, $0 \leq m \leq (q-1)$. It again follows from Observation 2\ref{obsa} and from the construction of $\mathcal{W'}$ that, $\sigma_2=[\bar{a}_{i_1},\bar{b}_{i_1}, \dots ,\bar{b}_{i_{q-1}}]$ if $m=0$, or $\sigma_2= [\bar{a}_{i_0}, \bar{b}_{i_0}, \dots \widehat{\bar{b}_{i_m}}, \dots ,\bar{b}_{i_{q-1}}]$, if $m >0$.
	 
	 \textbf{Case I:} Let $\tau_1=[\bar{a}_{i_0}, \bar{b}_{i_0}, \dots ,\bar{b}_{i_{q-1}}]$.
	 
	  \emph{Subcase I:} Let $m=0$. Then $\sigma_2=[\bar{a}_{i_1},\bar{b}_{i_1}, \dots ,\bar{b}_{i_{q-1}}]$. Clearly, $\widetilde{\sigma_1} \ne \sigma_2$. Therefore it follows from Observation 1\ref{obsc} that $\widetilde{\sigma_1}=[\bar{a}_{i_0}, \bar{b}_{i_1}, \dots ,\bar{b}_{i_{q-1}}]$. From the construction of $\mathcal{W}$, either $\tau_1$ is critical or, from Observation 2\ref{obsa}, $\tau_1$ is paired with $[\bar{a}_{i_0}, \bar{b}_{i_0}, \dots , \widehat{b_{i_j}}, \dots , \bar{b}_{i_{q-1}}]$. Therefore, again from the construction of $\mathcal{W}$, $\widetilde{\tau_1}=[\bar{a}_{i_0}, \bar{a}_{i_1}, \bar{b}_{i_1}, \dots ,\bar{b}_{i_{q-1}}]$. Here, if $\widetilde{\sigma_2} \ne \sigma_2$, then, again from Observation 1\ref{obsc}, $\widetilde{\sigma_2}=[\bar{a}_{i_0}, \bar{a}_{i_1}, \bar{b}_{i_2}, \dots ,\bar{b}_{i_{q-1}}]$ and hence, $\tau_2=[\bar{a}_{i_0}, \bar{a}_{i_1}, \bar{a}_{i_2}, \bar{b}_{i_2}, \dots ,\bar{b}_{i_{q-1}}]$. Henceforth, at any $j$-th step, $\widetilde{\tau_j}=[\bar{a}_{i_0}, \dots, \bar{a}_{i_j}, \bar{b}_{i_j}, \dots , \bar{b_{i_{q-1}}}]$ for $j \geq 2$. So, $\sigma_2$ cannot be a facet of any $\widetilde{\tau_j}$ for $j \geq 2$. Therefore, $\widetilde{\sigma_2}=\sigma_2$. Thus, in this case, $\tau_1, \widetilde{\sigma_1}, \widetilde{\tau_1}, \sigma_2$ is the unique trajectory from $\tau_1$ to $\sigma_2$ with the property that the ground simplex remains same for every simplex except $\sigma_2$.
	 
	 \emph{Subcase II:} Let $m > 0$. Then $\sigma_2= [\bar{a}_{i_0}, \bar{b}_{i_0}, \dots \widehat{\bar{b}_{i_m}}, \dots ,\bar{b}_{i_{q-1}}]$. Here, if $\widetilde{\sigma_1} \ne \sigma_2$, then, following the same argument as in the previous subcase, we can show that $\widetilde{\tau_j}=[\bar{a}_{i_0}, \dots, \bar{a}_{i_j}, \bar{b}_{i_j}, \dots , \bar{b_{i_{q-1}}}]$ for any $j \geq 1$. So $\sigma_2$ cannot be a facet of $\widetilde{\tau_j}$ for $j \geq 1$. Hence, $\widetilde{\sigma_1}=\sigma_2$ and $\tau_1, \sigma_2$ is the unique trajectory from $\tau_1$ to $\sigma_2$  with the property that the ground simplex remains same for every simplex except $\sigma_2$.\\
	 
	 \textbf{Case II:} Let $\tau_1=[\bar{a}_{i_0}, \bar{a}_{i_1},\bar{b}_{i_1}, \dots ,\bar{b}_{i_{q-1}}]$.
	 
	In this case, we deduce from Observation 2\ref{obsa} that $\tau_1$ must be paired with $\sigma_1=[\bar{a}_{i_1}, \bar{b}_{i_1}, \dots , \bar{b}_{i_{q-1}}]$. But $\sigma_1 \ne \sigma_2$. Therefore $m > 0$ and hence, $\sigma_2= [\bar{a}_{i_0}, \bar{b}_{i_0}, \dots \widehat{\bar{b}_{i_m}}, \dots ,\bar{b}_{i_{q-1}}]$.
	 
	 Here, we observe that $\widetilde{\sigma_1} \ne \sigma_2$. Then, from Observation 1\ref{obsc}, we know that $\widetilde{\sigma_1}=[\bar{a}_{i_0},\bar{b}_{i_1}, \dots ,\bar{b}_{i_{q-1}}]$, or $\widetilde{\sigma_1}= [\bar{a}_{i_0}, \bar{a}_{i_1},\bar{b}_{i_2}, \dots ,\bar{b}_{i_{q-1}}]$. If  $\widetilde{\sigma_1}= [\bar{a}_{i_0}, \bar{a}_{i_1},\bar{b}_{i_2}, \dots ,\bar{b}_{i_{q-1}}]$, then we argue as in the previous cases, and show that $\widetilde{\tau_j}= [\bar{a}_{i_0}, \dots, \bar{a}_{i_{j+1}}, \bar{b}_{i_{j+1}}, \dots ,\bar{b}_{i_{q-1}}]$ for $j \geq 1$. So, $\sigma_2$ cannot be the facet of any $\widetilde{\tau_j}$ for $j \geq 1$. Thus, $\widetilde{\sigma_1}= [\bar{a}_{i_0},\bar{b}_{i_1}, \dots ,\bar{b}_{i_{q-1}}]$ and hence $\widetilde{\tau_1}= [\bar{a}_{i_0}, \bar{b}_{i_0},\bar{b}_{i_1}, \dots ,\bar{b}_{i_{q-1}}]$ (follows from the construction of $\mathcal{W}$). If $\widetilde{\sigma_2} \ne \sigma_2$, then $\GS(\widetilde{\tau_1})=\GS(\widetilde{\sigma_2})$, which is not possible as there is no facet of $\widetilde{\tau_1}$ in $\mathbb{P}_{int}(\overline{A \cap B})$, excepting $\widetilde{\sigma_1}$, with the same ground simplex as $\widetilde{\tau_1}$. Hence $\widetilde{\sigma_2}=\sigma_2$. Thus, $\tau_1, \widetilde{\sigma_1}, \widetilde{\tau_1}, \sigma_2$ is the unique trajectory from $\tau_1$ to $\sigma_2$  with the property that the ground simplex remains same for every simplex except $\sigma_2$.
	\end{proof}
	
	\begin{prop}\label{prop2}
		Let $\tau_1 \in \mathbb{P}_{int}(\overline{A \cap B})$ such that $\tau_1\in \Crit_q^{\mathcal{W}}(\widetilde{X})$, or  $(\sigma_1^{(q-1)}, \tau_1^{(q)}) \in \mathcal{W}$ such that $\GS(\sigma_1) \neq \GS(\tau_1)$. Let $\sigma_2^{(q-1)} \in (A \cap B)_{\bar{A}} \cup (A \cap B)_{\bar{B}}$, such that $\GS(\tau_1)=\GS(\sigma_2)$. Then there exists a unique $\mathcal{W}$-trajectory $P$ from $\tau_1$ to $\sigma_2$ with the property that the ground simplex remains same for each simplex in $P$.(There may be more than one trajectories from $\tau_1$ to $\sigma_2$ but if there exists a trajectory $P'$ from $\tau_1$ to $\sigma_2$ with the property that the ground simplex remains unchanged throughout $P'$, then $P'=P$.)
	\end{prop}
	\begin{proof}
		  Let, if possible, $P: \tau_1^{(q)}, \widetilde{\sigma_1}^{(q-1)}, \widetilde{\tau_1}^{(q)}, \dots , \widetilde{\sigma_k}^{(q-1)}, \widetilde{\tau_k}^{(q)}, \sigma_2^{(q-1)}$ be a $\mathcal{W}$-trajectory from $\tau_1$ to $\sigma_2$ such that $\widetilde{\sigma_i}, \widetilde{\tau_i} \in \mathbb{P}(\overline{A \cap B})$, for all $i \in [k]$ and $\GS(\tau_1)=\GS(\widetilde{\sigma_1})= \dots =\GS(\widetilde{\tau_k})= \GS(\sigma_2)$. Since $\GS(\sigma_2)=\GS(\widetilde{\tau_k})$, therefore from Observation 1\ref{obs2a}, $\widetilde{\tau_k} \in \mathbb{P}_{int}(\overline{A \cap B})$. So, $(\widetilde{\sigma_k}, \widetilde{\tau_k}) \in \mathcal{W'}$. Therefore, from \autoref{prop}, we can infer that $P$ is a trajectory of the fourth or fifth type. We will show that such a trajectory is indeed possible and is unique.
		 
		 Let $\GS(\tau_1)=\GS(\sigma_2)=[\bar{c}_{i_0}, \dots , \bar{c}_{i_{q-1}}]$. 	From Observation 2\ref{obsa}, and from the construction of $\mathcal{W}$, it follows that $\tau_1$ can be either $[\bar{a}_{i_0}, \bar{b}_{i_0}, \dots ,\bar{b}_{i_{q-1}}]$, $[\bar{a}_{i_0}, \bar{a}_{i_1},\bar{b}_{i_1}, \dots ,\bar{b}_{i_{q-1}}]$.\\
		 
		  \textbf{Case I:} Let $\sigma_2 \in (A \cap B)_{\bar{A}}$, i.e., $\sigma_2=[\bar{a}_{i_0}, \dots ,\bar{a}_{i_{q-1}}] $.
		
		 \emph{Subcase I:} Let $\tau_1 = [\bar{a}_{i_0}, \bar{b}_{i_0}, \dots ,\bar{b}_{i_{q-1}}] $. Clearly, $\widetilde{\sigma_1} \ne \sigma_2$. So, $(\widetilde{\sigma_1}, \widetilde{\tau_1}) \in \mathcal{W'}$, and hence, from Observation 1\ref{obsc}, $\widetilde{\sigma_1}=[\bar{a}_{i_0}, \bar{b}_{i_1}, \dots ,\bar{b}_{i_{q-1}}]$. From the construction of $\mathcal{W'}$, $\widetilde{\tau_1}=[\bar{a}_{i_0}, \bar{a}_{i_1}, \bar{b}_{i_1}, \dots ,\bar{b}_{i_{q-1}}]$. Again, we use Observation 1\ref{obsc} to deduce that $\widetilde{\sigma_2}=[\bar{a}_{i_0}, \bar{a}_{i_1}, \bar{b}_{i_2}, \dots ,\bar{b}_{i_{q-1}}]$. So, $\widetilde{\tau_2}=[\bar{a}_{i_0}, \bar{a}_{i_1}, \bar{a}_{i_2}, \bar{b}_{i_2}, \dots ,\bar{b}_{i_{q-1}}]$. We continue in this way and observe that $\widetilde{\sigma_i} \ne \sigma_2$ for each $2 \leq i < q-1$. However, $\widetilde{\tau_{q-1}}=[\bar{a}_{i_0}, \dots ,\bar{a}_{i_{q-1}}, \bar{b}_{i_{q-1}}]$. Here, we observe that the only facet of $\widetilde{\tau_{q-1}}$, distinct from $\widetilde{\sigma_{q-1}}$ and with the same ground simplex as $\widetilde{\tau_{q-1}}$ is $\sigma_2$. So, $\widetilde{\sigma_q}=\sigma_2$. Therefore $P: \tau_1, \widetilde{\sigma_1}, \widetilde{\tau_1}, \dots ,\widetilde{\sigma_{q-1}}, \widetilde{\tau_{q-1}}, \sigma_2$ is the unique trajectory from $\tau_1$ to $\sigma_2$  with the property that the ground simplex remains same for every simplex in that trajectory.
		 
		  \emph{Subcase II:} Let, $\tau_1=[\bar{a}_{i_0}, \bar{a}_{i_1},\bar{b}_{i_1}, \dots ,\bar{b}_{i_{q-1}}]$. 
		 
		 Then, we note from Observation 1\ref{obsc}  that $\widetilde{\sigma_1}= [\bar{a}_{i_0},\bar{b}_{i_1}, \dots ,\bar{b}_{i_{q-1}}]$ or $\widetilde{\sigma_1}=[\bar{a}_{i_0}, \bar{a}_{i_1},\bar{b}_{i_2}, \dots ,\bar{b}_{i_{q-1}}]$. In the first case, $\widetilde{\tau_1}=[\bar{a}_{i_0}, \bar{b}_{i_0}, \bar{b}_{i_1}, \dots ,\bar{b}_{i_{q-1}}]$, from the construction of $\mathcal{W'}$ . However, since $\widetilde{\sigma_2} \ne \sigma_2$, therefore $(\widetilde{\sigma_2}, \widetilde{\tau_2}) \in \mathcal{W'}$ but we observe that $\widetilde{\tau_1}$ has no  facet distinct from $\widetilde{\sigma_1}$,  paired in $\mathcal{W'}$ and with the same ground simplex as $\widetilde{\tau_1}$. So the trajectory cannot continue. Therefore, $\widetilde{\sigma_1}=[\bar{a}_{i_0}, \bar{a}_{i_1},\bar{b}_{i_2}, \dots ,\bar{b}_{i_{q-1}}]$. So, $\widetilde{\tau_1}=[\bar{a}_{i_0}, \bar{a}_{i_1},\bar{a}_{i_2}, \bar{b}_{i_2}, \dots ,\bar{b}_{i_{q-1}}]$. We continue in the same way and observe that $\sigma_2 \ne \widetilde{\sigma_j}$ for $2 \leq j < (q-2)$. However, $\widetilde{\tau_{q-2}}=[\bar{a}_{i_0}, \dots ,\bar{a}_{i_{q-1}}, \bar{b}_{i_{q-1}}]$. Following the same argument as in Subcase I, we observe that, $\widetilde{\sigma_{q-1}}$ must be $\sigma_2$ and hence $P: \tau_1, \widetilde{\sigma_1}, \widetilde{\tau_1}, \dots ,\widetilde{\sigma_{q-2}}, \widetilde{\tau_{q-2}}, \sigma_2$ is the unique trajectory from $\tau_1$ to $\sigma_2$ with the property that the ground simplex remains same for every simplex in that trajectory. \\
		 
		  \textbf{Case II:}  Let $\sigma_2 \in (A \cap B)_{\bar{B}}$, i.e., $\sigma_2=[\bar{b}_{i_0}, \dots ,\bar{b}_{i_{q-1}}]$.
		 
		   \emph{Subcase I:} Let $\tau_1 = [\bar{a}_{i_0}, \bar{b}_{i_0}, \dots ,\bar{b}_{i_{q-1}}] $.
		  
		  From Observation 1\ref{obsc}, either $\widetilde{\sigma_1}=\sigma_2$, or $\widetilde{\sigma_1}=[\bar{a}_{i_0}, \bar{b}_{i_1}, \dots ,\bar{b}_{i_{q-1}}]$. In the latter case, $\widetilde{\tau_1}= [\bar{a}_{i_0}, \bar{a}_{i_1}, \bar{b}_{i_1}, \dots ,\bar{b}_{i_{q-1}}]$ (follows from the construction of $\mathcal{W}$). Then, from Observation 1\ref{obsc}, $\widetilde{\sigma_2}=[\bar{a}_{i_0}, \bar{a}_{i_1}, \bar{b}_{i_2}, \dots ,\bar{b}_{i_{q-1}}]$ and hence $\tau_2=[\bar{a}_{i_0}, \bar{a}_{i_1}, \bar{a}_{i_2}, \bar{b}_{i_2}, \dots ,\bar{b}_{i_{q-1}}]$. We continue in the same manner and show that $\widetilde{\tau_j}= [\bar{a}_{i_0},  \dots, \bar{a}_{i_j}, \bar{b}_{i_j}, \dots ,\bar{b}_{i_{q-1}}]$ for $j \geq 1$. Thus $\sigma_2$ cannot be a facet of $\widetilde{\tau_j}$ for $j \geq 1$. Therefore, $\widetilde{\sigma_1}=\sigma_2$ and $\tau_1, \sigma_2$ is the unique trajectory from $\tau_1$ to $\sigma_2$ with the property that the ground simplex remains same for every simplex in that trajectory.
		 
		  \emph{Subcase II:} Let $\tau_1=[\bar{a}_{i_0}, \bar{a}_{i_1},\bar{b}_{i_1}, \dots ,\bar{b}_{i_{q-1}}]$. 
		 
		 From Observation 1\ref{obsc}, $\widetilde{\sigma_1}= [\bar{a}_{i_0}, \bar{a}_{i_1},\bar{b}_{i_2}, \dots ,\bar{b}_{i_{q-1}}] $ or $\widetilde{\sigma_1}=[\bar{a}_{i_0},\bar{b}_{i_1}, \dots ,\bar{b}_{i_{q-1}}]$. In the first case, $\widetilde{\tau_1}= [\bar{a}_{i_0}, \bar{a}_{i_1},\bar{a}_{i_2},\bar{b}_{i_2}, \dots ,\bar{b}_{i_{q-1}}]$, from the construction of $\mathcal{W'}$. We continue in the same way to show that $\widetilde{\tau_{j}}=[\bar{a}_{i_0}, \dots , \bar{a}_{i_{j+1}}, \bar{b}_{j+1}, \dots , \bar{b}_{i_{q-1}}]$ for $j \geq 1$. Therefore, $\sigma_2$ cannot be the facet of any $\widetilde{\tau_j}$ for $j \geq 1$. So $\widetilde{\sigma_1}=[\bar{a}_{i_0},\bar{b}_{i_1}, \dots ,\bar{b}_{i_{q-1}}]$ and hence, $\widetilde{\tau_1}=[\bar{a}_{i_0}, \bar{b}_{i_0},\bar{b}_{i_1}, \dots ,\bar{b}_{i_{q-1}}]$, from the construction of $\mathcal{W'}$. The only facet of $\widetilde{\tau_1}$ distinct from $\widetilde{\sigma_1}$ and with the same ground simplex is $\sigma_2$. Therefore $\widetilde{\sigma_2}= \sigma_2$ and $\tau_1, \widetilde{\sigma_1}, \widetilde{\tau_1}, \sigma_2$ is the unique trajectory from $\tau_1$ to $\sigma_2$ with the property that the ground simplex remains same for every simplex in that trajectory.

	\end{proof}
	
	 Let $\tau_1^{(q)}, \sigma_2^{(q-1)} \in \widetilde{X}$ such that they satisfy the conditions of \autoref{prop1} and \autoref{prop2}. Then the unique trajectory from $\tau_1$ to $\sigma_2$ with the ground simplex preserving property as stated in \autoref{prop1} and \autoref{prop2}, is denoted by $P_{(\tau_1, \sigma_2)}$.
	
	Now, we make the following observations (already implicit in the proofs of the above propositions), for our future reference.
	
	\begin{obs}
		Let $\tau_1 \in \mathbb{P}_{int}(\overline{A \cap B})$ such that $\tau_1 \in \Crit_q^{\mathcal{W}}(\widetilde{X})$, or  $(\sigma_1^{(q-1)}, \tau_1^{(q)}) \in \mathcal{W}$ such that $\GS(\sigma_1) \neq \GS(\tau_1)$. Let $ \sigma_2 \in \mathbb{P}_{int}(\overline{A \cap B})$ such that $\sigma_2 \ne \sigma_1$ and $\sigma_2 \in \Crit_{q-1}^{\mathcal{W}}(\widetilde{X})$, or $(\sigma_2^{(q-1)}, \tau_2^{(q)}) \in \mathcal{W}$ such that $\GS(\sigma_2) \ne \GS(\tau_2)$. Suppose $\GS(\tau_1)=[\bar{c}_{i_0}, \dots ,\bar{c}_{i_{q-1}}]$, $\GS(\sigma_2)= [\bar{c}_{i_0}, \dots, \widehat{\bar{c}_{i_m}} , \dots ,\bar{c}_{i_{q-1}}]$. Then,
		
		\begin{enumerate}[label=(\roman*)]
			\item For $\tau_1=[\bar{a}_{i_0}, \bar{b}_{i_0}, \dots ,\bar{b}_{i_{q-1}}]$,\label{obs3a}
			\begin{enumerate}
				\item if $m=0$, then $\sigma_2=[\bar{a}_{i_1},\bar{b}_{i_1}, \dots ,\bar{b}_{i_{q-1}}]$ and $P_{(\tau_1, \sigma_2)}: \tau_1,\widetilde{\sigma_1}, \widetilde{\tau_1}, \sigma_2$, where, $\widetilde{\sigma_1}=[\bar{a}_{i_0},\bar{b}_{i_1}, \dots ,\bar{b}_{i_{q-1}}]$, $\widetilde{\tau_1}=[\bar{a}_{i_0}, \bar{a}_{i_1}, \bar{b}_{i_1}, \dots ,\bar{b}_{i_{q-1}}]$.
				\item if $m > 0$, then $\sigma_2 = [\bar{a}_{i_0}, \bar{b}_{i_0}, \dots, \widehat{\bar{b}_{i_m}}, \dots ,\bar{b}_{i_{q-1}}]$. In this case, $P_{(\tau_1, \sigma_2)}: \tau_1,\sigma_2$.
			\end{enumerate}
	\item  For $\tau_1=[\bar{a}_{i_0}, \bar{a}_{i_1}, \bar{b}_{i_1}, \dots ,\bar{b}_{i_{q-1}}]$,  $m > 0$ and $\sigma_2= [\bar{a}_{i_0}, \bar{b}_{i_0}, \dots, \widehat{\bar{b}_{i_m}}, \dots ,\bar{b}_{i_{q-1}}]$. In this case, $P_{(\tau_1, \sigma_2)}: \tau_1,\widetilde{\sigma_1}, \widetilde{\tau_1}, \sigma_2$, where, $\widetilde{\sigma_1}=[\bar{a}_{i_0},\bar{b}_{i_1}, \dots ,\bar{b}_{i_{q-1}}]$, $\widetilde{\tau_1}=[\bar{a}_{i_0}, \bar{b}_{i_0}, \bar{b}_{i_1}, \dots ,\bar{b}_{i_{q-1}}]$. \label{obs3b}

		\end{enumerate}
		(These follow from the proof of \autoref{prop1}).
	\end{obs}
	
	\begin{obs}
			Let $\tau_1 \in \mathbb{P}_{int}(\overline{A \cap B})$ such that $\tau_1\in \Crit_q^{\mathcal{W}}(\widetilde{X})$, or  $(\sigma_1^{(q-1)}, \tau_1^{(q)}) \in \mathcal{W}$ such that $\GS(\sigma_1) \neq \GS(\tau_1)$. Let $\sigma_2^{(q-1)} \in (A \cap B)_{\bar{A}} \cup (A \cap B)_{\bar{B}}$ and $\GS(\tau_1)=\GS(\sigma_2)= [\bar{c}_{i_0}, \dots , \bar{c}_{i_{q-1}}]$. Then,
			
			\begin{enumerate}[label=(\roman*)]
				\item For $\sigma_2=[\bar{a}_{i_0}, \dots, \bar{a}_{i_{q-1}}]$,\label{obs4a}
				\begin{enumerate}
					\item If $\tau_1=[\bar{a}_{i_0}, \bar{b}_{i_0}, \dots , \bar{b}_{i_{q-1}}]$, then $P_{(\tau_1, \sigma_2)}: \tau_1, \widetilde{\sigma_1}, \widetilde{\tau_1}, \dots, \widetilde{\sigma_{q-1}}, \widetilde{\tau_{q-1}}, \sigma_2$, where, $\widetilde{\sigma_1}= [\bar{a}_{i_0}, \bar{b}_{i_1}, \dots , \bar{b}_{i_{q-1}}]$, $\widetilde{\tau_1}=[\bar{a}_{i_0}, \bar{a}_{i_1}, \bar{b}_{i_1}, \dots , \bar{b}_{i_{q-1}}]$, $\widetilde{\sigma_2}=[\bar{a}_{i_0}, \bar{a}_{i_1} \bar{b}_{i_2}, \dots , \bar{b}_{i_{q-1}}]$, $\dots, \widetilde{\sigma_j}= [\bar{a}_{i_0}, \dots , \bar{a}_{i_{j-1}}, \bar{b}_{i_j}, \dots , \bar{b}_{i_{q-1}}]$, $\widetilde{\tau_j}=[\bar{a}_{i_0}, \dots  \bar{a}_{i_j}, \bar{b}_{i_j}, \dots , \bar{b}_{i_{q-1}}]$, \dots , $\widetilde{\sigma_{q-1}}=[\bar{a}_{i_0}, \dots ,\bar{a}_{i_{q-2}}, \bar{b}_{i_{q-1}}]$, $\widetilde{\tau_{q-1}}=[\bar{a}_{i_0}, \dots ,\bar{a}_{i_{q-1}}, \bar{b}_{i_{q-1}}]$. \\
					
					\item  If $\tau_1=[\bar{a}_{i_0}, \bar{a}_{i_1}, \bar{b}_{i_1}, \dots , \bar{b}_{i_{q-1}}]$, then $P_{(\tau_1, \sigma_2)}: \tau_1, \widetilde{\sigma_1}, \widetilde{\tau_1}, \dots, \widetilde{\sigma_{q-2}}, \widetilde{\tau_{q-2}}, \sigma_2$, where, \\ $\widetilde{\sigma_1}= [\bar{a}_{i_0}, \bar{a}_{i_1}, \bar{b}_{i_2}, \dots , \bar{b}_{i_{q-1}}]$, $\widetilde{\tau_1}=[\bar{a}_{i_0}, \bar{a}_{i_1}, \bar{a}_{i_2}, \bar{b}_{i_2}, \dots , \bar{b}_{i_{q-1}}]$, $\widetilde{\sigma_2}=[\bar{a}_{i_0}, \bar{a}_{i_1} \bar{b}_{i_2}, \dots , \bar{b}_{i_{q-1}}]$, $\dots, \widetilde{\sigma_j}= [\bar{a}_{i_0}, \dots , \bar{a}_{i_{j}}, \bar{b}_{i_{j+1}}, \dots , \bar{b}_{i_{q-1}}]$, $\widetilde{\tau_j}=[\bar{a}_{i_0}, \dots  \bar{a}_{i_{j+1}}, \bar{b}_{i_{j+1}}, \dots , \bar{b}_{i_{q-1}}]$, $\dots , \widetilde{\sigma_{q-2}}=[\bar{a}_{i_0}, \dots ,\bar{a}_{i_{q-2}}, \bar{b}_{i_{q-1}}]$, $\widetilde{\tau_{q-2}}=[\bar{a}_{i_0}, \dots ,\bar{a}_{i_{q-2}},\bar{a}_{i_{q-1}}, \bar{b}_{i_{q-1}}]$. 
				\end{enumerate}
				
				\item  For $\sigma_2=[\bar{b}_{i_0}, \dots , \bar{b}_{i_{q-1}}]$,\label{obs4b}
				
				\begin{enumerate}
					\item If $\tau_1=[\bar{a}_{i_0}, \bar{b}_{i_0}, \dots , \bar{b}_{i_{q-1}}]$, then $P_{(\tau_1, \sigma_2)}: \tau_1, \sigma_2$.
					
					\item  If $\tau_1=[\bar{a}_{i_0}, \bar{a}_{i_1}, \bar{b}_{i_1}, \dots , \bar{b}_{i_{q-1}}]$, then $P_{(\tau_1, \sigma_2)}: \tau_1, \widetilde{\sigma_1}, \widetilde{\tau_1}, \sigma_2$, where, $\widetilde{\sigma_1}=[\bar{a}_{i_0}, \bar{b}_{i_1}, \dots , \bar{b}_{i_{q-1}}]$, $\widetilde{\tau_1}=[\bar{a}_{i_0}, \bar{b}_{i_0}, \bar{b}_{i_1}, \dots , \bar{b}_{i_{q-1}}]$.
					
				\end{enumerate}
				
			\end{enumerate}
			(These follow from the proof of \autoref{prop2}).
	\end{obs}

		\subsection{Isomorphism between the simplicial homology of $\widetilde{X}$ and the homology groups of $\mathcal{D}_{\#}(X)$}

 We recall from Section~\ref{intro} that   $D_0(X)=\{\Crit^{\mathcal{W}_{\bar{A}}}_0(\bar{A}) \cup \Crit^{\mathcal{W}_{\bar{B}}}_0(\bar{B})\}$,
$D_q (X)= \{\Crit^{\mathcal{W}_{\bar{A}}}_q(\bar{A}) \cup \Crit^{\mathcal{W}_{\bar{B}}}_q(\bar{B}) \cup \Crit^{\mathcal{W}_{\overline{A \cap B}}}_{q-1}(\overline{A \cap B})\}$, $q \geq 1$. We also recall that the generators of $C_q^{\mathcal{W}}(\widetilde{X}, \mathbb{Z})$ are given by $\Crit_q^{\mathcal{W}}(\widetilde{X})= \{C_q^{\bar{A}} \cup C_q^{\bar{B}} \cup C_q^{int}\}$. Now we construct an isomorphism $f_{\#}$ between $C_{\#}^{\mathcal{W}}(\widetilde{X}, \mathbb{Z})$ and $\mathcal{D}_{\#}(X)$.\\

 Let $f_q: C_{q}^{\mathcal{W}}(\widetilde{X}, \mathbb{Z}) \rightarrow \mathcal{D}_{q}(X)$ in the following manner. It suffices to define the map on the set of generators. Let $\alpha \in \Crit_q^{\mathcal{W}}(\widetilde{X})$. \\
If $\alpha \in C_q^{\bar{A}}$ or $\alpha \in C_q^{\bar{B}}$, then $\alpha = \sigma_{\bar{A}}$, for some $\sigma \in A$, or $\alpha= \sigma_{\bar{B}}$, for some $\sigma \in B$ respectively. So, we note that $\alpha \in \Crit_q^{\mathcal{W}_{\bar{A}}}(\bar{A})$ or $\alpha \in \Crit_q^{\mathcal{W}_{\bar{B}}}(\bar{B})$ respectively. Therefore, in this case, we define $f_q(\alpha)=\alpha$.

If $\alpha \in C_q^{int}$, then we note from Observation 2\ref{obsd} that $\GS(\alpha) \in \Crit^{\mathcal{W}_{\overline{A \cap B}}}_{q-1}(\overline{A \cap B})\}$. So we define $f_q(\alpha)= \GS(\alpha)$.

 Now, let $g_q: \mathcal{D}_{q}(X) \rightarrow C_{q}^{\mathcal{W}}(\widetilde{X}, \mathbb{Z})$. As before, it suffices to define $g_q$ on the set of generators. Let $\alpha \in D_{q}(X) $. Then we define,
\begin{equation*}
	g_q(\alpha) = \begin{cases}
		\sigma_{\bar{A}} & \alpha\in \Crit_q^{\mathcal{W}_{\bar{A}}}(\bar{A}) \text{, and }\alpha=\sigma_{\bar{A}},\\
		\sigma_{\bar{B}} & \alpha \in \Crit_q^{\mathcal{W}_{\bar{B}}}(\bar{B}) \text{, and }\alpha=\sigma_{\bar{B}},\\
		[\bar{a}_{i_0}, \bar{b}_{i_0}, \dots \bar{b}_{i_{q-1}}] & \alpha \in \Crit_{q-1}^{\mathcal{W}_{\overline{A \cap B}}}(\overline{A \cap B}) \text{, where } \alpha =[\bar{c}_{i_0}, \dots \bar{c}_{i_{q-1}}]. \\
	\end{cases}
\end{equation*}

 It suffices to prove that $f_q \circ g_q = \operatorname{id}$ and $g_q \circ f_q= \operatorname{id}$ on the set of generators.

 If $\sigma_{\bar{A}} \in \Crit_q^{\mathcal{W}_{\bar{A}}}(\bar{A})$ or $\sigma_{\bar{B}} \in \Crit_q^{\mathcal{W}_{\bar{B}}}(\bar{B})$, then $f_q$ and $g_q$ are simply identity maps, which are isomorphisms. However, if $\alpha=[\bar{c}_{i_0}, \dots \bar{c}_{i_{q-1}}] \in \Crit_{q-1}^{\mathcal{W}_{\overline{A \cap B}}}(\overline{A \cap B})$, then $g_q(\alpha)=[\bar{a}_{i_0},\bar{b}_{i_0}, \dots ,\bar{b}_{i_{q-1}}]$. We denote this as $\sigma$. So, $f_q \circ g_q(\alpha)=\GS(\sigma)=[\bar{c}_{i_0}, \dots ,\bar{c}_{i_{q-1}}]$.

 Let $\alpha \in C_q^{int}$. Now, $g_q \circ f_q (\alpha)=g_q(\GS(\alpha))$. We know that $\GS(\alpha) \in \Crit_{q-1}^{\mathcal{W}_{\overline{A \cap B}}}(\overline{A \cap B})$ from Observation 2\ref{obsd}. So, if $\GS(\alpha)=[\bar{c}_{i_0}, \dots ,\bar{c}_{i_{q-1}}]$, then $g_q(\GS(\alpha))= [\bar{a}_{i_0},\bar{b}_{i_0}, \dots,\bar{b}_{i_{q-1}}] \in C^{int}_q$. Since both $\alpha$ and $[\bar{a}_{i_0},\bar{b}_{i_0} \dots ,\bar{b}_{i_{q-1}}]$ are in $S_{\GS(\alpha)}'$, therefore, from Observation 2\ref{obsd}, $\alpha=[\bar{a}_{i_0},\bar{b}_{i_0} \dots ,\bar{b}_{i_{q-1}}]$. Thus, $g_q(\GS(\alpha))=\alpha$.

 Hence, we have established an isomorphism between  $C_q(\widetilde{X}, \mathbb{Z})$ and $\mathcal{D}_q(X)$.\\

\begin{lem}\label{bij}
	Let $\tau \in \Crit_q^{\mathcal{W}}(\widetilde{X})$ and $\sigma \in \Crit_{q-1}^{\mathcal{W}}(\widetilde{X})$. Then, there exists a bijection between $\Gamma(\tau, \sigma)$ and  $\MV(f_q(\tau), f_{q-1}(\sigma))$.
\end{lem}
\begin{proof}
	
 First we construct a map $\mu : \Gamma(\tau, \sigma) \rightarrow \MV(f_q(\tau),f_{q-1}(\sigma))$.

 Let $P \in \Gamma(\tau, \sigma)$.

 \textbf{Case I:} $P$ is a $\mathcal{W}$-trajectory of the first or second type, i.e., $\tau \in C^{\bar{A}}_q$   (or $\tau \in C^{\bar{B}}_q$).

Therefore, from Proposition~\ref{prop}, $P$ must be either a $\mathcal{W}_{\bar{A}}$-trajectory (or a $\mathcal{W}_{\bar{B}}$-trajectory). Hence $\sigma \in C^{\bar{A}}_q$ (or $\sigma \in C^{\bar{B}}_q$). So, we define $\mu(P)$ as $P$, which we observe is also an $\MV$-trajectory from $f_q(\tau)=\tau$ to $f_{q-1}(\sigma)= \sigma$.\\

 \textbf{Case II:} $P$ is a $\mathcal{W}$-trajectory of the third type.

 Therefore $\tau \in C^{int}_q$ and $P$ is of the form,
$$(\tau=)~\tau_0^{(q)}, \sigma_1^{(q-1)}, \tau_1^{(q)} \dots  ,\sigma_k^{(q-1)}, \tau_k^{(q)} ,\sigma_{k+1}^{(q-1)}~(= \sigma),$$

 where $\sigma \in C_{q-1}^{int}$.

If the ground simplex remains same for each simplex in $P$, then $\tau, \sigma \in S_{\gamma}' $ for some $\gamma \in \Crit^{\mathcal{W}_{\overline{A \cap B}}}(\overline{A \cap B})$. From Observation 2\ref{obsd}, we note that this is not possible. Therefore, there exists a simplex in $P$ whose ground simplex differs from its preceding simplices, i.e., either of the following occurs for some $i$.
$$ \GS(\tau)= \GS(\sigma_1) = \dots =\GS(\sigma_i) \neq \GS(\tau_i),$$
\begin{center}
	or
\end{center}
$$ \GS(\tau)= \GS(\sigma_1) = \dots =\GS(\tau_{i-1}) \neq \GS(\sigma_i).$$
In the first case, it follows from Observation 2\ref{obsa} that $\tau_{i-1} \in A_{\alpha}$ for some $\alpha \in \overline{A \cap B}$ and hence from Observation 1\ref{obsc}, we note that $\sigma_i \in B_{\alpha}$. We again use Observation 2\ref{obsa} to deduce that $\GS(\sigma_i) = \GS(\tau_i)$ and hence arrive at a contradiction. In the second case, the ground simplex differs for the first time at a lower dimensional simplex $\sigma_i$. Let us rename it as $\sigma_{r_1}$. 

If $\sigma_{r_1}= \sigma$, then we obtain a sequence of simplices, $P': \GS(\tau), \GS(\sigma)$. Otherwise we note from Observation 1\ref{obsc} that $\sigma_{r_1} \in A_{\alpha}$ for some $\alpha \in \overline{A \cap B}$ and from Observation 2\ref{obsa}, we note that $\GS(\sigma_{r_1}) \neq \GS(\tau_{r_1})$. Now, it is not possible that $\GS(\tau_{r_1}) = \GS(\sigma_{r_1+1})= \GS(\tau_{r_1+1})= \dots = \GS(\sigma)$. This is because, from Observation 2\ref{obse}$, (\GS(\sigma_{r_1}), \GS(\tau_{r_1})) \in \mathcal{W}_{\overline{A \cap B}}$ but $\GS(\sigma) \in \Crit^{\mathcal{W}_{\overline{A \cap B}}}(\overline{A \cap B})$, from Observation 2\ref{obsd}. So, again we follow the same argument as before and note that ground simplex changes next at a lower dimensional simplex $\sigma_{r_2}$, i.e., $\GS(\tau_i)= \GS(\sigma_{r_1+1})= \GS(\tau_{r_1+1})= \dots = \GS(\tau
_{r_2-1}) \ne \GS(\sigma_{r_2})$. We note if $\sigma_{r_2}= \sigma$ and accordingly continue the same procedure until $\sigma_{r_m}= \sigma$ for some $m$ and thus obtain the following sequence of simplices in $\overline{A \cap B}$.

$$ P': f_q(\tau_0)= \GS(\tau_0), \GS(\sigma_{r_1}), \GS(\tau_{r_1}), \dots , \GS(\sigma_{r_{m-1}}), \GS(\tau_{r_{m-1}}), \GS(\sigma)=f_{q-1}(\sigma).$$

We observe that $\GS(\tau_0)= \GS(\tau_{r_1-1}) \ne \GS(\sigma_{r_1})$. Also, for each $0 \leq i \leq (m-2)$, $\GS(\tau_{r_i})= \GS(\tau_{r_{i+1}-1}) \ne \GS(\sigma_{r_{i+1}})$. Thus, from Observation 1\ref{obsb}, we note that $\GS(\sigma_{r_1})^{(q-2)} \subseteq \GS(\tau_0)^{(q-1)}$ and $\GS(\sigma_{r_{i+1}})^{(q-2)} \subseteq \GS(\tau_{r_i})^{(q-1)}$, $1 \leq i \leq (m-2)$. Now we use Observation 2\ref{obse} to deduce that, $(\GS(\sigma_{r_i}), \GS(\tau_{r_i}))) \in \mathcal{W}_{\overline{A \cap B}}$ for $i \in [m-1]$. Since, $\GS(\tau_0), \GS(\sigma)$ are $\mathcal{W}_{\overline{A \cap B}}$-critical, therefore it follows that $(\GS(\tau_0), \GS(\sigma_{r_1})) \notin \mathcal{W}_{\overline{A \cap B}}$, $(\GS(\sigma),\GS(\tau_{r_{m-1}})) \notin \mathcal{W}_{\overline{A \cap B}}$. The fact that $(\GS(\sigma_{r_i}),\GS(\tau_{r_{i-1}})) \notin \mathcal{W}_{\overline{A \cap B}}$ for each $2 \leq i \leq (m-1)$, follows from Observation 2\ref{obse} and the acyclicity of $\mathcal{W}_{\overline{A \cap B}}$. Therefore $P' \in \MV(f_q(\tau), f_{q-1}(\sigma))$.

We define $\mu(P)$ as $P'$. It follows from the construction of $P'$, that for each $P \in \Gamma(\tau, \sigma)$, $P'$ is unique.\\

 \textbf{Case III:} $P$ is a $\mathcal{W}$-trajectory of the fourth or fifth type.

 Therefore, $\tau \in C_q^{int}$ and $P$ is either of the following forms.
$$(\tau=)~\tau_0^{(q)}, \sigma_1^{(q-1)},\tau_1^{(q)}, \dots  ,\sigma_{p}^{(q-1)}, \tau_{p}^{(q)} ,(\sigma_{p+1})_{\bar{A}}^{(q-1)}, (\tau_{p+1})_{\bar{A}}^{(q)}, \dots ,(\sigma_{p+l-1})_{\bar{A}}^{(q-1)}, (\tau_{p+l-1})_{\bar{A}}^{(q)}, (\sigma_{p+l})_{\bar{A}}^{(q-1)}~(= \sigma).$$ 
\begin{center}
or
\end{center}
$$(\tau=)~\tau_0^{(q)}, \sigma_1^{(q-1)},\tau_1^{(q)} \dots  ,\sigma_{p}^{(q-1)}, \tau_{p}^{(q)} ,(\sigma_{p+1})_{\bar{B}}^{(q-1)}, (\tau_{p+1})_{\bar{B}}^{(q)}, \dots ,(\sigma_{p+l-1})_{\bar{B}}^{(q-1)}, (\tau_{p+l-1})_{\bar{B}}^{(q)}, (\sigma_{p+l})_{\bar{B}}^{(q-1)}~(= \sigma).$$ 

 Without loss of generality, let us assume $P$ is of the first form.

We recall from the end of Section~\ref{prelim} that for a trajectory $P$, we denote a subsequence consisting of consecutive simplices $\alpha_{i_1},\dots , \alpha_{i_r}$ in $P$ as $\alpha_{i_1} P \alpha_{i_r}$.

Let $Q= \tau P (\sigma_{p+1})_{\bar{A}}$ and $R=(\sigma_{p+1})_{\bar{A}} P \sigma$. We now construct a sequence $Q'$ depending on $Q$, as follows. 

If the ground simplex remains same for each simplex in $Q$, then $Q': \GS(\tau)$. Otherwise, we follow the same argument as in Case II and show that the ground simplex changes for the first time at a lower dimensional simplex, $\sigma_{r_1}$, i.e., $ \GS(\tau)= \GS(\sigma_1) = \dots =\GS(\tau_{r_1-1}) \neq \GS(\sigma_{r_1})$. Since $(\sigma_{p+1})_{\bar{A}} \in \bar{A} $ and $\tau_{{r_1}-1} \in \mathbb{P}_{int}(\overline{A \cap B})$, therefore it follows from Observation 1\ref{obs2a} that $(\sigma_{p+1})_{\bar{A}} \ne \sigma_{r_1}$. So, we follow the same argument as in Case II to deduce that $\GS(\tau_{r_1}) \ne \GS(\sigma_{r_1})$. 


If the ground simplex changes again, then we continue the same procedure till we reach $\tau_{r_m}$ for some $m$, such that the ground simplex remains same for all the simplices in $Q$ following $\tau_{r_m}$, i.e., $\GS(\tau_{r_m}) = \GS(\sigma_{r_m+1})= \GS(\tau_{r_m+1})= \dots = \GS((\sigma_{p+1})_{\bar{A}})$. This is inevitable because, from Observation 2\ref{obs2a}, we know that $\GS(\tau_p)= \GS((\sigma_{p+1})_{\bar{A}})$. Thus we obtain a sequence of simplices,
$$Q': f_q(\tau)= \GS(\tau), \GS(\sigma_{r_1}), \GS(\tau_{r_1}), \dots ,\GS(\sigma_{r_m}), \GS(\tau_{r_m}).$$

 where $\GS(\tau_{r_m})=\GS((\sigma_{p+1})_{\bar{A}})= (\sigma_{p+1})_{\overline{A \cap B}}$ (follows from Observation 2\ref{obs2a} and Observation 2\ref{obs2b}).
 
 We define $\mu(P)$ as the sequence of simplices $Q', R$. Following analogous arguments as in Case II, and from the fact that $R$ itself is a $\mathcal{W}_{\bar{A}}$-trajectory it can be observed that $\mu(P) \in \MV(f_q(\tau), f_{q-1}(\sigma))$. In this case also, it follows from our construction that for each $P \in \Gamma(\tau, \sigma)$, $\mu(P)$ is unique.\\
 
  Next, let $\beta \in D_q(X)$, $\alpha \in D_{q-1}(X)$. We define a map $\eta: \MV(\beta, \alpha) \rightarrow \Gamma(g_q(\beta), g_{q-1}(\alpha))$.
 
 Let $P \in \MV(\beta, \alpha)$.

 \textbf{Case I:} Let $\beta \in \Crit_q^{\mathcal{W}_{\bar{A}}}(\bar{A})$ and $\alpha \in \Crit_{q-1}^{\mathcal{W}_{\bar{A}}}(\bar{A})$, or $\beta \in \Crit_q^{\mathcal{W}_{\bar{B}}}(\bar{B})$ and $\alpha \in \Crit_{q-1}^{\mathcal{W}_{\bar{B}}}(\bar{B})$.

 Then $P$ is either of the following.

$$(\beta=)~(\beta_0)_{\bar{A}}^{(q)}, (\alpha_1)_{\bar{A}}^{(q-1)}, (\beta_1)_{\bar{A}}^{(q)} \dots  ,(\alpha_k)_{\bar{A}}^{(q-1)}, (\beta_k)_{\bar{A}}^{(q)},(\alpha_{k+1})_{\bar{A}}^{(q-1)}~(=\alpha),$$
\begin{center}
	or
\end{center}
$$(\beta=) (\beta_0)_{\bar{B}}^{(q)}, (\alpha_1)_{\bar{B}}^{(q-1)}, (\beta_1)_{\bar{B}}^{(q)} \dots  ,(\alpha_k)_{\bar{B}}^{(q-1)}, (\beta_k)_{\bar{B}}^{(q)},(\alpha_{k+1})_{\bar{B}}^{(q-1)}(=\alpha).$$

 In this case, we observe that $P$ is itself a $\mathcal{W}$-trajectory of the first or second type. Hence we define $\eta(P)$ as $P$.

\textbf{Case II:} Let $\beta \in \Crit_{q-1}^{\mathcal{W}_{\overline{A \cap B}}}(\overline{A \cap B})$ and $\alpha \in \Crit_{q-2}^{\mathcal{W}_{\overline{A \cap B}}}(\overline{A \cap B})$.\\
 Then let $P$ be the following sequence of simplices.
 $$ (\beta=)~(\beta_0)_{\overline{A \cap B}}^{(q-1)}, (\alpha_1)_{\overline{A \cap B}}^{(q-2)}, (\beta_1)_{\overline{A \cap B}}^{(q-1)}, \dots , (\alpha_k)_{\overline{A \cap B}}^{(q-2)}, (\beta_k)_{\overline{A \cap B}}^{(q-1)}, (\alpha_{k+1})_{\overline{A \cap B}}^{(q-2)}~(=\alpha).$$

We know from Observation 2\ref{obsd} that there exists a unique simplex $\tau_0 \in S_{\beta}'$ such that $\tau_0$ is $\mathcal{W}$-critical and there exists a unique simplex $\sigma_{k+1} \in S_{\alpha}'$ such that $\sigma_{k+1}$ is $\mathcal{W}$-critical. So, $g_q(\beta)= \tau_0$,  $g_q(\alpha)= \sigma_{k+1}$. From Observation 2\ref{obse}, we know that for each $((\alpha_i)_{\overline{A \cap B}}, (\beta_i)_{ \overline{A \cap B}})$, there exists a unique pair $(\sigma_i, \tau_i) \in \mathcal{W'}$ such that $\GS(\sigma_i) = (\alpha_i)_{\overline{A \cap B}}$, $\GS(\tau_i)= (\beta_i)_{\overline{A \cap B}}$. Now, we use Proposition~\ref{prop1} to deduce that there exists a unique $\mathcal{W}$-trajectory from $\tau_i$ to $\sigma_{i+1}$ with the property that the ground simplex remains same for every simplex in the trajectory except $\sigma_{i+1}$ for each $0 \leq i \leq k$. We recall that such a trajectory is denoted as $P_{(\tau_i, \sigma_{i+1})}$.

We define $\eta(P)$ as $P_{(\tau_0, \sigma_1)}, P_{(\tau_1, \sigma_2)}, \dots , P_{(\tau_k, \sigma_{k+1})}$.

Since, $\tau_0, \sigma_{k+1} \in \mathbb{P}_{int}(\overline{A \cap B})$, therefore $\eta(P)$ is a $\mathcal{W}$-trajectory of the third type and hence $\eta(P) \in \Gamma(\tau, \sigma)$. We also observe that such a trajectory is unique for each $P \in \MV(\beta, \alpha)$.\\

\textbf{Case III:} Let $\beta \in \Crit_{q-1}^{\mathcal{W}_{\overline{A \cap B}}}(\overline{A \cap B})$ and $ \alpha \in \Crit_{q-1}^{\mathcal{W}_{\bar{A}}}(\bar{A})$, or $\beta \in \Crit_{q-1}^{\mathcal{W}_{\overline{A \cap B}}}(\overline{A \cap B})$ and $\alpha \in \Crit_{q-1}^{\mathcal{W}_{\bar{B}}}(\bar{B})$.

Then $P$ can be either of the following.
$$ 
\begin{aligned}
	 &(\beta=)~(\beta_0)_{\overline{A \cap B}}^{(q-1)}, (\alpha_1)_{\overline{A \cap B}}^{(q-2)}, (\beta_1)_{\overline{A \cap B}}^{(q-1)} \dots  (\alpha_p)_{\overline{A \cap B}}^{(q-2)}, (\beta_p)_{\overline{A \cap B}}^{(q-1)}, (\beta_p)_{\bar{A}}^{(q-1)}, (\gamma_p)_{\bar{A}}^{(q)}, \dots (\beta_{p+l-1})_{\bar{A}}^{(q-1)},\\ & (\gamma_{p+l-1})_{\bar{A}}^{(q)}, (\beta_{p+l})_{\bar{A}}^{(q-1)}~(=\alpha)
\end{aligned}
$$

\begin{center}
	or
\end{center}
$$  
\begin{aligned}
	&(\beta=)~(\beta_0)_{\overline{A \cap B}}^{(q-1)}, (\alpha_1)_{\overline{A \cap B}}^{(q-2)}, (\beta_1)_{\overline{A \cap B}}^{(q-1)} \dots  (\alpha_p)_{\overline{A \cap B}}^{(q-2)}, (\beta_p)_{\overline{A \cap B}}^{(q-1)}, (\beta_p)_{\bar{B}}^{(q-1)}, (\gamma_p)_{\bar{B}}^{(q)}, \dots (\beta_{p+l-1})_{\bar{B}}^{(q-1)}, \\ &(\gamma_{p+l-1})_{\bar{B}}^{(q)}, (\beta_{p+l})_{\bar{B}}^{(q-1)}~(=\alpha)
\end{aligned}
$$

 Without loss of generality, let us assume that $P$ is the first sequence of simplices.

We recall from the end of Section~\ref{prelim} that for an $\MV$-trajectory $P$, we denote a subsequence of consecutive simplices $\alpha_{i_1}, \dots , \alpha_{i_r}$ in $P$ as $\alpha_{i_1} P \alpha_{i_r}$.

Let $Q=\beta P (\beta_p)_{\overline{A \cap B}}$ and $R= (\gamma_p)_{\bar{A}} P \alpha$. Let as before, $\tau_0$ be the unique $\mathcal{W}$-critical simplex such that $\tau_0 \in S_{\beta}'$. So, $g_q(\beta)= \tau_0$. As in Case II, for each $(\alpha_i, \beta_i) \in \mathcal{W}_{\overline{A \cap B}}$, let $(\sigma_i, \tau_i) \in \mathcal{W}$ be the unique pair such that $\GS(\sigma_i)=(\alpha_i)_{\overline{A \cap B}}$, $\GS(\tau_i)=(\beta_i)_{\overline{A \cap B}}$ for each $i \in [p]$. So, we use \autoref{prop1} to obtain the trajectories $P_{(\tau, \sigma_1)}, P_{(\tau_1, \sigma_2)}, \dots P_{(\tau_{p-1}, \sigma_p)}$. 


Now, let $\sigma_{p+1}= (\beta_p)_{\bar{A}}$. We note that $\sigma_{p+1} \in (A \cap B)_{\bar{A}}$ and $(\sigma_p, \tau_p) \in \mathcal{W'}$. Also, from Observation 1\ref{obs2b}, $\GS(\tau_p)=(\beta_p)_{\overline{A \cap B}} =\GS(\sigma_{p+1})$. So we use Proposition~\ref{prop2} to obtain the trajectory $P_{(\tau_p, \sigma_{p+1})}$.
Now, we define $\eta(P)$ as $P_{(\tau, \sigma_1)}, P_{(\tau_1, \sigma_2)}, \dots , P_{(\tau_p, \sigma_{p+1})}, R$.

In each case we observe that $\eta(P)$ is a $\mathcal{W}$-trajectory of the fourth type and is unique for each $P \in \MV(\beta, \alpha)$. \\

Next we show that these maps are indeed bijections, i.e., $\eta \circ \mu (P) =P$ for each $P \in \Gamma(\tau, \sigma)$, where $\tau \in \Crit_q^{\mathcal{W}}(\widetilde{X}), \sigma \in \Crit_{q-1}^{\mathcal{W}}(\widetilde{X})$ and $\mu \circ \eta (P)$ for each $P \in \MV(\beta, \alpha)$, where $\beta \in D_q(X)$, $\alpha \in D_{q-1}(\widetilde{X})$. 

Let $P \in \Gamma(\tau, \sigma)$, where $\tau \in \Crit_q^{\mathcal{W}}(\widetilde{X}), \sigma \in \Crit_{q-1}^{\mathcal{W}}(\widetilde{X})$. From our construction of the maps, it suffices to show that $\eta \circ \mu(P)=P$, when $P$ is a $W$-trajectory of either the third, fourth or fifth type.\\

 \textbf{Case I:} $P$ is a $\mathcal{W}$-trajectory of the third type. 
Then let, 
$$P: (\tau=)~\tau_0^{(q)}, \sigma_1^{(q-1)}, \tau_1^{(q)} \dots  ,\sigma_k^{(q-1)}, \tau_k^{(q)} ,\sigma_{k+1}^{(q-1)}~(= \sigma).$$
where $\tau \in C^{int}_q$, $\sigma \in C^{int}_{q-1}$.

 Suppose, the ground simplex changes at $\sigma_{r_i}, \tau_{r_i}$, $i \in [m]$, and at $\sigma_{r_m}= \sigma$,  i.e.
\[
\begin{aligned}
	& \GS(\tau_0) = \GS(\sigma_1)= \GS(\tau_1) = \dots \GS(\tau_{r_1-1}) \ne \GS(\sigma_{r_1}) \ne \GS(\tau_{r_1}) = \dots =\GS(\tau_{r_2-1}) \\&\ne \GS(\sigma_{r_2}) \ne \GS(\tau_{r_2})= \dots = \GS(\tau_{r_{m-1}-1}) \ne \GS(\sigma_{r_{m-1}}) \ne \GS(\tau_{r_{m-1}}) = \dots =\GS(\tau_{r_m-1}) \neq \GS(\sigma_{r_m}).
\end{aligned}
\] 

Therefore, as per our construction, $\mu(P): \GS(\tau_0), \GS(\sigma_{r_1}), \GS(\tau_{r_1}), \dots, \GS(\sigma_{r_{m-1}}), \GS(\tau_{r_{m-1}}), \GS(\sigma_{r_m})$, where $(\GS(\sigma_{r_i}), \GS(\tau_{r_i})) \in \mathcal{W}_{\overline{A \cap B}}$. Now, from Observation 2\ref{obsd}, $\tau_0$, $\sigma_{r_m}$ are the unique critical simplices in $S_{\GS(\tau_0)}'$ and $S_{\GS(\sigma)}'$ respectively and for each $i \in [m]$. From Observation 2\ref{obse}, $(\sigma_{r_i}, \tau_{r_i}) \in \mathcal{W}'$ is the unique pair such that $\GS(\sigma_{r_i}, \GS(\tau_{r_i})) \in \mathcal{W}_{\overline{A \cap B}}$ for each $i \in [m]$. So we use \autoref{prop1} to deduce that $\tau_0 P \sigma_{r_1}= P_{(\tau_0, \sigma_{r_1})}$ and $\tau_{r_i} P \sigma_{r_{i+1}}=P_{(\tau_{r_i}, \sigma_{r_{i+1}})}$ for $1 \leq i \leq (m-1)$. Therefore it follows from the construction of $\eta$ that $\eta(P)=P$.\\

 \textbf{Case II:} $P$ is a trajectory of the fourth or fifth type.

 Without loss of generality let,
$$
\begin{aligned}
	& P: (\tau=)~\tau_0^{(q)}, \sigma_1^{(q-1)},\tau_1^{(q)} \dots  ,\sigma_{p}^{(q-1)}, \tau_{p}^{(q)} ,(\sigma_{p+1})_{\bar{A}}^{(q-1)}, (\tau_{p+1})_{\bar{A}}^{(q)}, \dots ,(\sigma_{p+l-1})_{\bar{A}}^{(q-1)},\\ & (\tau_{p+l-1})_{\bar{A}}^{(q)}, (\sigma_{p+l})_{\bar{A}}^{(q-1)}~(= \sigma),
\end{aligned}
$$
 where $\tau \in C^{int}_q$, $\sigma \in C^{\bar{A}}_{q-1}$.
 
Let $Q= \tau P (\sigma_{p+1})_{\bar{A}}$, $R = (\sigma_{p+1})_{\bar{A}} P \sigma$. Suppose, the ground simplex changes in $Q$ at $\sigma_{r_i}, \tau_{r_i}$, $i \in [m]$, i.e.,
\[
\begin{aligned}
	& \GS(\tau_0) = \GS(\sigma_1)= \GS(\tau_1) = \dots \GS(\tau_{r_1-1}) \ne \GS(\sigma_{r_1}) \ne \GS(\tau_{r_1}) = \dots =\GS(\tau_{r_2-1}) \\&\ne \GS(\sigma_{r_2}) \ne \GS(\tau_{r_2})= \dots = \GS(\tau_{r_{m}-1}) \ne \GS(\sigma_{r_{m}}) \ne \GS(\tau_{r_{m}}) = \dots = \GS((\sigma_{p+1})_{\bar{A}}).
\end{aligned}
\] 
Therefore, $\mu(P): \GS(\tau_0), \GS(\sigma_{r_1}), \GS(\tau_{r_1}), \dots, \GS(\sigma_{r_{m}}), \GS(\tau_{r_{m}}), R$. 

From Observation 2\ref{obse}, and Proposition~\ref{prop1}, we deduce that  $\tau_{r_i} P \sigma_{r_{i+1}}=P_{(\tau_{r_i}, \sigma_{r_{i+1}})}$ for $0 \leq i \leq (m-1)$. From Proposition~\ref{prop2}, we infer that $\tau_{r_m} P (\sigma_{p+1})_{\bar{A}}=P_{(\tau_{r_m}, (\sigma_{p+1})_{\bar{A}})}$. Let $R'= (\tau_{p+1})_{\bar{A}} P \sigma$. Then $\eta(\mu(P)): P_{(\tau_0, \sigma_{r_1})}, P_{(\tau_{r_1}, \sigma_{r_2})},\dots P_{(\tau_{r_{m-1}}, \sigma_{r_m})}, R'$, which we observe is same as $P$.\\

Now, let $P \in \MV(\beta, \alpha)$, where $\beta \in D_q(X)$, $\alpha \in D_{q-1}(X)$. From the construction of the maps, we observe that it suffices to prove that $\mu \circ \eta(P)=P$ for the following two cases.\\

 \textbf{Case I:} Let $\beta \in \Crit_{q-1}^{\mathcal{W}_{\overline{A \cap B}}}(\overline{A \cap B})$ and $\alpha \in \Crit_{q-2}^{\mathcal{W}_{\overline{A \cap B}}}(\overline{A \cap B})$.

 Let,
$$ P: (\beta=)~(\beta_0)_{\overline{A \cap B}}^{(q-1)}, (\alpha_1)_{\overline{A \cap B}}^{(q-2)}, (\beta_1)_{\overline{A \cap B}}^{(q-1)}, \dots , (\alpha_k)_{\overline{A \cap B}}^{(q-2)}, (\beta_k)_{\overline{A \cap B}}^{(q-1)}, (\alpha_{k+1})_{\overline{A \cap B}}^{(q-2)}~(=\alpha).$$

According to the construction of $\eta$, $\eta(P)= P_{(\tau_0, \sigma_1)}, P_{(\tau_1, \sigma_2)}, \dots , P_{(\tau_k, \sigma_{k+1})}$. Here $\tau_0$, $\sigma_{k+1}$ are the unique $\mathcal{W}$-critical simplices in $S_{\beta}'$ and $S_{\alpha}'$ respectively, $(\sigma_i, \tau_i) \in \mathcal{W}$ are the unique pair such that $\GS(\sigma_i)=(\alpha_i)_{\overline{A \cap B}}$, $\GS(\tau_i) = (\beta_i)_{\overline{A \cap B}}$ for all $i \in [k]$, and $P_{(\tau_i, \sigma_{i+1})}$, $0 \leq i \leq k$, are the trajectories obtained from \autoref{prop1} . Therefore, from the ground-simplex-preserving property of $P_{(\tau_i, \sigma_{i+1})}$ for each $i \in \{0, \dots ,k\}$, we note that the ground simplex in changes in $\eta(P)$ at precisely $\tau_0, \sigma_1, \tau_1, \dots , \sigma_k, \tau_k, \sigma_{k+1}$. Therefore $\mu(\eta(P)): \GS(\tau_0), \GS(\sigma_1), \GS(\tau_1), \dots, \GS(\sigma_k), \GS(\tau_k) , \GS(\sigma_{k+1})$, which is precisely $P$. Hence it follows that $\mu(\eta(P))$ is $P$.\\

 \textbf{Case II:} Let $\beta \in \Crit_{q-1}^{\mathcal{W}_{\overline{A \cap B}}}(\overline{A \cap B})$ and $\alpha \in \Crit_{q-1}^{\mathcal{W}_{\bar{A}}}(\bar{A})$, or $\beta \in \Crit_{q-1}^{\mathcal{W}_{\overline{A \cap B}}}(\overline{A \cap B})$ and $\alpha \in \Crit_{q-1}^{\mathcal{W}_{\bar{B}}}(\bar{B})$.

 Without loss of generality, let,
$$
\begin{aligned}
	 & P: (\beta=)~(\beta_0)_{\overline{A \cap B}}^{(q-1)}, (\alpha_1)_{\overline{A \cap B}}^{(q-2)}, (\beta_1)_{\overline{A \cap B}}^{(q-1)} \dots  (\alpha_p)_{\overline{A \cap B}}^{(q-2)}, (\beta_p)_{\overline{A \cap B}}^{(q-1)}, (\beta_p)_{\bar{A}}^{(q-1)}, (\gamma_p)_{\bar{A}}^{(q)}, \dots (\beta_{p+l-1})_{\bar{A}}^{(q-1)},\\ & (\gamma_{p+l-1})_{\bar{A}}^{(q)}, (\beta_{p+l})_{\bar{A}}^{(q-1)}~(=\alpha).
\end{aligned}
 $$

Let $Q= \beta P (\beta_p)_{\overline{A \cap B}}$, $R = (\gamma_p)_{\bar{A}} P \alpha$. Therefore $\eta(P)= P_{(\tau_0, \sigma_1)}, P_{(\tau_1, \sigma_2)}, \dots , P_{(\tau_p, \sigma_{p+1})}, R$, where $\tau$ is the unique $\mathcal{W}$-critical simplex in $S_{\beta}'$, and $(\sigma_i, \tau_i) \in \mathcal{W}$ are the unique pair such that $\GS(\sigma_i)=(\alpha_i)_{\overline{A \cap B}}$, $\GS(\tau_i) = (\beta_i)_{\overline{A \cap B}}$ for $i \in [p]$, and $\sigma_{p+1}=(\beta_p)_{\bar{A}}$. We recall that the ground simplex remains unchanged for $P_{(\tau_p, \sigma_{p+1})}$. So, again from the ground-simplex-preserving property of $P_{(\tau_i, \sigma_{i+1})}$ for each $i \in \{0, \dots ,p\}$, we note that the ground simplex changes at precisely $\tau_0, \sigma_1, \tau_1, \dots, \sigma_p, \tau_p$. Let $R': (\beta_p)_{\bar{A}}, R$. Therefore $\mu(\eta(P)): \GS(\tau_0), \GS(\sigma_1), \GS(\tau_1), \dots, \GS(\sigma_p), \GS(\tau_p) , R'$, which, we observe is precisely $P$. \\
This proves Lemma~\ref{bij}.

\end{proof}

\begin{lem}\label{weight}
  Let $P \in \MV(\beta, \alpha)$, where $\beta \in D_q(X), \alpha \in D_{q-1}(X)$, $q \geq 1$. Then the weight of $P$ is same as the weight of $\eta(P)$. In other words, $w_M(P)= w(\eta(P))$ for each $P \in \MV(\beta, \alpha)$. 
\end{lem}

\begin{proof}

 Let $P \in \MV(\beta, \alpha)$, where $\beta \in D_q(X)$, $\alpha \in D_{q-1}(X)$. From the construction of $\eta$,  it is clear that if $\beta \in \Crit_q^{\mathcal{W}_{\bar{A}}}(\bar{A})$, $\alpha \in \Crit_{q-1}^{\mathcal{W}_{\bar{A}}}(\bar{A})$ or $\beta \in \Crit_{q}^{\mathcal{W}_{\bar{B}}}(\bar{B})$, $\alpha \in \Crit_{q-1}^{\mathcal{W}_{\bar{B}}}(\bar{B}) $, then the result holds trivially. So it suffices to prove the result for the following three cases.\\

 \textbf{Case I:} Let, $\beta \in \Crit_{q-1}^{\mathcal{W}_{\overline{A \cap B}}}(\overline{A \cap B})$ and $\alpha \in \Crit_{q-2}^{\mathcal{W}_{\overline{A \cap B}}}(\overline{A \cap B}) $.

 Let $P: (\beta=)~(\beta_0)_{\overline{A \cap B}}^{(q-1)}, (\alpha_1)_{\overline{A \cap B}}^{(q-2)}, (\beta_1)_{\overline{A \cap B}}^{(q-1)}, \dots , (\alpha_k)_{\overline{A \cap B}}^{(q-2)}, (\beta_k)_{\overline{A \cap B}}^{(q-1)}, (\alpha_{k+1})_{\overline{A \cap B}}^{(q-2)}~(=\alpha).$

 We recall from Section~\ref{intro} that the weight of such a trajectory is given by,

		$$w_M(P)= - \left(\prod_{i=0}^{k-1}- \langle \beta_i, \alpha_{i+1} \rangle \langle \beta_{i+1}, \alpha_{i+1} \rangle\right)\langle \beta_k,\alpha_{k+1}\rangle$$\\
		$$ =- \langle \beta_0, \alpha_1 \rangle \prod_{i=1}^{k}(- \langle \beta_i, \alpha_i \rangle \langle \beta_i, \alpha_{i+1} \rangle).$$
		
 We denote $\langle \beta_0, \alpha_1 \rangle$ as $w_0$ and for each $i \in [k]$, we denote $- \langle \beta_i, \alpha_i \rangle \langle \beta_i, \alpha_{i+1} \rangle$ as $w_i$. Thus $w_M(P)= -\prod_{i=0}^{k}w_i$.

Let $g_{q}(\beta) = \tau_0$ and $g_{q-1}(\alpha)= \sigma_{k+1}$. Let $(\sigma_i, \tau_i) \in \mathcal{W'}$ be the unique pair such that $\GS(\sigma_i)=(\alpha_i)_{\overline{A \cap B}}$, $\GS(\tau_i)=(\beta_i)_{\overline{A \cap B}}$, for each $i \in [k]$. Then $\eta(P)$ is defined as 
$$P': P_{(\tau_0, \sigma_1)}, P_{(\tau_1, \sigma_2)}, \dots ,P_{(\tau_k, \sigma_{k+1})}.$$
where $P_{(\tau_{i-1}, \sigma_{i})}$, $i \in [k+1]$, carry the usual meaning as mentioned after \autoref{prop2}. We note here that the subsequences $\tau_0 P' \sigma_1$ of $P'$ and $\sigma_i P' \sigma_{i+1}$ are themselves $\mathcal{W}$-trajectories for each $i \in [k]$.

So, let $w(\tau_0 P' \sigma_1)=w_0'$ and for each $i \in [k]$, let $w(\sigma_i P' \sigma_{i+1})=w_i'$. Then $w(P')= \prod_{i=0}^{k}w_i'$.

Let $(\beta_0)_{\overline{A \cap B}}=[\bar{c}_{i_0}, \dots ,\bar{c}_{i_{q-1}}]$, $(\alpha_1)_{\overline{A \cap B}}= [\bar{c}_{i_0}, \dots , \widehat{\bar{c}_{i_m}}, \dots ,\bar{c}_{i_{q-1}}]$. Therefore $w_0= \langle (\beta_0)_{\overline{A \cap B}}, (\alpha_1)_{\overline{A \cap B}}\rangle= \langle \beta_0, \alpha_1\rangle= (-1)^m$.
Here, we note that $\langle \tau_{\overline{A \cap B}}, \sigma_{\overline{A \cap B}}\rangle= \langle \tau, \sigma \rangle$, $\langle \tau_{\bar{A}}, \sigma_{\bar{A}} \rangle= \langle \tau, \sigma \rangle$, $\langle \tau_{\bar{B}}, \sigma_{\bar{B}}\rangle= \langle \tau, \sigma \rangle$ for $\tau, \sigma \in X$. Henceforth, we use this observation throughout the proof without mentioning it again.

Now we determine $P_{(\tau_0, \sigma_1)}$ for the calculation of $w_0'$. Now, $\tau_0=g_q(\beta)=[\bar{a}_{i_0}, \bar{b}_{i_0}, \dots ,\bar{b}_{i_{q-1}}]$. So, we refer to Observation 3\ref{obs3a} to determine $P_{(\tau_0, \sigma_1)}$ .

If $m>0$, then $P_{(\tau_0, \sigma_1)}: \tau_0, \sigma_1$, where $\sigma_1=[\bar{a}_{i_0}, \bar{b}_{i_0}, \dots, \widehat{\bar{b}_{i_m}}, \dots ,\bar{b}_{i_{q-1}}]$. Then $w_0'= \langle \tau_0, \sigma_1\rangle= (-1)^{m+1} =-w_0$.

If $m=0$, then $P_{(\tau_0, \sigma_1)}: \tau_0, \widetilde{\sigma_1}, \widetilde{\tau_1}, \sigma_1$, where $\widetilde{\sigma_1}=[\bar{a}_{i_0}, \bar{b}_{i_1}, \dots , \bar{b}_{i_{q-1}}]$, $\widetilde{\tau_1}=[\bar{a}_{i_0}, \bar{a}_{i_1}, \bar{b}_{i_1}, \dots , \bar{b}_{i_{q-1}}]$, $\sigma=[\bar{a}_{i_1}, \bar{b}_{i_1}, \dots, \dots \bar{b}_{i_{q-1}}]$. Then, $w_0'= \langle \tau_0, \widetilde{\sigma_1} \rangle (- \langle \widetilde{\tau_1}, \widetilde{\sigma_1} \rangle\langle \widetilde{\tau_1}, \sigma_1 \rangle)= (-1)(-(-1)(-1)^0)=-1= -(-1)^{m}=-w_0$.

Now, let $(\beta_i)_{\overline{A \cap B}}=[\bar{c}_{j_0}, \dots , \bar{c}_{i_{q-1}}]$, $(\alpha_i)_{\overline{A \cap B}}=[\bar{c}_{j_0}, \dots \widehat{\bar{c}_{j_p}}, \dots, \bar{c}_{i_{q-1}}]$, $(\alpha_{i+1})_{\overline{A \cap B}}= [\bar{c}_{j_0}, \dots \widehat{\bar{c}_{j_m}}, \dots, \bar{c}_{i_{q-1}}]$, for each $i \in [k]$. Therefore $w_i= - \langle \beta_i, \alpha_i \rangle \langle \beta_i, \alpha_{i+1} \rangle=- (-1)^{p}(-1)^{m}$. 

 Now we compute $w_i'$ for each $i \in [k]$. We use Observation 3\ref{obs3a} and Observation 3\ref{obs3b} to determine $P_{(\tau_i, \sigma_{i+1})}$ for $i \in [k]$.

 Here, if $p> 0$, then from Observation 2\ref{obsa}, $\sigma_i= [\bar{a}_{j_0}, \bar{b}_{j_0}, \dots, \widehat{\bar{b}_{j_p}}, \dots ,\bar{b}_{j_{q-1}}]$, $\tau_i=[\bar{a}_{j_0}, \bar{b}_{j_0}, \dots ,\bar{b}_{j_{q-1}}]$.

Now, if $m=0$, then from Observation 3\ref{obs3a} $P_{(\tau_i, \sigma_{i+1})}: \tau_i, \widetilde{\sigma_i}, \widetilde{\tau_i}, \sigma_{i+1}$, where $\widetilde{\sigma_i}=[\bar{a}_{j_0}, \bar{b}_{j_1}, \dots , \bar{b}_{j_{q-1}}]$, $\widetilde{\tau_i}=[\bar{a}_{j_0}, \bar{a}_{j_1}, \bar{b}_{j_1}, \dots , \bar{b}_{j_{q-1}}]$, $\sigma_{i+1}=[\bar{a}_{j_1}, \bar{b}_{j_1}, \dots ,\bar{b}_{j_{q-1}}]$. Therefore $w_i'= w(\sigma_i P' \sigma_{i+1})=( - \langle \tau_i, \sigma_i\rangle \langle \tau_i, \widetilde{\sigma_i}\rangle)\\(-\langle \widetilde{\tau_i}, \widetilde{\sigma_i} \rangle \langle \widetilde{\tau_i}, \sigma_{i+1} \rangle)=(-(-1)^{p+1}(-1))(-(-1)(-1)^{0})=-(-1)^p(-1)^0=w_i$.

If $m >0$, then again from Observation 3\ref{obs3a}, $P_{(\tau_i, \sigma_{i+1})}: \tau_i, \sigma_{i+1}$, where $\sigma_{i+1}= [\bar{a}_{j_0}, \bar{b}_{j_0}, \dots, \widehat{\bar{b}_{j_m}}, \dots ,\bar{b}_{j_{q-1}}]$. Therefore $w_i'=w(\sigma_i P' \sigma_{i+1})= - \langle \tau_i, \sigma_i \rangle \langle \tau_i, \sigma_{i+1} \rangle =-(-1)^{p+1}(-1)^{m+1}=w_i$.

Next, if $p=0$, then from Observation 2\ref{obsa}, $\sigma_i= [\bar{a}_{j_1},\bar{b}_{j_1}, \dots ,\bar{b}_{j_{q-1}}]$, $\tau_i= [\bar{a}_{j_0},\bar{a}_{j_1}, \bar{b}_{j_1}, \dots , \bar{b}_{j_{q-1}}]$. Therefore, from Observation 3\ref{obs3b}, $m > 0$, and $P_{(\tau_i, \sigma_{i+1})}: \tau_i, \widetilde{\sigma_i}, \widetilde{\tau_i}, \sigma_{i+1}$, where, $\widetilde{\sigma_{i}}= [\bar{a}_{j_0}, \bar{b}_{j_1}, \dots ,\bar{b}_{j_{q-1}}]$, $\widetilde{\tau_{i}}=[\bar{a}_{j_0}, \bar{b}_{j_0}, \bar{b}_{j_1}, \dots ,\bar{b}_{j_{q-1}}]$, $\sigma_{i+1}=[\bar{a}_{j_0}, \bar{b}_{j_0}, \dots , \widehat{\bar{b}_{j_m}}, \dots ,\bar{b}_{j_{q-1}}]$.  So, $w_i'=( - \langle \tau_i, \sigma_i\rangle \langle \tau_i, \widetilde{\sigma_i}\rangle)(-\langle \widetilde{\tau_i}, \widetilde{\sigma_i} \rangle \langle \widetilde{\tau_i}, \sigma_{i+1} \rangle)= (-(-1)^{0}(-1))(-(-1)(-1)^{m+1})=-(-1)^0(-1)^m=w_i'$.

 Therefore, $w_M(P)= -\prod_{i=0}^k w_i = \prod_{i=0}^k w_i' = w(\eta(P))$.


 \textbf{Case II:} Let, $\beta\in \Crit_{q-1}^{\mathcal{W}_{\overline{A \cap B}}}(\overline{A \cap B})$ and $\alpha \in \Crit_{q-1}^{\mathcal{W}_{\bar{A}}}(\bar{A})$.\\

 Let $P \in \MV(\beta, \alpha)$ be as follows.
$$  
\begin{aligned}
	& (\beta=)~(\beta_0)_{\overline{A \cap B}}^{(q-1)}, (\alpha_1)_{\overline{A \cap B}}^{(q-2)}, (\beta_1)_{\overline{A \cap B}}^{(q-1)} \dots  (\alpha_p)_{\overline{A \cap B}}^{(q-2)}, (\beta_p)_{\overline{A \cap B}}^{(q-1)}, (\beta_p)_{\bar{A}}^{(q-1)}, (\gamma_p)_{\bar{A}}^{(q)}, \dots (\beta_{p+l-1})_{\bar{A}}^{(q-1)}, \\ &(\gamma_{p+l-1})_{\bar{A}}^{(q)}, (\beta_{p+l})_{\bar{A}}^{(q-1)}~(=\alpha).
\end{aligned}
$$

 We recall from Section~\ref{intro} that the weight of $P$ is given by, 
\begin{equation*}
	\begin{aligned}
		w_M(P) &= - \left(\prod_{i=0}^{p-1} -\langle \beta_i,\alpha_{i+1},  \rangle \langle \beta_{i+1}, \alpha_{i+1} \rangle \right) \left(\prod_{i=p}^{p+l-1}-\langle \gamma_i, \beta_i\rangle \langle \gamma_i, \beta_{i+1}\rangle \right)\\ &= - \langle \beta_0, \alpha_1 \rangle \left(\prod_{i=1}^{p-1} -\langle \beta_i,\alpha_i,  \rangle \langle \beta_{i}, \alpha_{i+1} \rangle \right) \langle \beta_p, \alpha_p \rangle \langle \gamma_p, \beta_p \rangle \left(\langle \gamma_p, \beta_{p+1} \rangle \prod_{i=p+1}^{p+l-1}-\langle \gamma_i, \beta_i\rangle \langle \gamma_i, \beta_{i+1}\rangle \right)\\ &= \left(- \prod_{i=0}^{p-1}w_i \right) \langle \beta_p, \alpha_p \rangle \langle \gamma_p, \beta_p \rangle w(R),
	\end{aligned}
\end{equation*}
where $w_0=\langle \beta_0, \alpha_1 \rangle$ and $w_i=-\langle \beta_i,\alpha_i,  \rangle \langle \beta_{i}, \alpha_{i+1} \rangle$, as introduced in Case I, and $R= (\gamma_p)_{\bar{A}} P \alpha$ is a $\mathcal{W}_{\bar{A}}$-trajectory.

Let $g_q(\beta)= \tau$ and let $(\sigma_i, \tau_i) \in \mathcal{W'}$ be the unique pair such that $\GS(\sigma_i)=(\alpha_i)_{\overline{A \cap B}}$, $\GS(\tau_i)= (\beta_i)_{\overline{A \cap B}}$. Let $\sigma_{p+1}= (\beta_p)_{\bar{A}}$, for each $i \in [p]$. Then $\eta(P)$ is given by, 
$$ P': P_{(\tau_0, \sigma_1)}, P_{(\tau_1, \sigma_2)}, \dots ,P_{(\tau_{p-1}, \sigma_p)}, P_{(\tau_p, \sigma_{p+1})}, R ,$$ where, for each $i \in [p]$, $P_{(\tau_i, \sigma_{i+1})}$ have their usual meanings as discussed after \autoref{prop2}. We note here that $\tau_0 P' \sigma_p$ and $\tau_p P' (\gamma_p)_{\bar{A}} = P_{(\tau_p, \sigma_{p+1})} $ are $\mathcal{W}$-trajectories. So, $w(P')= w(\tau_0 P' \sigma_p)(- \langle \tau_p, \sigma_p \rangle) \\ w(\tau_p P' (\gamma_p)_{\bar{A}})w(R)$.

Now, we use analogous arguments and computations as in Case I, to show that $w(\tau_0 P' \sigma_p) = - \prod_{i=0}^{p-1}w_i$. Since the weight of $R$ remains unchanged, it suffices to show that $- \langle \tau_p, \sigma_p \rangle w(\tau_p P' (\gamma_p)_{\bar{A}}) = \langle \beta_p, \alpha_p \rangle \langle \gamma_p, \beta_p \rangle$, . 

Let $(\beta_p)_{\overline{A \cap B}}=[\bar{c}_{i_0}, \dots , \bar{c}_{i_{q-1}}]$, $(\alpha_p)_{\overline{A \cap B}}=[\bar{c}_{i_0}, \dots , \widehat{\bar{c}_{i_r}}, \dots , \bar{c}_{i_{q-1}}]$, and \\
 $(\gamma_p)_{\bar{A}}= (\beta_p)_{\bar{A}} \cup \{\bar{a}_{i_m}\}= [\bar{a}_{i_0}, \dots,\bar{a}_{i_m}, \dots , \bar{a}_{i_{q-1}}]$, $i_0 < i_1 < \dots < i_m < \dots < i_{q-1}$. So, $\langle \beta_p, \alpha_p \rangle \langle \gamma_p, \beta_p \rangle= (-1)^r(-1)^m$.
We use Observation 4\ref{obs4a} and 4\ref{obs4b} to determine $P_{(\tau_p, \sigma_{p+1})}$. 

If $r > 0$, then from Observation 2\ref{obsa}, $\tau_p= [\bar{a}_{i_0}, \bar{b}_{i_0}, \dots , \bar{b}_{i_{q-1}}]$, $\sigma_p= [\bar{a}_{i_0}, \bar{b}_{i_0}, \dots ,\widehat{\bar{b}_{i_r}},\dots  ,\bar{b}_{i_{q-1}}]$. So, from Observation 4\ref{obs4a}, $P_{(\tau_p, \sigma_{p+1})}: \tau_p, \widetilde{\sigma_p}, \widetilde{\tau_{p}}, \dots , \widetilde{\sigma_{p+q-2}}, \widetilde{\tau_{p+q-2}}, \sigma_{p+1}$, where, $\widetilde{\sigma_p}=[\bar{a}_{i_0}, \bar{b}_{i_1}, \dots , \bar{b}_{i_{q-1}}]$, $\widetilde{\tau_p}=[\bar{a}_{i_0}, \bar{a}_{i_1}, \bar{b}_{i_1}, \dots , \bar{b}_{i_{q-1}}], \dots , \widetilde{\sigma_{p+q-2}}=[\bar{a}_{i_0}, \bar{a}_{i_1},  \dots , \bar{a}_{i_{q-2}}, \bar{b}_{i_{q-1}}]$, $\widetilde{\tau_{p+q-2}}= [\bar{a}_{i_0}, \bar{a}_{i_1},  \dots , \bar{a}_{i_{q-2}},\bar{a}_{i_{q-1}}, \bar{b}_{i_{q-1}}]$, $\sigma_{p+1}=[\bar{a}_{i_0}, \bar{a}_{i_1},  \dots , \bar{a}_{i_{q-2}},\bar{a}_{i_{q-1}}]$. Therefore, $\tau_p P' (\gamma_p)_{\bar{A}}= P_{(\tau_p, \sigma_{p+1})}, (\gamma_p)_{\bar{A}}$. Thus,

\begin{equation*}
	\begin{aligned}
		-\langle \tau_p, \sigma_p \rangle w(\tau_p P' (\gamma_p)_{\bar{A}})& = - (-1)^{r+1}\left(\prod_{i=1}^{q-1}-(-1)^i(-1^i)\right)(-1)^q(-(-1)^{m}) \\ &=-(-1)^{r+1} (-1) (-(-1)^{m})\\ & =\langle \beta_p, \alpha_p \rangle \langle \gamma_p, \beta_p \rangle.
	\end{aligned}
\end{equation*}

If $r=0$, then again, from Observation 2\ref{obsa}, $\sigma_p= [\bar{a}_{i_1}, \bar{b}_{i_1}, \dots ,\widehat{\bar{b}_{i_r}},\dots  ,\bar{b}_{i_{q-1}}]$,$\tau_p= [\bar{a}_{i_0}, \bar{a}_{i_1}, \bar{b}_{i_1}, \dots , \bar{b}_{i_{q-1}}]$. Therefore,from Observation 4\ref{obs4a}, $P_{(\tau_p, \sigma_{p+1})}: \tau_p, \widetilde{\sigma_p} \widetilde{\tau_{p}}, \dots , \widetilde{\sigma_{p+q-3}}, \widetilde{\tau_{p+q-3}}, \sigma_{p+1}$, where, $\widetilde{\sigma_p}=[\bar{a}_{i_0}, \bar{a}_{i_1}, \bar{b}_{i_2} \dots , \bar{b}_{i_{q-1}}]$, $\widetilde{\tau_p}=[\bar{a}_{i_0}, \bar{a}_{i_1}, \bar{a}_{i_2}, \bar{b}_{i_2}, \dots , \bar{b}_{i_{q-1}}], \dots , \widetilde{\sigma_{p+q-3}}=[\bar{a}_{i_0}, \bar{a}_{i_1},  \dots , \bar{a}_{i_{q-2}}, \bar{b}_{i_{q-1}}]$, $\widetilde{\tau_{p+q-3}}= [\bar{a}_{i_0}, \bar{a}_{i_1},  \dots , \bar{a}_{i_{q-2}},\bar{a}_{i_{q-1}} \bar{b}_{i_{q-1}}]$, $\sigma_{p+1}=[\bar{a}_{i_0}, \bar{a}_{i_1},  \dots , \bar{a}_{i_{q-2}},\bar{a}_{i_{q-1}}]$.
Now, $\tau_p P' (\gamma_p)_{\bar{A}}= P_{(\tau_p, \sigma_{p+1})}, (\gamma_p)_{\bar{A}}$. 
 Therefore, 

\begin{equation*}
	\begin{aligned}
		-\langle \tau_p, \sigma_p \rangle w(\tau_p P' (\gamma_p)_{\bar{A}})& = - (-1)^{0}\left(\prod_{i=2}^{q-1}(-1)^i(-1^i)\right)(-1)^q(-(-1)^{m}) \\ &=-(-1)^{0} (1) (-(-1)^{m})\\ & =\langle \beta_p, \alpha_p \rangle \langle \gamma_p, \beta_p \rangle.	\end{aligned}
\end{equation*}

 Thus, in this case also, we have proved that $w_M(P)=w(\eta(P))$.\\

 \textbf{Case III:} Let $\beta \in \Crit_{q-1}^{\mathcal{W}_{\overline{A \cap B}}}(\overline{A \cap B})$ and $\alpha \in \Crit_{q-1}^{\mathcal{W}_{\bar{B}}}(\bar{B})$.

 Let $P \in \MV(\beta, \alpha)$, where,
$$  
\begin{aligned}
	& (\beta=)~(\beta_0)_{\overline{A \cap B}}^{(q-1)}, (\alpha_1)_{\overline{A \cap B}}^{(q-2)}, (\beta_1)_{\overline{A \cap B}}^{(q-1)} \dots  (\alpha_p)_{\overline{A \cap B}}^{(q-2)}, (\beta_p)_{\overline{A \cap B}}^{(q-1)}, (\beta_p)_{\bar{B}}^{(q-1)}, (\gamma_p)_{\bar{B}}^{(q)}, \dots (\beta_{p+l-1})_{\bar{B}}^{(q-1)}, \\ &(\gamma_{p+l-1})_{\bar{B}}^{(q)}, (\beta_{p+l})_{\bar{B}}^{(q-1)}~(=\alpha).
\end{aligned}
$$

Now, we recall from Section \ref{intro} that,
\begin{equation*}
	\begin{aligned}
		w_M(P) &= \left(\prod_{i=0}^{p-1} -\langle \beta_i,\alpha_{i+1},  \rangle \langle \beta_{i+1}, \alpha_{i+1} \rangle \right) \left(\prod_{i=p}^{p+l-1}-\langle \gamma_i, \beta_i\rangle \langle \gamma_i, \beta_{i+1}\rangle \right)\\ &= \left( \prod_{i=0}^{p-1}w_i \right) \langle \beta_p, \alpha_p \rangle \langle \gamma_p, \beta_p \rangle w(R), \text{ (by similar computations as in Case I)}
	\end{aligned}
\end{equation*}
 where $w_i$ carry usual meanings as in the previous cases and $R= (\gamma_p)_{\bar{B}} P \alpha$ is a $\mathcal{W}_{\bar{B}}$-trajectory. 

As before, let $g_q(\beta)= \tau_0$ and let $(\sigma_i, \tau_i) \in \mathcal{W'}$ be the unique pair such that $\GS(\sigma_i)=(\alpha_i)_{\overline{A \cap B}}$, $\GS(\tau_i)= (\beta_i)_{\overline{A \cap B}}$. Let $\sigma_{p+1}= (\beta_p)_{\bar{B}}$, for each $i \in [p]$. Then $\eta(P)$ is given by, 
$$ P': P_{(\tau_0, \sigma_1)}, P_{(\tau_1, \sigma_2)}, \dots ,P_{(\tau_{p-1}, \sigma_p)}, P_{(\tau_p, \sigma_{p+1})}, R ,$$ where $R= (\gamma_p)_{\bar{B}} P \alpha$ and for each $i \in [p]$, $P_{(\tau_i, \sigma_{i+1})}$ have their usual meanings as discussed in the construction of $\eta$. We note here that $\tau_0 P' \sigma_p$ and $\tau_p P' (\gamma_p)_{\bar{B}}$ are $\mathcal{W}$-trajectories. Therefore $w(P')= w(\tau_0 P' \sigma_p)(- \langle \tau_p, \sigma_p \rangle) w(\tau_p P' (\gamma_p)_{\bar{B}})w(R)$.

Using similar arguments and computations as in Case I, we can show that $w(\tau_0 P' \sigma_p)= - \prod_{i=0}^{p-1}w_i$. So, it suffices to prove that  $- \langle \tau_p, \sigma_p \rangle w(\tau_p P' (\gamma_p)_{\bar{B}}) = - \langle \beta_p, \alpha_p \rangle \langle \gamma_p, \beta_p \rangle$.

Let $(\beta_p)_{\overline{A \cap B}}=[\bar{c}_{i_0}, \dots , \bar{c}_{i_{q-1}}]$, $(\alpha_p)_{\overline{A \cap B}}=[\bar{c}_{i_0}, \dots , \widehat{\bar{c}_{i_r}},\dots, \bar{c}_{i_{q-1}}]$, $(\gamma_p)_{\bar{B}}= (\beta_p)_{\bar{B}} \cup \{\bar{b}_{i_m}\} \\= [\bar{b}_{i_0}, \dots,\bar{b}_{i_m}, \dots , \bar{b}_{i_{q-1}}]$, $i_0 < i_1 < \dots < i_m < \dots < i_{q-1}$. So, $\langle \beta_p, \alpha_p \rangle \langle \gamma_p, \beta_p \rangle= (-1)^r(-1)^m$.

Now, we use Observation 4\ref{obs4b} to determine $P_{(\tau_p, \sigma_{p+1})}$. 

If $r > 0$, then  from Observation 2\ref{obsa}, $\tau_p= [\bar{a}_{i_0}, \bar{b}_{i_0}, \dots , \bar{b}_{i_{q-1}}]$, $\sigma_p= [\bar{a}_{i_0}, \bar{b}_{i_0}, \dots ,\widehat{\bar{b}_{i_r}},\dots  ,\bar{b}_{i_{q-1}}]$. So, from Observation 4\ref{obs4b}, $P_{(\tau_p, \sigma_{p+1})}: \tau_p, \sigma_{p+1}$.
So,  
\begin{equation*}
	\begin{aligned}
		- \langle \tau_p, \sigma_p \rangle w(\tau_p P' (\gamma_p)_{\bar{B}}) &=- \langle \tau_p, \sigma_p \rangle (-\langle \tau_p, \sigma_{p+1} \rangle \langle (\gamma_p)_{\bar{B}}, \sigma_{p+1}\rangle) \\ &=-(-1)^{r+1}(-(-1)^0(-1)^m)\\ &=  -\langle \beta_p, \alpha_p \rangle \langle \gamma_p, \beta_p \rangle.
	\end{aligned}
\end{equation*}

If $r=0$, then from Observation 2\ref{obsa}, $\tau_p= [\bar{a}_{i_0}, \bar{a}_{i_1},\bar{b}_{i_1}, \dots , \bar{b}_{i_{q-1}}]$, $\sigma_p= [\bar{a}_{i_1}, \bar{b}_{i_1}, \dots ,\widehat{\bar{b}_{i_r}},\dots  ,\bar{b}_{i_{q-1}}]$. From Observation 4\ref{obs4b}, $P_{(\tau_p, \sigma_{p+1})}: \tau_p, \widetilde{\sigma_p}, \widetilde{\tau_p}, \sigma_{p+1}$, where $\widetilde{\sigma_p}=[\bar{a}_{i_0},\bar{b}_{i_1}, \dots , \bar{b}_{i_{q-1}}]$, $\widetilde{\tau_p}=[\bar{a}_{i_0}, \bar{b}_{i_0},\bar{b}_{i_1}, \dots , \bar{b}_{i_{q-1}}]$, $\sigma_{p+1}=[\bar{b}_{i_0}, \dots ,\bar{b}_{i_{q-1}}]$. So,

\begin{equation*}
	\begin{aligned}
		- \langle \tau_p, \sigma_p \rangle w(\tau_p P' (\gamma_p)_{\bar{B}}) &=- \langle \tau_p, \sigma_p \rangle (-\langle \tau_p, \widetilde{\sigma_{p}} \rangle \langle \widetilde{\tau_p}, \widetilde{\sigma_p} \rangle)(-\langle \widetilde{\tau_p}, \sigma_{p+1} \rangle \langle (\gamma_p)_{\bar{B}}, \sigma_{p+1}\rangle) \\ &=-(-1)^0(-(-1)^1(-1)^1)(-(-1)^{0}(-1)^m)\\ &=  -\langle \beta_p, \alpha_p \rangle \langle \gamma_p, \beta_p \rangle.
	\end{aligned}
\end{equation*}

 Hence we have proved that $w_M(P)=w(\eta(P))$ for each $P \in \MV(\beta, \alpha)$.

\end{proof}

 Now we are ready to prove the next result.

\begin{thm}\label{new}
	The chain complex $\mathcal{D}_{\#}(X)$ is isomorphic to the Thom-Smale Complex of $\widetilde{X}$, $C_{\#}^{\mathcal{W}}(\widetilde{X}, \mathbb{Z})$, i.e.,
	
	$$ \mathcal{D}_{\#}(X) \cong C_{\#}^{\mathcal{W}}(\widetilde{X}, \mathbb{Z}).$$
	 Hence, $H_{\#}^{\mathcal{W}}(\widetilde{X}, \mathbb{Z}) \cong H_{\#}^{\mathcal{D}}(X)$.
\end{thm}
\begin{proof}
		
	 We have already proved earlier that an isomorphism $f_{\#}$ exists between $C_{\#}^{\mathcal{W}}(\widetilde{X}, \mathbb{Z})$ and $\mathcal{D}_{\#}(X)$. Now all we need to prove is that $f_{\#}$ is a chain map, i.e., the following diagram commutes.
	
	\[\begin{tikzcd}
		\cdots & {C_{q}^{\mathcal{W}}(\widetilde{X}, \mathbb{Z})} && {C_{q-1}^{\mathcal{W}}(\widetilde{X}, \mathbb{Z})} & \cdots \\
		\\
		\cdots & {\mathcal{D}_{q}(X)} && {\mathcal{D}_{q-1}(X)} & \cdots
		\arrow[from=1-1, to=1-2]
		\arrow["{\partial_q^{\mathcal{W}}}", from=1-2, to=1-4]
		\arrow["{f_q}"', from=1-2, to=3-2]
		\arrow[from=1-4, to=1-5]
		\arrow["{f_{q-1}}"', from=1-4, to=3-4]
		\arrow[from=3-1, to=3-2]
		\arrow["{\partial^{\mathcal{D}}_{q-1}}"', from=3-2, to=3-4]
		\arrow[from=3-4, to=3-5]
		\arrow[draw=none, from=1-2, to=3-4, "{\text{\Huge$\circlearrowleft$}}" description]
	\end{tikzcd}\]

	We will show that $\partial_{q-1}^{\mathcal{D}} \circ f_q = f_{q-1} \circ  \partial_q^{\mathcal{W}} $, for each $q \geq 0$. It suffices to show that this holds for the generators. 
	
	 Let $\tau \in \Crit_q^{\mathcal{W}}(\widetilde{X})$ and let $f_q(\tau)=\beta$. Now,
	\[
		\begin{aligned}
			\partial_q^{\mathcal{D}}(f_q(\tau)) &= \partial_q^{\mathcal{D}}(\beta)\\
			&=\sum_{\alpha \in D_{q-1}(X)}\left(\sum_{P \in \MV(\beta, \alpha)}w_M(P)\right)\alpha\\
			&=\sum_{\alpha \in D_{q-1}(X)}\left(\sum_{P \in \MV(\beta, \alpha)}w_M(P)\right)f_{q-1}(g_{q-1}(\alpha))\\
			&= f_{q-1}\left(\sum_{\alpha \in D_{q-1}(X)}\left(\sum_{P \in \MV(\beta, \alpha)}w_M(P)\right)g_{q-1}(\alpha)\right).
		\end{aligned}
	\]
	
	 Let $g_{q-1}(\alpha)=\sigma$ and let $\eta(P)=P'$, for each $P \in \MV(\beta, \alpha)$. Now, we know from \autoref{bij}  that $\eta$ is a bijection between $\MV(\beta, \alpha)$ and $\Gamma(\tau, \sigma) $. Also, from \autoref{weight}, $w_M(P)=w(\eta(P))$. Thus, we can write,
	\[
		\begin{aligned}
			\partial_q^{\mathcal{D}}(f_q(\beta)) &= f_{q-1}\left(\sum_{\sigma \in \Crit_{q-1}^{\mathcal{W}}(\widetilde{X})}\left(\sum_{\eta(P) \in \Gamma(\tau, \sigma)}w(\eta(P)\right)\sigma \right) \text{, since } g_{q-1} \text{ is an isomorphism} .\\
			&= f_{q-1}\left( \sum_{\sigma \in \Crit_{q-1}^{\mathcal{W}}(\widetilde{X})}\left(\sum_{P' \in \Gamma(\tau, \sigma)}w(P')\right)\sigma\right) \text{, where $\eta(P)=P'$. } \\ &= f_{q-1}(\partial_q^{\mathcal{W}}(\tau)).
		\end{aligned}	
	\]
	
	 Hence the result follows.
	
\end{proof}

\begin{proof}[Proof of \autoref{main}]
	We know from Theorem~\ref{homeq} that, $H_{\#}(X) \cong H_{\#}(\widetilde{X}, \mathbb{Z})$. Therefore, from \autoref{hom} and \autoref{new}, it follows that $H_{\#}(\widetilde{X}, \mathbb{Z}) \cong H_{\#}^{\mathcal{W}}(\widetilde{X}, \mathbb{Z}) \cong H_{\#}^{\mathcal{D}}(X)$. Hence $H_{\#}(X) \cong H^{\mathcal{D}}_{\#}(X)$.
\end{proof}

		
		\end{section}
		
		\bibliographystyle{plain}
		\bibliography{ref.bib}

\end{document}